\documentclass[12pt,centertags]{amsart}
\usepackage{amsfonts}
\usepackage{}
\usepackage{amsmath,amstext,amsthm,a4,amssymb,amscd}
\usepackage[mathscr]{eucal}
\usepackage{mathrsfs}
\usepackage{epsf}
\usepackage{tikz}

\usetikzlibrary{matrix,arrows,decorations.pathmorphing}
\usepackage[a4paper,top=3cm,bottom=3cm,left=2cm,right=2cm]
{geometry}

\numberwithin{equation}{section}
\newcommand{\field}[1]{\mathbb{#1}}
\newcommand{\Z}{\field{Z}}
\newcommand{\R}{\field{R}}
\newcommand{\C}{\field{C}}
\newcommand{\N}{\field{N}}

\def\cC{\mathscr{C}}

\def\cE{\mathscr{E}}

\def\mA{\mathcal{A}}
\def\mB{\mathcal{B}}

\def\mE{\mathcal{E}}
\def\mF{\mathcal{F}}
\def\mG{\mathcal{G}}
\def\mH{\mathcal{H}}

\def\mT{\mathcal{T}}
\def\mU{\mathcal{U}}

\newcommand\mS{\mathcal{S}}

\def\Im{{\rm Im}}

\def\la{\langle}
\def\ra{\rangle}

\DeclareMathOperator{\End}{End}

\DeclareMathOperator{\Ker}{Ker}

\DeclareMathOperator{\Id}{Id}

\DeclareMathOperator{\tr}{Tr}
\DeclareMathOperator{\ind}{Ind}

\DeclareMathOperator{\td}{Td}

\DeclareMathOperator{\ch}{ch}

\newtheorem{thm}{Theorem}[section]
\newtheorem{lemma}[thm]{Lemma}
\newtheorem{prop}[thm]{Proposition}

\theoremstyle{definition}
\newtheorem{rem}[thm]{Remark}
\theoremstyle{definition}
\newtheorem{defn}[thm]{Definition}

\newtheorem{example}[thm]{Example}

\newcommand{\be}{\begin{eqnarray}}
\newcommand{\ee}{\end{eqnarray}}
\newcommand{\ov}{\overline}
\newcommand{\wi}{\widetilde}
\newcommand{\var}{\varepsilon}
\numberwithin{equation}{section}

\numberwithin{thm}{section}

\newcommand{\comment}[1]{}

\begin{document}
	
	\title{Equivariant eta forms and equivariant differential K-theory}
	
	\author{Bo Liu}
	
	\address{School of Mathematical Sciences,
		Shanghai  Key Laboratory of PMMP,
		East China Normal University, 
		Shanghai, 200241, 
		P.R. China}
	\email{bliu@math.ecnu.edu.cn}
	
	
	\begin{abstract}
		In this paper, for a compact Lie group action,
		 we prove the anomaly formula and the functoriality
		of the equivariant Bismut-Cheeger eta forms with 
		perturbation operators when the equivariant
		family index vanishes. In order to prove them, we extend the Melrose-Piazza spectral section and its main properties to the equivariant case
		and introduce the equivariant version of the Dai-Zhang higher 
		spectral flow for arbitrary dimensional fibers.
		Using these results, we construct a new analytic model of
		the equivariant differential K-theory for compact manifolds
		when the group action has finite stabilizers only, 
		which modifies the Bunck-Schick
		model of the differential K-theory. This model could also be regarded
		as an analytic model of the differential K-theory for
		compact orbifolds. Especially, we answer a question proposed
		by Bunke and Schick about the well-definedness of 
		the push-forward map.

%
	\end{abstract}
	\maketitle

	
	
	\tableofcontents

	\setcounter{section}{-1}
	
	\section{Introduction} \label{s00}
			

	The differential K-theory is the differential extension of
    the topological K-theory, whose basic idea is to combine the 
    topological K-theory with the
	differential form information.
	It is
	 partly motivated by the study of D-branes in theoretical physics
	 (see e.g., \cite{FH00,Wi98}). 
	 Various models of differential K-theory have been given:
	Hopkins-Singer \cite{HS05}, Bunke-Schick \cite{BS09}, Freed-Lott \cite{FL10},  Simons-Sullivan \cite{SS10}, Tradler-Wilson-Zeinalian \cite{TWZ13}, 
Gorokhovsky-Lott \cite{GL18}, etc. For a detailed survey, see \cite{BS12}.
	
	Until now, the equivariant version of the differential K-theory is not well understood yet.
	When the group is finite, the equivariant differential K-theory was studied by Szabo-Valentino \cite{SV10} and Ortiz \cite{Or09}.
	In \cite{BS13}, Bunke and Schick extended their model to the orbifold case using the language of stacks, which could be regarded as a model of the equivariant differential K-theory when the action has finite stabilizers only.
		
	Inspired by the model of Bunke-Schick \cite{BS13},
	as a parallel version, in this paper,
	 we will construct a purely analytic model of the equivariant differential K-theory for compact manifolds when the action has finite stabilizers only, using the local index technique developed in \cite{BL92}. Moreover, a detailed proof of the well-definedness of
	the push-forward map is given here, which is a question proposed in \cite{BS09,BS13} and is the main motivation for this new construction. This model is a direct generalization of \cite{BS09} without using the language of stacks and could also be regarded as an analytic model of the differential K-theory for compact orbifolds.
		
	The study of the differential K-theory is always related to the
	Bismut-Cheeger eta form \cite{BC89}, which is defined for a family 
	of Dirac operators and is the family 
	extension of the 
	famous Atiyah-Patodi-Singer eta invariant. 
	Usually, the well-definedness of the eta form needs one of the following 
	additional conditions: 
	\begin{enumerate}
	\item the kernels of the 
	family of Dirac operators form a vector bundle over the base 
	manifold \cite{BC89,D91};
	\item the family index of the family of Dirac operators vanishes
	as an element of the K-group of the base manifold \cite{Bu09,MP97a,MP97b}.
	\end{enumerate} 

   In the model of Freed-Lott \cite{FL10}, the eta form with the first condition
    is used to define the push-forward map. 
    In
    the model of Bunke-Schick \cite{BS09}, the eta form 
    under the second condition, which is defined by Bunke \cite{Bu09}
    using the taming, is used to define the differential K-group.
%
    From this point of view,
    in order to extend the differential K-theory to the equivariant case,
    we firstly need to extend the Bismut-Cheeger eta form
    to the equivariant case.
    
    In \cite{Liu17}, the author systematically studied the equivariant
    eta form under the first condition and prove the anomaly formula
    and the functoriality of them, which should be used to establish 
    an equivariant version of the Freed-Lott model. In this paper,
    we will establish the properties of the eta form 
    under the second condition, extend them
    to the equivariant case and use them to construct our model.
    In order to do finer spectral analysis, we use the notion
    of the spectral section developed by Melrose and Piazza in
    \cite{MP97a,MP97b} and the 
    Dai-Zhang higher spectral flow \cite{DZ98} instead of the taming
    and the Kasporov KK-theory in \cite{Bu09,BS09}.
    
    
%
	
    In \cite{MP97a,MP97b}, in order to prove the family index theorem for manifolds with boundary, Melrose and Piazza defined the spectral section 
    and the eta form under the second condition.
	In \cite{DZ98}, using the spectral section, Dai and Zhang introduced the higher spectral flow for a family of Dirac type operators
	on a family of odd dimensional manifolds. 	
	In this paper, we will extend the spectral section, the higher spectral flow and the eta form to the equivariant case. 
	Especially, we will define the higher spectral flow for a family of even dimensional manifolds. Furthermore, we will prove the anomaly formula and the functoriality of equivariant eta forms using the language of equivariant higher spectral flow, which is an analogue of the results in \cite{Be09,BM04,Liu17} and using the techniques in \cite{D91,Ma99,Ma00,Ma02}. Note that our proof of the anomaly formula of the eta forms for a family of even dimensional manifolds relies on the
	 funtoriality of equivariant eta forms Theorem \ref{d188}, which is highly nontrivial and is the main technical difficulty of this paper. 
	 Since the second condition is a topological condition,
	 there is no additional rigidity assumption in the formulas here.

	Let $\pi: W\rightarrow B$ be a proper smooth submersion of  compact manifolds
	with orientable fibers $Z$. Let $TZ=\ker(d\pi)$ be the relative tangent bundle
	to the fibers $Z$ with Riemannian metric $g^{TZ}$ and $T^HW$ be a horizontal subbundle of $TW$, such that $TW=T^HW\oplus TZ$. Let $o$ be an orientation of $TZ$. 
	Let $\nabla^{TZ}$ be the Euclidean connection on $TZ$ defined in (\ref{e01014}).
	We assume that $TZ$ has a Spin$^c$ structure. Let $L_Z$ be the complex line bundle associated with the Spin$^c$ structure of $TZ$ with a Hermitian metric $h^{L_Z}$ and a Hermitian connection $\nabla^{L_Z}$.
	Let $(E,h^E)$ be a $\Z_2$-graded Hermitian vector bundle with a Hermitian connection $\nabla^E$.
	Let $G$ be a compact Lie group which acts on $W$ and $B$ such that $\pi\circ g=g\circ \pi$ for any $g\in G$. We assume that the $G$-action preserves everything.
	 The family of $G$-equivariant geometric data $\mF=(W, L_Z, E, o, T_{}^HW, g^{TZ}, h^{L_Z}, \nabla^{L_Z}, h^E, \nabla^{E})$ is enough to define
	 the equivariant Bismut superconnection. We call $\mF$ an equivariant geometric family over $B$ for short. Let $D(\mF)$ be the fiberwise Dirac operators of $\mF$ defined in (\ref{e01029}). 
	 Let $K_G^i(B)$, $i=0,1$, be the equivariant K-group of $B$.
	 Then the family index map $\ind(D(\mF))\in K_G^*(B)$, where $*=0$ or $1$ corresponds to the even or odd dimensions of fibers $Z$.
	 
	 Let $\mathrm{F}_G^0(B)$ (resp. $\mathrm{F}_G^1(B)$) be the set of equivalence classes of isomorphic equivariant geometric families such that the dimensions of all fibers are even (resp. odd). We denote by $\mF\sim \mF'$ if $\ind(D(\mF))=\ind(D(\mF'))$. The following proposition is proved in \cite{BS13}. 
	 
	 \begin{prop}\label{D01}
	  There is a ring isomorphism
	 	\begin{align}\label{D02}
	 	\mathrm{F}_G^*(B)/\sim\ \simeq K_G^*(B).
	 	\end{align}
	 \end{prop} 
	 	
		Let $D$ be a family of first-order pseudodifferential operators on the fibers of $\mF$, which is
		self-adjoint, fiberwise elliptic and commutes with the $G$-action. Furthermore, 
		if $\mF\in \mathrm{F}_G^1(B)$, we assume that $D$ preserves
		the $\Z_2$-grading of $E$; if $\mF\in \mathrm{F}_G^0(B)$,
		we assume that $D$ anti-commutes with the $\Z_2$-grading
		of $\mS(TZ,L_Z)\widehat{\otimes}E$, where
		$\mS(TZ,L_Z)$ is the spinor with respect to the 
		Spin$^c$ structure of $TZ$.
		As in \cite{DZ98}, we call such $D$ an equivariant $B$-family on $\mF$ (see Definition \ref{d199}). If $\ind(D)=0\in K_G^{*}(B)$ and at least one component of the fiber has nonzero dimension, there exists an equivariant spectral section $P$ (see Definition \ref{d003}) and a family of smoothing operators $A_P$ associated with $P$, such that $D+A_P$ is an invertible equivariant $B$-family (see Proposition \ref{d006}). Let $P$, $Q$ be equivariant spectral sections, we could define the difference $[P-Q]\in K_G^{*}(B)$ (see (\ref{d009}) and (\ref{d138})).
	
		Let $\mF, \mF'\in \mathrm{F}_G^1(B)$ (resp. $\mathrm{F}_G^0(B)$) which have the same topological structure, that is, the only differences between them are horizontal subbundles, metrics and connections. Let $D_0$, $D_1$ be two equivariant $B$-families on $\mF$, $\mF'$ respectively.
		Let $Q_0$, $Q_1$ be equivariant spectral sections of $D_0$, $D_1$ respectively. We define the equivariant higher spectral flow $\mathrm{sf}_G\{(D_0, Q_0), (D_1, Q_1)\}$ between the pairs
		$(D_0, Q_0)$ and $(D_1, Q_1)$ to be an element in $K_G^0(B)$ (resp. $K_G^1(B)$) in Definitions \ref{d152} and \ref{d15}. Note that when $\mF$ is odd, it is the direct extension of the Dai-Zhang higher spectral flow in \cite{DZ98}; when $\mF$ is even, it is defined by adding an additional dimension.
	
	Moreover, besides the equivariant geometric family, we could also represent the elements of equivariant K-group as equivariant higher spectral flows
	(see Proposition 2.7).	
	From this point of view, the equivariant higher spectral flow here is the same as the term $\ind((\mE\times I)_{bt})$ in \cite[2.5.8]{BS13}, which 
	is studied using the KK-theory there. This enable us to replace the techniques of KK-theory in \cite{BS13} by that of equivariant higher spectral flow, which is purely analytic.
	
		Let $D$ be an equivariant $B$-family on $\mF$. A perturbation operator with respect to $D$ is a family of bounded pseudodifferential operators $A$ such that $D+A$ is an invertible equivariant $B$-family on $\mF$, which is a generalization of $A_P$.		
	Note that if at least one component of the fibers of $\mF$ has nonzero dimension, a perturbation operator exists with respect to $D$ if and only if $\ind D=0\in K_G^*(B)$. 	
		
	If the $G$-action on $B$ is trivial, for any $g\in G$, we define the equivariant eta form $\tilde{\eta}_g(\mF, A)
	\in \Omega^*(B,\C)/d\Omega^*(B,\C)$ with respect to a perturbation operator $A$ in Definition \ref{d017}. If the equivariant geometric families $\mF$ and $\mF'$ have the same topological structure,  we prove the anomaly formula as follows.
	
	\begin{thm}\label{D06}
		Assume that the $G$-action on $B$ is trivial.
		Let $\mF$, $\mF'\in \mathrm{F}_G^*(B)$ which have the same topological structure.	Let $A$, $A'$ be perturbation operators with respect to $D(\mF)$, $D(\mF')$
		and $P$, $P'$ be the APS projections onto the eigenspaces of the positive spectrum of $D(\mF)+A$, $D(\mF')+A'$ respectively.
		For any $g\in G$, modulo exact forms on $B$, we have
		\begin{multline}\label{D07}
		\tilde{\eta}_g(\mF', A')-\tilde{\eta}_g(\mF, A)=\int_{Z^g}\widetilde{\td}_g(\nabla^{TZ},\nabla^{L_Z},\nabla^{' TZ},  \nabla^{'L_Z} )
		\, \ch_g(E, \nabla^{E})
		\\
		+\int_{Z^g}\td_g(\nabla^{'TZ}, \nabla^{'L_Z} )\, \widetilde{\ch}_g(\nabla^{E},\nabla^{'E})
		+
		\ch_g\left(\mathrm{sf}_G\{(D(\mF)+A, P), (D(\mF')+A', P')\}\right),
		\end{multline}
		where $Z^g$ is the fixed point set of $g$ on the fibers $Z$ and
		the characteristic forms $\ch_g(\cdot)$, $\td_g(\cdot)$ 
		and the Chern-Simons forms $\widetilde{\ch}_g(\cdot)$, $\widetilde{\td}_g(\cdot)$ are defined in Section 2. 
	\end{thm}
	
	Note that when $\mF$, $\mF'\in \mathrm{F}_G^0(B)$, the proof of the anomaly formula relies on a special case of functoriality of equivariant eta forms. 
	
	If $B=\mathrm{pt}$, $\mF\in \mathrm{F}_G^1(\mathrm{pt})$,
	taking $A=P_{\ker D}$, the orthogonal projection onto the kernel
	of $D(\mF)$, the equivariant eta form here is just the
	classical reduced equivariant APS eta invariant. Using 
	Theorem \ref{D06}, we could write the equivariant spectral flow 
	term in the anomaly formula of eta invariants explicitly. 
	
	
	Let $\pi:V\rightarrow B$ be an equivariant surjective
	 proper submersion with compact orientable  fibers $Y$. We assume that $B$ is compact, $G$ acts trivially on $B$ and $TY$ is equivariant Spin$^c$. Let $\mF_X=(W, L_X, E, o_X, T_{\pi_X}^HW, g^{TX}, h^{L_X}, \nabla^{L_X}, h^E, \nabla^{E})$ be an equivariant 
	geometric family over $V$ for an equivariant surjective proper submersion $\pi_X:W\to V$ with compact orientable fibers $X$ (see (\ref{d155})). Then 
	$\pi_Z:=\pi_Y\circ\pi_X:W\to B$ is an equivariant proper 
	submersion with compact orientable fibers $Z$.
	We could obtain a new equivariant geometric family $\mF_Z$ over $B$ in (\ref{d073}). 
	For any $g\in G$, let $Y^g$ and $Z^g$ be the fixed point sets of $g$ on the fibers $Y$ and $Z$ respectively.
	We obtain the functoriality of equivariant eta forms.
	
	\begin{thm}\label{D08}
		Let $A_Z$ and $A_X$ be perturbation operators with respect to $D(\mF_Z)$ and $D(\mF_X)$.
		Then there exists $T_0\geq 1$, such that for any $T\geq T_0$ and any $g\in G$, modulo exact forms on $B$, we have
		\begin{multline}\label{D09}
		\widetilde{\eta}_g(\mF_{Z}, A_{Z})=\int_{Y^g}\td_g(\nabla^{TY}, \nabla^{L_Y})\, \widetilde{\eta}_g(\mF_{X},A_X)
		\\
		+\int_{Z^g}\widetilde{\td}_g(\nabla^{TY,TX},\nabla^{L_Z},\nabla^{TZ},  \nabla^{L_Z} )
		\, \ch_g(E, \nabla^{E})
		\\
		+ \ch_g(\mathrm{sf}_G\{(D(\mF_{Z,T})+ 1\widehat{\otimes}TA_X, P), (D(\mF_Z)+A_Z, P') \}),
		\end{multline}
		where $\mF_{Z,T}$ is the equivariant geometric family defined in (\ref{d028}), $\nabla^{TY,TX}$ is defined in (\ref{e02114}) and $P$, $P'$ are the associated APS projections respectively.
	\end{thm}
	
	In the last section, inspired by \cite{BS09,BS13,Or09}, we use the results above to define the equivariant differential K-theory for the compact manifolds when $G$
	acts with finite stabilizers and study the properties of it.
	
	Essential to our definition is that when the group action has finite stabilizers, $K_G^*(B)\otimes \R$ is isomorphic to the delocalized cohomology $H_{deloc, G}^*(B,\R)$ defined in (\ref{d102}), which is the cohomology of complex $(\Omega_{deloc, G}^*(B,\R), d)$ of differential forms on the disjoint union of the fixed point set of a representative element in the conjugacy classes. Furthermore, we could define $\widetilde{\eta}_G(\mF, A)\in \Omega_{deloc, G}^*(B,\R)/\Im d$ (see Definition \ref{d122}) when $G$ acts with finite stabilizers on $B$.
	
	A cycle for an equivariant differential K-theory class over $B$
	is a pair $(\mF, \rho)$, where $\mF\in \mathrm{F}_G^*(B)$ and $\rho\in \Omega^{*}_{deloc, G}(B,\R)/\Im \,d$.
	The cycle $(\mF, \rho)$ is called even (resp. odd) if $\mF$ is even (resp. odd) and
	$\rho\in \Omega^{\mathrm{odd}}_{deloc, G}(B,\R)/\Im \,d$ (resp. $\rho\in \Omega^{\mathrm{even}}_{deloc, G}(B,\R)/\Im \,d$).
	Two cycles ($\mF, \rho$) and ($\mF', \rho'$) are called isomorphic if $\mF$ and $\mF'$ are isomorphic
	and $\rho=\rho'$. 
	Let $\widehat{\mathrm{IC}}_G^0(B)$ (resp. $\widehat{\mathrm{IC}}_G^1(B)$) denote the set of isomorphism classes of even (resp. odd) cycles over $B$. 
	Let $\mF^{\mathrm{op}}$ be the equivariant geometric family reversing 
	the $\Z_2$-grading of $E$ in $\mF$, which implies that $\ind(D(\mF^{\mathrm{op}}))=-\ind(D(\mF))$.
		We call two cycles ($\mF, \rho$) and ($\mF', \rho'$) paired if \begin{align}\label{eq:0.04}
		\ind(D(\mF))=\ind(D(\mF')),
		\end{align}
		and there exists
	a perturbation operator $A$ with respect to $D(\mF+\mF^{'\mathrm{op}})$ such that
	\begin{align}\label{D10}
	\rho-\rho'= \widetilde{\eta}_G(\mF+\mF^{'\mathrm{op}}, A).
	\end{align}
	Note that from (\ref{eq:0.04}),  $\ind(\mF+\mF^{'\mathrm{op}})
	=0\in K_G^*(B)$. So $\widetilde{\eta}_G(\mF+\mF^{'\mathrm{op}}, A)$ is well defined.
	Let $\sim$ denote the equivalence relation generated by the relation "paired".
	
	\begin{defn}\label{D11}(Compare with \cite[Definition 2.14]{BS09})
		The equivariant differential K-theory $\widehat{K}_G^0(B)$ (resp. $\widehat{K}_G^{1}(B)$) is the group completion
		of the abelian semigroup $\widehat{\mathrm{IC}}_G^{0}(B)/\sim$ (resp. $\widehat{\mathrm{IC}}_G^{1}(B)/\sim$).
	\end{defn}
	
	Let $\pi_Y: V\rightarrow B$ be an equivariant proper submersion of compact smooth $G$-manifolds with compact fibers $Y$ such that $TY$ is oriented and equivariant Spin$^c$.
	We assume that the $G$-action on $B$ has only finite stabilizers. 
	Thus, so is the action on $V$. As in \cite{BS09}, we define the equivariant differential K-orientation with respect to $\pi_Y$ in Definition \ref{d081} and the map $\widehat{\pi}_Y!:\widehat{\mathrm{IC}}_G^{*}(V)\rightarrow \widehat{\mathrm{IC}}_G^{*}(B)$ in (\ref{d082}). Then using Theorems \ref{D06} and \ref{D08}, we prove that
	
	\begin{thm}
		The map $\widehat{\pi}_Y!:\widehat{K}_G^*(V)\rightarrow \widehat{K}_G^*(B)$ is well-defined.
	\end{thm}
	 
	By Theorems \ref{D06} and \ref{D08}, in Section 3, we also prove that our model is a ring valued functor with the usual properties of the differential extension of a generalized cohomology. Finally, we explain that this model could be naturally regarded as a model of differential K-theory for orbifolds.

	Note that there is no adiabatic limit in Theorem \ref{D08}. So in non-equivariant case, our proofs in the construction of the differential
	K-theory simplify that in \cite{BS09} a little.
		
	This paper is organized as follows.
	
	In Section 1, we give a geometric description of the equivariant K-theory.
	In Section 2, we extend the spectral section to the equivariant case, introduce the equivariant higher spectral flow for arbitrary dimensional fibers and use them to obtain the anomaly formula and the functoriality of the equivariant eta forms.
	In Section 3, we construct an analytic model for the equivariant differential K-theory and prove some properties.
	
	\textbf{Notation}: All manifolds in this paper are smooth and without boundary.
	We denote by $d$ the exterior differential operator and 
	$d^B$ when we like to insist the base manifold $B$.
	
	We use the Einstein summation convention in this paper: when an index variable appears twice in a single 
	term and is not otherwise defined, it implies summation of that 
	term over all the values of the index.
	
	We also use the superconnection formalism of Quillen 
	\cite{Qu85} and Bismut-Cheeger \cite{BC89}. 
	If $\mA$ is a $\Z_2$-graded algebra, and if $a,b\in \mA$, 
	then we will note $[a,b]$ as the supercommutator of $a$, $b$.
	In the whole paper, if $E$, $E'$ are two $\Z_2$-graded spaces
	we will note $E\widehat{\otimes}E'$ as the $\Z_2$-graded 
	tensor product as in \cite[\S 1.3]{BGV04}.
	If one of $E, E'$ is ungraded, we understand 
	it as $\Z_2$-graded by taking its odd part as zero.
	Let $P$ be a trace class operator on a $\Z_2$-graded space
	$E=E_+\oplus E_-$. Let $P|_{E_+}$ and $P|_{E_-}$ be the restrictions of $P$ on $E_+$ and $E_-$ respectively.
	We denote by
	\begin{align}\label{i22}
	\tr_s[P]=\tr[P|_{E_+}]-\tr[P|_{E_-}].
	\end{align}
	
For the fiber bundle $\pi: W\rightarrow B$, we will often use 
the integration of the differential forms along the oriented 
fibres $Z$ in this paper. Since the fibers may be odd 
dimensional, we must make precisely our sign conventions: 
for $\alpha\in \Omega^{\bullet}(B)$ and 
$\beta\in \Omega^{\bullet}(W)$, then
	\begin{align}\label{e01136}
	\int_Z(\pi^*\alpha)\wedge\beta=\alpha\wedge \int_Z\beta.
	\end{align}
		
	\section{Equivariant K-theory}\label{s01} 
	
	In this section, we explain a geometric description of the equivariant K-theory in \cite{BS09,BS13} for any compact Lie group action and define the push-forward map of equivariant
	K-groups in this point of view. In Section \ref{s0101}, we recall some elementary results of 
	the Clifford algebra. In Section \ref{s0102}, we introduce the equivariant geometric 
	family. In Section \ref{s0103}, we give a geometric description of the equivariant
	K-theory. In Section \ref{s0104}, we study the push-forward map in equivariant K-theory using equivariant geometric families. 
	
	\subsection{Clifford algebra}\label{s0101}
	Let $E$ be an oriented Euclidean space of dimension $n$.
	Let $C(E)$ denote the complex Clifford algebra of $E$.
	Relative to an orthonormal basis of $E$,
	 $\{e_i\}_{1\leq i\leq n}$, $C(E)$ is
	defined by the relations
	\begin{align}\label{e01001}
	e_ie_j+e_je_i=-2\delta_{ij}.
	\end{align}
	To avoid ambiguity, we denote by $c(e_i)$ the element of $C(E)$ corresponding to $e_i$. 
	The Clifford algebra $C(E)$ is naturally $\Z_2$-graded from the $\N$-grading of the 
	tensor algebra after reduction mod $2$. We denote by
	$C(E)=C_0(E)\oplus C_1(E)$.
	Let $\mathrm{Spin}_n^c$ be the Spin$^c$ group associated with $C(E)$ (cf. 
	\cite[Appendix D]{LM89}).

If $n=2k$ is even, up to isomorphism, $C(E)$ has a unique irreducible representation, 
the spinor $\mS(E)$, which has a $\Z_2$-grading obtained from the chirality operator
$$\tau_E=(\sqrt{-1})^kc(e_1)\cdots c(e_{2k}).$$ 
We write $\mS(E)=\mS_+(E)\oplus \mS_-(E)$ with respect to this $\Z_2$-grading.
In fact, there are isomorphisms of $\Z_2$-graded algebras
\begin{align}\label{bl0006}
C(E)\simeq \End(\mS(E))\simeq \mS(E)\widehat{\otimes} \mS(E)^*.
\end{align}
For any $A\in C(E)$, we write $\tr_s[A]:=\tr[\tau_EA]$ the supertrace on $\mS(E)$.
Note that $\mS(E)$ is also a representation of 
$\mathrm{Spin}_n^c$ induced by the Clifford action.

If $n=2k-1$ is odd, 
$C(E)$ has only two inequivalent irreducible representations. 
For arbitrary $n$, $c(e_j)\mapsto c(e_j)c(e_{n+1})$, $1\leq j\leq n$,
defines an isomorphism $C(E)\simeq C_0(E\oplus \R)$
of algebras. Since $n$ is odd,
we can regard $\mS_{\pm}(E\oplus \R)$ as the two inequivalent irreducible 
representations of $C(E)$. 
Their restrictions to $\mathrm{Spin}_n^c$ are equivalent. 
In the following, we may and we take $\mS_{+}(E\oplus \R)$
as the spinor for $C(E)$, also denoted by $\mS(E)$ for the convenience. 
In particular, the notation $\tr[\cdot]$ 
on the spinor refers to 
the representation $\mS_{+}(E\oplus \R)$.

Let $F$ be another oriented Euclidean space.  
Let $\mS(E)\widehat{\otimes}\mS(F)$ be the $\Z_2$-graded
tensor product of $\mS(E)$ and $\mS(F)$.
Then it is a $\Z_2$-graded representation of $C(E)\widehat{\otimes}C(F)$
defined by
\begin{align}\label{eq:0.03}
(a_1\widehat{\otimes} a_2)(s_1\widehat{\otimes}s_2)
=(-1)^{|a_2|\cdot |s_1|}(a_1s_1)\widehat{\otimes}(a_2s_2),
\end{align}
where $a_1\in C(E)$, $a_2\in C(F)$, $s_1\in \mS(E)$, $s_2\in 
\mS(F)$ and $|a_2|$, $|s_1|$ are degrees of $a_2$, $s_1$
associated with the $\Z_2$-gradings of $C(F)$, $\mS(E)$
respectively. 
We express $\mathcal{S}(E) \widehat{\otimes} 
\mathcal{S}(F)$ by $\mS(E)$ and $\mS(F)$ using ungraded tensor product as follows (cf. \cite[(1.10), (1.11)]{BC89}).
%

If both $\dim E$ and $\dim F$ are odd, let $\C^2=\C\oplus \C$ define
the grading on $\C^2$ and let $J,K\in \End(\C^2)$ denote the involutions
\begin{align}\label{eq:1.04}
J=\left(\begin{array}{cc}
0 & -\sqrt{-1} \\
\sqrt{-1} & 0
\end{array}
\right),\quad
K=\left(\begin{array}{cc}
0 & 1 \\
1 & 0
\end{array}
\right).
\end{align}
Note that $J^2=K^2=1$, $JK=-KJ$. Then 
 $\mS(E)\otimes \mS(F)\otimes \C^2$ with involution
$1\otimes 1\otimes \sqrt{-1} JK$ is the unique irreducible $\Z_2$-graded 
representation of $C(E)\widehat{\otimes}C(F)$ defined by
\begin{align}\label{eq:1.05}
a_i\widehat{\otimes} b_j\rightarrow a_i\otimes b_j\otimes J^iK^j,\quad 
a_i\in C_i(E), b_j\in C_j(F).
\end{align}
It is isomorphic to $\mathcal{S}(E) \widehat{\otimes} 
\mathcal{S}(F)$ as $\Z_2$-graded $C(E)\widehat{\otimes}C(F)$-representations.

If $\dim E$ is even and $\dim F$ is odd, 
then as representations,
$\mS(E)\widehat{\otimes}\mS(F)$ is isomorphic to
$\mS(E)\otimes \mS(F)$ with 
$C(E)\widehat{\otimes}C(F)$-action defined by 
\begin{align}\label{eq:1.07}
a\widehat{\otimes} b_i\rightarrow a \tau_E^i\otimes b_i,\quad 
a\in C(E), b_i\in C_i(F).
\end{align}

If $\dim E$ is odd and  $\dim F$ is even, 
then as representations,
$\mS(E)\widehat{\otimes}\mS(F)$ is isomorphic to
$\mS(E)\otimes \mS(F)$ with 
$C(E)\widehat{\otimes}C(F)$-action defined by 
\begin{align}\label{eq:1.06}
a_i\widehat{\otimes} b\rightarrow a_i\otimes \tau_F^i b,\quad 
a_i\in C_i(E), b\in C(F).
\end{align}


If both $\dim E$ and $\dim F$ are even, the representation $\mS(E)\otimes \mS(F)$ with $C(E)\widehat{\otimes}C(F)$-action
defined by (\ref{eq:1.06})
is the unique irreducible one and $\Z_2$-graded for the tensor product
grading on $\mS(E)\otimes \mS(F)$. It is isomorphic to 
$\mS(E)\widehat{\otimes} \mS(F)$ as representations.

Since
\begin{align}\label{e01087}
C(E\oplus F)\simeq C(E)\widehat{\otimes}C(F),
\end{align}
$\mS(E)\widehat{\otimes}\mS(F)$ is also a $\Z_2$-graded representation
of $C(E\oplus F)$.
By (\ref{e01087}), we have the isomorphism of representations
\begin{align}
\mathcal{S}(E\oplus F)\simeq \mathcal{S}(E) \widehat{\otimes} 
\mathcal{S}(F). 
\end{align}

	\subsection{Equivariant geometric family}\label{s0102}
	
	In this subsection, we introduce the equivariant geometric family
	(cf. \cite{BS09,BS13}).
	
	Let $\pi:W\rightarrow B$ be a smooth surjective proper submersion of compact manifolds with compact fibers $Z$ (possibly non-connected). 
	Let $B=\sqcup_i B_i$ be a finite disjoint union of compact connected manifolds.
	Let $W_i$ be the restriction of $W$ on $B_i$. Let $W_i=\sqcup_jW_{ij}$ be a finite disjoint union of compact connected manifolds. Let $Z_{ij}$ be the fibers of the submersions restricted on $W_{ij}$.
	 We note here that the dimension of $Z_{ij}$ might be zero. In the sequel, we will often omit the subscripts $i,j$.
	
	Let $TZ=\ker(d\pi)$ be the relative tangent bundle to the fibers $Z$ over $W$, which is a subbundle of $TW$.
	We assume that $TZ$ is orientable and carries an orientation $o\in H^0(W,\Z_2)$.
	Let $T_{\pi}^HW$ be a horizontal subbundle of $TW$ such that
	\begin{align}\label{e01010}
	TW=T_{\pi}^HW\oplus TZ.
	\end{align}
	The splitting (\ref{e01010}) gives an identification
	\begin{align}\label{e01011}
	T_{\pi}^HW\cong \pi^*TB.
	\end{align}
	If there is no ambiguity, we will omit the subscript $\pi$ in $T_{\pi}^HW$.
	Let $P^{TZ}$ be the projection
	\begin{align}\label{e01012}
	P^{TZ}:TW=T_{}^HW\oplus TZ \rightarrow TZ.
	\end{align}
	
	Let $g^{TZ}$, $g^{TB}$ be Riemannian metrics on $TZ$, $TB$. We equip $TW=T_{}^HW\oplus TZ$ with the Riemannian metric
	\begin{align}\label{e01013}
	g^{TW}=\pi^*g^{TB}\oplus g^{TZ}.
	\end{align}
	
	Let $\nabla^{TW}$ be the Levi-Civita connection on $(W, g^{TW})$. Set
	\begin{align}\label{e01014}
	\nabla^{TZ}=P^{TZ}\nabla^{TW}P^{TZ}.
	\end{align}
	Then $\nabla^{TZ}$ is a Euclidean connection on $TZ$. By \cite[Theorem 1.9]{Bi86}, we know that
	$\nabla^{TZ}$ only depends on $(T_{}^HW, g^{TZ})$.
		
	Let $C(TZ)$ be the Clifford algebra bundle of $(TZ, g^{TZ})$, whose fiber at $x\in W$ is the Clifford algebra
	$C(T_xZ)$ of the Euclidean space $(T_xZ, g^{T_xZ})$.
	
	\textbf{We make the assumption that the oriented vector
		bundle $(TZ, o)$ has a Spin$^c$ structure.} Then there exists a complex line bundle $L_Z$
	over $W$ such that $\omega_2(TZ)= c_1(L_Z) \mod 2$,
	where $\omega_2$ denotes the second Stiefel-Whitney class
	and $c_1$ denotes the first Chern class.
	Let $\mS(TZ,L_Z)$ be the fundamental
	complex spinor bundle for $(TZ,L_Z)$, which
	has a smooth action of $C(TZ)$ (cf. \cite[Appendix D.9]{LM89}).
	Locally, the spinor bundle $\mS(TZ,L_Z)$ may be written
	as
	\begin{align}\label{e01071}
	\mS(TZ,L_Z) = \mS(TZ)\otimes L_Z^{1/2},
	\end{align}
	where $\mS(TZ)$ is the fundamental spinor bundle for the (possibly non-existent) spin
	structure on $TZ$, and $L_Z^{1/2}$ is the (possibly non-existent) square root of $L_Z$.
	Let $h^{L_Z}$ be a Hermitian metric on $L_Z$.
	Then from (\ref{e01071}), the metrics $g^{TZ}$ and $h^{L_Z}$
	induce a Hermitian metric on $\mS(TZ,L_Z)$, which we denote by 
	$h^{\mS_Z}$ for simplicity. 
	Let $\nabla^{L_Z}$ be a Hermitian connection on $(L_Z, h^{L_Z})$. Similarly, we denote by $\nabla^{\mS_Z}$ the 
	connection on $\mS(TZ,L_Z)$ induced by $\nabla^{TZ}$
	and $\nabla^L$ from (\ref{e01071}).	
	Then $\nabla^{\mS_Z}$ is a Hermitian connection on ($\mS(TZ,L_Z), h^{\mS_Z}$).
	Moreover, it is a Clifford connection associated with $\nabla^{TZ}$, i.e., for any $U\in TW$, $V\in \cC^{\infty}(W,TZ)$,
	\begin{align}\label{e01024}
	\left[\nabla_U^{\mS_Z}, c(V)\right]=c\left(\nabla^{TZ}_UV\right).
	\end{align}
	In the following, we often simply denote the spinor bundle
	$\mS(TZ, L_Z)$ by $\mS_Z$.
	If $n=\dim Z$ is even, $\mS_Z$ is $\Z_2$-graded and the action of $TZ$ exchanges the $\Z_2$-grading. 
	
	Let $E=E_+\oplus E_-$ be a $\Z_2$-graded smooth complex vector bundle over $W$ with Hermitian metric $h^E$, for which $E_+$ and $E_-$ are orthogonal, and let $\nabla^E$ be a Hermitian connection on $(E, h^E)$ preserving the $\Z_2$-grading. Set
	\begin{align}
	\nabla^{\mS_Z\widehat{\otimes} E}:=\nabla^{\mS_Z}\widehat{\otimes} 1+1\widehat{\otimes} \nabla^E.
	\end{align}
	Then $\nabla^{\mS_Z\widehat{\otimes} E}$ is a Hermitian connection on $(\mS_Z\widehat{\otimes} E, h^{\mS_Z}\otimes h^E)$.
		
	Let $G$ be a compact Lie group which acts on $W$ and $B$ such that for any $g\in G$, $\pi\circ g=g\circ\pi$.
	We assume that the $G$-action preserves the splitting (\ref{e01010}) and the $\mathrm{Spin}^c$ structure of $TZ$. Thus $TZ$, $L_Z$, $\mS_Z$ are 
	$G$-equivariant vector bundles.
	We assume that $g^{TZ}$, $h^{L_Z}$, $\nabla^{L_Z}$ are $G$-invariant.
	We further assume that $E$ is a $G$-equivariant $\Z_2$-graded complex vector bundle and $h^E$, $\nabla^E$ are $G$-invariant.
	Note that the $G$-action here may be nontrivial on $B$.
	
	\begin{defn}\label{d001}(Compare with \cite[Definition 2.2]{BS09})
		An equivariant geometric family $\mF$ over $B$ is a family of $G$-equivariant geometric data
		\begin{align}\label{d002}
		\mF=(W, L_Z, E, o, T_{}^HW, g^{TZ}, h^{L_Z}, \nabla^{L_Z}, h^E, \nabla^{E})
		\end{align}
		described above. We call the equivariant geometric family $\mF$ is even (resp. odd) if
		for any connected component of fibers, the dimension is even (resp. odd).
	\end{defn}
	
	\begin{defn}\label{d125}(Compare with \cite[\S 2.1.7]{BS09})
		Let $\mF$ and $\mF'$ be two equivariant geometric families over $B$. An isomorphism $\mF\overset{\sim}{\rightarrow}\mF'$
		consists of the following data:
		\begin{center}\label{d044}
			\begin{tikzpicture}[>=angle 90]
			\matrix(a)[matrix of math nodes,
			row sep=2em, column sep=2.5em,
			text height=1.5ex, text depth=0.25ex]
			{ E & & E'\\
				W& &W'\\
				& B & \\};
			\path[->](a-1-1) edge node[above]{\footnotesize{$f_E
					$}} (a-1-3);
			\path[->](a-2-1) edge node[above]{\footnotesize{$f$}} (a-2-3);
			\path[->](a-1-1) edge node[right]{} (a-2-1);
			\path[->](a-1-3) edge node[above]{} (a-2-3);
			\path[->](a-2-1) edge node[above]{\footnotesize{$\pi$}} (a-3-2);
			\path[->](a-2-3) edge node[above]{\footnotesize{$\pi'$}} (a-3-2);
			\end{tikzpicture}
		\end{center}
		where
		
		1. $f$ is a diffeomorphism commuting with the $G$-action such that $\pi'\circ f=\pi$, which implies  that 
		$f$ preserves the relative tangent bundle;
		
		2. $f$ preserves the orientation and the Spin$^c$ structure of the relative tangent bundle, which implies that there exists an equivariant complex line bundle isomorphism
		$f_L: L_Z\to L_Z'$;
		
		3. $f_E:E\to E'$ is an equivariant vector bundle isomorphism over $f$, which preserves the $\Z_2$-grading;
			
		4. $f$ preserves the horizontal subbundle and the vertical metric;
		
		5. $f_L$ and $f_E$ preserve the metrics and the connections on the vector bundles.
		
		If only the first three conditions hold, we say that $\mF$ and $\mF'$ have\textbf{ the same topological structure}.
				
		Let $\mathrm{F}_G^0(B)$ (resp. $\mathrm{F}_G^1(B)$) be the set of equivalence classes of even (resp. odd) equivariant geometric families over $B$.
	\end{defn}
		
	For two equivariant geometric families $\mF,\mF'$ over $B$, we can form their sum $\mF + \mF'$ over $B$ as a new equivariant geometric family:
	the underlying fibration of the sum is $\pi\sqcup \pi':W\sqcup W'\rightarrow B$, where $\sqcup$
	is the disjoint union and the remaining structures of $\mF+\mF'$ are induced in the obvious way.  Let $\mathrm{F}_G^*(B)=\mathrm{F}_G^0(B)\oplus \mathrm{F}_G^1(B)$. It is an 	additive abelian semigroup.

	
		For $\mF,\mF'\in \mathrm{F}_G^*(B)$, we can also form their product $\mF \times_B \mF'$ over $B$.
	The total space of the underlying fibration of $\mF \times_B \mF'$ is $ W\times_B W':=\{(w,w')\in W\times W' : \pi(w)=\pi'(w') \}$ and the fiber is $Z\times Z'$. 
	Let $\mathrm{pr}_W:W\times_B W'\rightarrow W$ and $\mathrm{pr}_{W'}:W\times_B W'\rightarrow W'$
	be the obvious projections.
	The complex vector bundle now is $\mathrm{pr}_W^*E \widehat{\otimes} \mathrm{pr}_{W'}^*E'$. 
	The remaining structures of $\mF\times_B \mF'$ are induced in the obvious way. 	

	Let $B$, $B'$ be two compact manifolds with smooth $G$-action.
Let $f:B\times B'\rightarrow B$ be the projection onto the first part. For any $\mF\in \mathrm{F}_G^*(B)$, we could construct the pullback $f^*\mF\in \mathrm{F}_G^*(B\times B')$ in a natural way. 
Remark that in general case, for a $G$-equivariant map
$f:B'\to B$ and $\mF\in \mathrm{F}_G^*(B)$, $f^*\mF$
is hard to define canonically because we cannot choose a canonical
horizontal subbundle in $f^*\mF$. We will show more details
in Section \ref{s0303} later.
	
	\begin{defn}\label{d041}
		The opposite family $\mF^{\mathrm{op}}$ of an equivariant geometric family $\mF$ is obtained by reversing 
		the $\Z_2$-grading of $E$.
	\end{defn}

	\subsection{Equivariant K-Theory}\label{s0103}
	
	In this subsection, we give some examples of the equivariant geometric families and a geometric description of the equivariant K-theory.
	
	Let $K_G^0(B)$ be the $G$-equivariant $K^0$-group of $B$, which is the Grothendieck group of the equivalence classes of $G$-equivariant topological complex vector bundles over $B$ (cf. \cite{Se68}). Since $G$ is compact, 
	by Proposition \ref{prop:a04}, it is also the 
	Grothendieck group of the 
	equivalence classes of $G$-equivariant smooth complex vector bundles. Note that the ring structure of the $K^0$-group is
	induced by the tensor product of the
	complex vector bundles.
	
		We lift the $G$-action on $B\times S^1$ such that the $G$-action on $S^1$ is trivial. Take $s\in S^1$ fixed.
		Let $i:B\ni b\rightarrow (b,s)\in B\times S^1$ be the $G$-equivariant inclusion map.
		Let $i^*:K_G^0(B\times S^1)\to K_G^0(B)$ be the induced
		homomorphism.
	Let $K_{G}^{1}(B)$ be the $G$-equivariant
	$K^{1}$-group 
	of $B$.
	By \cite[Definitions 2.7 and 2.8]{Se68}, we have the split short exact sequence
	\begin{align}\label{d124}
	0\longrightarrow K_G^1(B)\overset{j}{\longrightarrow} K_G^0(B\times S^1)\overset{i^*}{\longrightarrow} K_G^0(B)\longrightarrow 0,
	\end{align}
	where $j$ is induced by the suspension isomorphism $K_G^1(B)\simeq \widetilde{K}_G^0( B\wedge S^1)\simeq \ker(i^*)$ (cf. \cite[p136]{Se68}). Here $B\wedge S^1$ 
	is the smash product of $B$ and $S^1$ and $\widetilde{K}_G^0( B\wedge S^1)$ is the $G$-equivariant reduced $K^0$-group
	of $B\wedge S^1$. 
	
	Now we introduce another explanation of $K_G^1(B)$.
	Let $V$ be a finite dimensional complex unitary 
	representation of $G$. 
	If $F\in\cC^{\infty}(B,\End(V))$ such
	that for any $b\in B$, $F(b)\in \End(V)$ is unitary and for any $g\in G$,
	$v\in V$,
	\begin{align}\label{1121}
	g(F(b)v)=F(gb)(gv),
	\end{align}
	we say $F$ is a $G$-invariant unitary element of $ \cC^{\infty}(B,\End(V))$. In this case,
	for $(b,t,v)\in B\times [0,1]\times V$, the relation $(b,0,v)\sim (b,1, F(b)v)$ forms a $G$-equivariant smooth Hermitian vector bundle $W$ over $B\times S^1$. Let $U=B\times S^1\times V$ be the $G$-equivariant trivial bundle  over $B\times S^1$ as in (\ref{eq:a02}). Then from (\ref{d124}), $[W]-[U]\in \ker(i^*)$ corresponds to an element $[F]\in K_G^1(B)$.

	\begin{lemma}\label{lemma:1.04}
	For any $y\in K_G^1(B)$, there exists a finite dimensional
	complex unitary representation $V$ of $G$, such that $y$ can be represented as a $G$-invariant unitary element of $ \cC^{\infty}(B,\End(V))$. 
	\end{lemma}
	\begin{proof}
		By (\ref{d124}), an element $y\in K_G^1(B)$ can be
		represented as an 
element $x=j(y)\in K_G^0(B\times S^1)$ such that $i^*x=0\in K_G^0(B)$.
We write $x=W-U$, where $W$ and $U$ are equivariant smooth 
complex vector bundles over $B\times S^1$. 
By Proposition \ref{prop:a02}, we may and we will assume that $U$ is an equivariant 
trivial complex vector bundle over $B\times S^1$
associated with a finite dimensional complex $G$-representation $V$ as in (\ref{eq:a02}).
Note that $B\times S^1\simeq B\times \R/\Z$. We assume that $i(B)=B\times \{1/2\}$. Since $i^*x=W|_{B\times \{1/2\}}-U|_{B\times \{1/2\}}
=0\in K_G^0(B)$, by Proposition \ref{prop:a02}, we may and we will assume that $W|_{B\times \{1/2\}}$
is isomorphic to the equivariant trivial bundle $(B\times\{1/2\})\times V$ over $B\times\{1/2\}$ as equivariant smooth complex vector bundles.
Since $(0,1)$ is contractible,  as equivariant smooth complex vector bundles over $B\times (0,1)$, $W|_{B\times (0,1)}\simeq (\Id_B\times p_{1/2})^*(W|_{B\times \{1/2\}})$ where $p_{1/2}:(0,1)\to 1/2$
is the constant map. Since $B\times (0,1)\times V=
(\Id_B\times p_{1/2})^*((B\times\{1/2\})\times V)$ as complex vector
bundles over $B\times (0,1)$,
there exists a $G$-equivariant smooth complex vector bundle isomorphism 
\begin{align}\label{1123}
f:W|_{B\times (0,1)}\rightarrow B\times (0,1)\times V.
\end{align} 
Let $h:B\times (0,1)\times V\to V$ be the obvious projection.
For any $b\in B$, $v\in V$, we could choose a section $s\in \cC^{\infty}(B\times S^1,W)$
such that $\lim_{t\rightarrow 0}h\circ f(s(b,t))=v$. Then we define 
\begin{align}
F(b)v:=\lim_{t\to 1}h\circ f(s(b,t))\in V.
\end{align}
Note that the definition of $F(b)\in \End(V)$ does not depend on the choices 
of the isomorphic map $f$ and the section $s$. 
Take a $G$-invariant Hermitian metric on $W$ which induces a $G$-invariant
Hermitian inner product on $V$.
It is obvious that $F$ is a $G$-invariant unitary element of $\cC^{\infty}(B,\End(V))$ and $[F]=y\in K_G^1(B)$.

The proof of Lemma \ref{lemma:1.04} is completed.
	\end{proof}

For an equivariant geometric family $\mF$, the fiberwise Dirac operator $D(\mF)$ associated with $\mF$ is defined by	
	\begin{align}\label{e01029}
	D(\mF):=\sum_ic(e_i)\nabla_{e_i}^{\mS_Z\widehat{\otimes} E},
	\end{align}
	where $\{e_i\}$ is a local orthonormal frame of $TZ$.
	Note that the definition of the fiberwise Dirac operator
	is independent of the choice of the local orthonormal frame.
	From (\ref{e01029}), the $G$-action commutes with $D(\mF)$.
	If $\mF$ is isomorphic to $\mF'$, from Definition \ref{d125}
	and (\ref{e01029}),
	the isomorphism preserves the fiberwise Dirac operator. 
	So the fiberwise Dirac operator can be defined on 
	an element of $\mathrm{F}_G^*(B)$. 
	For an even (resp. odd) equivariant geometric family $\mF$,
	the classical construction of Atiyah-Singer
	assigns to this family its equivariant (analytic)
	index $\ind(D(\mF))\in K_G^0(B)$ (resp. $K_G^1(B)$)  (cf. \cite{AS69,AS71}).
	Remark that $\ind(D(\mF))$ depends only on the topological
	structure of $\mF$. It induces a map
		\begin{align}\label{d126}
	\begin{split}
	\ind: \mathrm{F}_G^*(B)&\rightarrow K_G^*(B),
	\\
	\mF\quad&\mapsto \ind(D(\mF)).
	\end{split}
	\end{align}	
	Let $K_G^*(B)=K_G^0(B)\oplus K_G^1(B)$.
	Since 
	\begin{align}\label{d171}
	\ind(D(\mF+\mF'))=\ind (D(\mF))+ \ind(D(\mF'))\in K_G^*(B),
    \end{align}
	the equivariant index map in (\ref{d126}) is a semigroup homomorphism.
		It is well-known that if $\mF$ and $\mF'$ are even,
		\begin{align}\label{d170}
		\ind(D(\mF\times_B \mF'))=\ind (D(\mF))\cdot \ind(D(\mF'))\in K_G^0(B).
		\end{align}

	\begin{example}\label{d129}
		a) Let $(E, h^E)$ be an equivariant $\Z_2$-graded smooth Hermitian vector bundle over $B$ with a $G$-invariant Hermitian connection $\nabla^E$. Then $(E, h^E, \nabla^E)$ can be regarded as an even equivariant geometric family $\mF$ for $Z=\mathrm{pt}$. In this case, $D(\mF)=0$ and $\ind(D(\mF))=[E_+]-[E_-]\in K_G^0(B)$.
		
		b) Let $W=B\times \C P^1$ with $G$-action which acts trivially on $\C P^1$.
		Then the complex line bundle $\mathcal{O}(1)$ over $\C P^1$ can be naturally extended on $W$. Thus $(W, \mathcal{O}(1))$ with canonical metrics, connections, the standard orientation $o$ of $\C P^1$ and the Spin structure on $\C P^1$ form an even equivariant geometric family $\mF_S$ over $B$. 
		Let $D_{\C P^1}^{\mathcal{O}(1)}$ be the Dirac operator on $\C P^1$ associated with $\mathcal{O}(1)$.
		Since $\ind(D_{\C P^1}^{\mathcal{O}(1)})=\la c_1(\mathcal{O}(1)), [\C P^1]\ra=1$, from (\ref{d170}), for even equivariant geometric family $\mF$ in a), we have $\ind(D(\mF\times_B \mF_S))=\ind(D(\mF))\in K_G^0(B)$.
		
		c) (Compare with \cite[\S 5]{MP97b} and \cite[\S 2.2.3.8]{Bu09})
		Let $B=S_{\theta}^1= \R/\Z$,
		$W=S_{\theta}^1\times S_{t}^1$ and $\pi:W\to B$ be the projection onto the first part. We consider the Hermitian line
		bundle $(L,h^L)$ which  is  obtained  by  identifying  
		\begin{align}\label{eq:1.25}
		(\theta=0, t, v),\ 
		(\theta=1, t, \exp(-2\pi t\sqrt{-1})v)\in [0,1]\times S_t^1 \times\C.
		\end{align}		
		Then 		
%
		\begin{align}\label{eq:1.26}
		\nabla^L=d+2\pi \left(\theta-1/2\right)\sqrt{-1}dt
		\end{align} 
		is a Hermitian connection on $(L,h^L)$ (cf. \cite[p.124]{BF86II}).
		We choose the $\Z_2$-grading of $L$ such that $L_+=L$ and $L_-=0$.
		We consider the Spin structure on $S_t^1$ as the desired Spin$^c$ structure. Then we get an odd geometric family $\mF^L$ after choosing the natural geometric data. 
		In fact, since $c_1(L)=dtd\theta$, $\ind(D(\mF^L))$ is a generator of $K^1(S^1)\simeq \Z$ by family index theorem. 

		d) Let $\mF
		\in\mathrm{F}_G^*(B)$.
		Let $p_1$ and $p_2$ be the projections onto the first and second parts of $B\times S^1$ respectively. We take $\mF^L$ as in c). Then $p_1^*\mF\times_{B\times S^1} p_2^*\mF^L$ is an equivariant geometric family over $B\times S^1$.	
		From Proposition \ref{prop:A.01} (cf. also the proof of \cite[Theorem 2.10]{BF86II}), for $\mF\in \mathrm{F}_G^1(B)$, there exists an inclusion $i:B\rightarrow B\times S^1$ such that $i^*\ind(D(p_1^*\mF\times_{B\times S^1} p_2^*\mF^L))=0$. Moreover, as an element of $K_G^1(B)$ in the sense of (\ref{d124}), by 
		an equivariant version of \cite[Proposition 6]{MP97b}, we have
		\begin{align}\label{d127}
		j\big(\ind(D(\mF))\big)=\ind(D(p_1^*\mF\times_{B\times S^1} p_2^*\mF^L)),
		\end{align}
		where $j$ is the map in (\ref{d124}).
		This example is essential in our construction of the higher spectral flow for even case. For the sake of reader's convenience, we will show more details in Appendix B.
	\end{example}
	
	We denote by $\mF\sim \mF'$ if $\ind(D(\mF))=\ind(D(\mF'))$. It is an equivalent relation and compatible with the semigroup structure. So $\mathrm{F}_G^*(B)/\sim$ is a semigroup and the map
		\begin{align}\label{d172}
		\ind: \mathrm{F}_G^*(B)/\sim\ \longrightarrow K_G^*(B)
		\end{align}
		is an injective semigroup homomorphism. 
		
		By Definition \ref{d041}, we have
		\begin{align}\label{d042}
		\ind(D(\mF^{\mathrm{op}}))=-\ind(D(\mF)).
		\end{align}
		After defining $-\mF:=\mF^{\mathrm{op}}$, the semigroup $\mathrm{F}_G^*(B)/\sim$ can be regarded as an abelian group.
		So, by (\ref{d042}), the equivariant index map in (\ref{d172}) is a group homomorphism.
		Note that  $K_G^*(B)$
		has a ring structure \cite{AtK},
		and by (\ref{d127}), (\ref{d170}) holds for any
		$\mF, \mF'\in  \mathrm{F}_G^*(B)$. Thus the equivariant index map in (\ref{d172}) is also a ring homomorphism. 
In fact, it is even a ring isomorphism \cite[\S 2.5.5]{BS13}.
We rewrite the proof
	in our notation here for the completeness.
		
	\begin{prop}\label{d128}
		The equivariant index map $\ind$ in (\ref{d172}) is surjective. In other words, we have the $\Z_2$-graded ring isomorphism
		\begin{align}\label{d045}
		 \mathrm{F}_G^*(B)/\sim\ \simeq K_G^*(B).
		\end{align}
	\end{prop}
	\begin{proof}
		When $*=0$, we can get the proposition directly from Example \ref{d129} a) or b).
		
		When $*=1$, from the proof of Lemma \ref{lemma:1.04}, for any $[F]\in K_G^1(B)$, there exist equivariant complex vector bundles $W$ and $U$ such that
		$[W]-[U]\in K_G^0(B\times S^1)$ corresponds to $[F]\in K_G^1(B)$.
		Moreover, after taking the natural geometric data, we get an odd equivariant geometric family $\mF$ over $B$ with fibers $S^1$ and $\Z_2$-graded equivariant complex vector bundle $W\oplus U$. As in \cite[Section 2.2.2.3]{Bu09}, we have $\ind(D(\mF))=[F]\in K_G^1(B)$.
		
		The proof of Proposition \ref{d128} is completed.
	\end{proof}
		
	\begin{rem}\label{d153}
		Note that if we replace the Spin$^c$ condition of the geometric family by the general Clifford module condition (which is the setting in \cite{BS09,BS13}) or the Spin condition,
		Proposition \ref{d128} also holds. Since we don't use the language of Clifford modules here, our definition of $\mF^{\mathrm{op}}$ in Definition \ref{d041} is simpler than that in \cite{BS09,BS13}. In fact, in the sense of (\ref{d045}), they are the same.
	\end{rem}
	
	\subsection{Push-forward map}\label{s0104}
	
	In this subsection, we define the push-forward map in equivariant K-theory using the equivariant geometric families.
	
	Let $\pi_Y:V\rightarrow B$ be a $G$-equivariant smooth surjective proper submersion of compact manifolds with compact orientable fibers $Y$. We simply assume that the dimensions of all connected components of $Y$ have the same parity. 
	Let $o_Y\in H^0(V,\Z_2)$ be an orientation of the 
	relative tangent bundle $TY$.
	
	\begin{defn}(Compare with \cite[Definition 3.1]{BS13})\label{d154}
		An equivariant K-orientation of $\pi_Y$ is
		 an equivariant Spin$^c$ structure of $TY$.
		Let $\mathcal{O}_G(\pi_Y)$ be the set of equivariant K-orientations.
	\end{defn}

	Suppose that $\pi_Y$ has an equivariant K-orientation $\mathcal{O}_Y\in \mathcal{O}_G(\pi_Y)$. For $j=0,1$, let $N(j):=j$ (resp. $(j+1)\,\mathrm{mod}\,2$) if $\dim Y$ is even (resp. odd). We will use Proposition \ref{d128} to define
	the push-forward map of equivariant K-groups $\pi_Y!: K_G^{j}(V)\rightarrow K_G^{N(j)}(B)$ as follows.
	
	Let $\pi_X:W\rightarrow V$ be the submersion with compact orientable fibers $X$. 
	Let
	\begin{align}\label{d155}
	\mF_X=(W, L_X, E, o_X, T_{\pi_X}^HW, g^{TX}, h^{L_X}, \nabla^{L_X}, h^E, \nabla^{E})
	\end{align}
	be a $G$-equivariant geometric family in $\mathrm{F}_G^{j}(V)$. Then $\pi_Z:=\pi_Y\circ \pi_X: W\rightarrow B$ is a smooth
	submersion with compact orientable fibers $Z$. 
	We have the diagram of smooth fibrations:
	
	\begin{center}\label{e02001}
		\begin{tikzpicture}[>=angle 90]
		\matrix(a)[matrix of math nodes,
		row sep=2em, column sep=2.5em,
		text height=1.5ex, text depth=0.25ex]
		{X&Z&W\\
			&Y&V&B.\\};
		\path[->](a-1-1) edge (a-1-2);
		\path[->](a-1-2) edge node[left]{\footnotesize{}} (a-2-2);
		\path[->](a-1-2) edge (a-1-3);
		\path[->](a-2-2) edge (a-2-3);
		\path[->](a-1-3) edge node[left]{\footnotesize{$\pi_X$}} (a-2-3);
		\path[->](a-2-3) edge node[above]{\footnotesize{$\pi_Y$}} (a-2-4);
		\path[->](a-1-3) edge node[above]{\footnotesize{$\pi_Z$}} (a-2-4);
		\end{tikzpicture}
	\end{center}

	Set $T_{\pi_X}^HZ:=T_{\pi_X}^HW\cap TZ$. Then we have the splitting of smooth vector bundles over $W$,
	\begin{align}\label{e02003}
	TZ=T_{\pi_X}^HZ\oplus TX,
	\end{align}
	and
	\begin{align}\label{e02002}
	T_{\pi_X}^HZ\cong \pi_X^*TY.
	\end{align}
		Let $o_Z:=\pi_X^*o_Y\cup o_X\in H^0(W,\Z_2)$.
	Since $TY$ and $TX$ have equivariant Spin$^c$ structures, so is $TZ$. Let $L_Y$ be the equivariant complex line bundle associated with the equivariant Spin$^c$ structure of $TY$.	
	Set
	\begin{align}\label{d156}
	L_Z:=\pi_X^*L_Y\otimes L_X.
	\end{align}
	
	
	Let $g^{TY}$ be a $G$-invariant Riemannian metric on $TY$.
	Let $h^{L_Y}$ be a $G$-invariant Hermitian metric on $L_Y$
	and $\nabla^{L_Y}$ be a $G$-invariant Hermitian connection
	on $(L_Y, h^{L_Y})$.
	Take geometric data 
	$(T_{\pi_Z}^HW, g^{TZ}, h^{L_Z}, \nabla^{L_Z})$ of $\pi_Z$ such that $T_{\pi_Z}^HW\subset T_{\pi_X}^HW$, $g^{TZ}=\pi_X^*g^{TY}\oplus g^{TX}$, $h^{L_Z}=\pi_X^*h^{L_Y}\otimes h^{L_X}$ and $\nabla^{L_Z}=\pi_X^*\nabla^{L_Y}\otimes 1+1\otimes \nabla^{L_X}$. We get a new equivariant geometric family over $B$,
	\begin{align}\label{d073}
	\mF_{Z}:=(W, L_Z, E, o_Z, T_{\pi_Z}^HW, g^{TZ}, h^{L_Z}, \nabla^{L_Z}, h^E, \nabla^{E}).
	\end{align}
	We write $\mF_Z=\pi_Y!(\mF_X)$.
	
	\begin{thm}\label{d157}
		For equivariant K-orientation $\mathcal{O}_Y\in \mathcal{O}_G(\pi_Y)$ fixed,
		the push-forward map
		\begin{align}\label{d158}
		\begin{split}
		\pi_Y!: K_G^{j}(V)&\rightarrow K_G^{N(j)}(B),
		\\
		[\mF_X]\quad&\mapsto [\mF_Z]
		\end{split}
		\end{align}
		is a well-defined group homomorphism and independent of the geometric data
		$(T_{\pi_Z}^HW, g^{TY}, h^{L_Y}, \nabla^{L_Y})$.
	\end{thm}
	\begin{proof}
		We firstly assume that the map $\pi_Y!$ is well-defined.
		Then the remaining results follow from the definition
		of the equivariant family index.
		
		The well-defined property of $\pi_Y!$ will be proved in Section \ref{s0209} later.
	\end{proof}
	
	Let $\pi_U:B\rightarrow S$ be a $G$-equivariant smooth surjective proper submersion of compact manifolds with compact oriented fibers $U$ and an
	equivariant K-orientation $\mathcal{O}_U$. 
	Then $\pi_A:=\pi_U\circ \pi_Y: V\rightarrow S$ is a $G$-equivariant smooth submersion with an
	equivariant K-orientation constructed by $\mathcal{O}_Y$ and $\mathcal{O}_U$.
	From the construction of the push-forward map and Theorem \ref{d157}, the following theorem is obvious.
	\begin{thm}\label{d194}
		We have the equality of homomorphisms 
		\begin{align}\label{d159}
		\pi_A!=\pi_U!\circ\pi_Y!: K_G^*(V)\rightarrow K_G^*(S).
		\end{align}		
	\end{thm}
%
%
	\section{Equivariant higher spectral flow and equivariant eta form}\label{s02}
	
	In this section, we extend the Melrose-Piazza spectral section to the equivariant case, introduce the equivariant version of Dai-Zhang higher spectral flow for arbitrary dimensional fibers and use them to prove the anomaly formula and the functoriality of the equivariant Bismut-Cheeger eta forms. In this section, we use the notation in Section \ref{s01}. 
	
	In Section \ref{s0201}, 
	we introduce the equivariant version of the spectral section
	and prove the main properties.
	In Section \ref{s0209}, we complete the proof of the well-definedness of the push-forward map in Theorem \ref{d157}.
	In Section 2.3, 
	we define the higher 
	spectral flow for fibrations with even-dimensional fibers and extend 
	 the higher spectral flow to the equivariant 
	case. Moreover, we prove that the equivariant K-group could be 
	generated by the equivariant higher spectral flows. In Section \ref{s0202},
	we explain the family local index theorem. In Section \ref{s0203},
	we define the equivariant eta form associated with a perturbation
	operator. In Section \ref{s0204}, we prove the anomaly formula 
	of equivariant eta forms in odd case. In Sections \ref{s0205}-\ref{s0207},
	we prove the functoriality of equivariant eta forms and use it to 
	prove the anomaly formula in even case.
	
	\subsection{Equivariant spectral section}\label{s0201}
	
	In this subsection, we extend the spectral section
	of Melrose-Piazza \cite{MP97a,MP97b} and the main properties of 
	them to the equivariant case.
		
	\begin{defn}(Compare with \cite[Definition 1.6]{DZ98})\label{d199}
			Let $\mF\in \mathrm{F}_G^*(B)$ and
			at least one component of the fiber has nonzero dimension.
			An \textbf{equivariant $B$-family} on $\mF$ is a smooth family of 
			self-adjoint pseudodifferential operators
			$D=\{D_b \}_{b\in B}$ on the fibers 
			of $\mF$, which commutes with the $G$-action
			and is first order on nonzero dimensional fibers, such that
%
			
			(a) it preserves the $\Z_2$-grading of $E$ when the fiber is odd dimensional;
			
			(b) it anti-commutes with the $\Z_2$-grading of $\mS_Z\widehat{\otimes}E$ when the fiber is even dimensional.
			
	If the dimension of the fiber is zero,
	an equivariant $B$-family is a
	self-adjoint endomorphism of $E$ which
	commutes with the $G$-action and
	anti-commutes with the $\Z_2$-grading of $E$.

	\end{defn}
	
	
	 	
	 If the principal symbol of $D_b$ is the same as that of
	 the fiberwise Dirac operator $D(\mF)|_{Z_b}$
	 for any $b\in B$, we call this equivariant $B$-family $D$ a $B$-family
	 of equivariant Dirac type operator.
	 In this case,
	 we have $\ind(D)=\ind(D(\mF))\in K_G^*(B)$.
	 Recall that if the fiber is a point, the fiberwise
	 Dirac operator is zero.
		
	\begin{defn}\label{d003}(Compare with \cite[Definition 1]{MP97a} and \cite[Definition 1]{MP97b})
		An equivariant Melrose-Piazza spectral section of
		 an equivariant $B$-family
		$D=\{D_b \}_{b\in B}$ is a continuous family of self-adjoint pseudodifferential
		projections $P_b$ on the $L^2$-completion of the domain of $D_b$,
%
		which commutes with the $G$-action,
		such that 
		
		(a) for some smooth function $f: B\rightarrow \R$ (depending on $P$)
		and every $b\in B$,
		\begin{align}\label{d004}
		D_b u=\lambda u\Longrightarrow
		\left\{
		\begin{aligned}
		&P_bu=u,  &\hbox{if $\lambda> f(b)$,} \\
		&P_bu=0,  &\hbox{if $\lambda<-f(b)$;}
		\end{aligned}
		\right.
		\end{align}
				
		(b) when the fiber is odd dimensional, $P$ commutes with the $\Z_2$-grading of $E$;
		
		(c) when the fiber is even dimensional,
		\begin{align}\label{d005}
		\tau^{\mS_Z\widehat{\otimes}E}\circ P+P\circ\tau^{\mS_Z\widehat{\otimes}E}=\tau^{\mS_Z\widehat{\otimes}E},
		\end{align}
		where $\tau^{\mS_Z\widehat{\otimes}E}$ is the $\Z_2$-grading of $\mS_Z\widehat{\otimes}E$.
	\end{defn}

The following proposition is the equivariant extension of the results in \cite{MP97a,MP97b}. Remark that in our setting the dimension of the fiber
might be zero. In the proof of this proposition, we will use the equivariant
version of the Fredholm theory for fiberwise elliptic operators. For the 
references, see \cite{ASe04} and \cite[Appendix A.5]{FHT11}.
We will also show some details in Appendix B.
	
	\begin{prop}\label{d006}
		Let $\mF\in\mathrm{F}_G^*(B)$ and $D$ be an equivariant $B$-family on $\mF$. 
		
		(i) (Compare with \cite[Proposition 1]{MP97a} and \cite[Proposition 2]{MP97b})
		If there exists an equivariant spectral section of $D$ on $\mF\in \mathrm{F}_G^0(B)$ (resp. $\mathrm{F}_G^1(B)$), then  $\ind (D)=0 \in K_G^0(B)$ (resp. $K_G^1(B)$). Conversely, if $\mF\in \mathrm{F}_G^0(B)$ (resp. $\mathrm{F}_G^1(B)$), $\ind (D)=0 \in K_G^0(B)$ (resp. $K_G^1(B)$) and at least one component of the fibers has the nonzero dimension,  there exists an equivariant spectral section of $D$.
		
		(ii) (Compare with \cite[Proposition 2]{MP97a})
		For $\mathcal{F}\in \mathrm{F}_G^1(B)$,
		given equivariant spectral sections $P$, $Q$ of $D$, there exists an equivariant spectral section $R$ of $D$ such
		that $PR=R$ and $QR=R$. 
		We say that $R$ majors $P$, $Q$.
				
		(iii) (Compare with \cite[Lemma 8]{MP97a} and \cite[Lemma 1]{MP97b})
		If there exists an equivariant spectral section $P$ of $D$,
		then there exists a family of self-adjoint equivariant smoothing operators $A_{P}$
		(when the dimension of the fibers are zero, it descends to a
		self-adjoint equivariant endmorphism
		of the complex vector bundle) with range in a finite
		sum of eigenspaces of $D$ such that $D+A_{P}$
		is an \textbf{invertible} 
		equivariant $B$-family and $P$ is the Atiyah-Patodi-Singer (APS) projection onto the eigenspaces of the positive spectrum of $D+A_P$.

	\end{prop}
	\begin{proof}
		
		Case 1: Let $\mathcal{F}\in \mathrm{F}_G^1(B)$.
We use the notation of Appendix B in this part of the proof.

We write $\mF=(W,o,E)$, $E=E_+\oplus E_-$ and omit other data for simplicity. 
Set $\mF_{+}=(W,o,E_{+}\oplus \{0\})$ and $\mF_-=(W,o, \{0\}\oplus E_-)$. Then 
$\ind(D(\mF))=\ind(D(\mF_++\mF_-))=\ind(D(\mF_+))+\ind(D(\mF_-))$.
Set $\mF^{\mathrm{re}}_-:=(W,-o, E_-\oplus \{0\})$, where $-o$ is the reversion of the 
orientation $o$ and $E_-$ is equipped with the positive $\Z_2$-grading. 
Then there is a natural bijective map
$\mF_-\rightarrow \mF_-^{\mathrm{re}}$ by identification.
Note that $\ind(D(\mF_-))=\ind(D(\mF_-^{\mathrm{re}}))$.
We have $\ind(D(\mF))=\ind(D(\mF_++ \mF_-^{\mathrm{re}}))$.
So using this bijective map, we only need to prove our proposition
in odd case
when $E_-=0$ in $\mF$.

		(i) Let $T=D/(1+D^2)^{1/2}$. Then $T$ is bounded,  $G$-equivariant
				 and $\ind(T)=\ind(D)\in K_G^1(B)$. 
		As in Appendix B, $\sqrt{-1}T$ can be extended to an equivariant map from $B$ to
		 $\mathrm{Fred}^1(L^2(G)\otimes H)$, where $H$ is a separable Hilbert space.
		Assume that there exists an equivariant spectral section $P$ of $D$.
		It could be extended on $L^2(G)\otimes C(\R)\otimes H$
		in the same way as $T$, which is also denoted by $P$.
		From the definition of the equivariant spectral section,
	    $PT(1-P)$ and $(1-P)TP$ are $G$-equivariant self-adjoint 
		finite rank operators. Let $K=(1+D^2)^{-1/2}$ on $\mE$. Then $K$ is a $G$-equivariant
		self-adjoint compact operator on $L^2(G)\otimes C(\R)\otimes H$
		by taking zero on the complement of $\mE$.
		Thus from the $G$-homotopy invariance, the equivariant family index of $T$ in $K_G^1(B)$ is the same as
		that of
		$P(T+rK)P+(1-P)(T-rK)(1-P)$ for $r>0$. When $r$ is large enough, for any $b\in B$, $P_b(T_b+rK_b)P_b$ is positive and $(1-P_b)(T_b-rK_b)(1-P_b)$ is negative.
		Therefore we have $\ind(T)=0\in K_G^1(B)$.
		
		If $\ind(D)=\ind(T)=0\in K_G^1(B)$, we modify the construction
		of the spectral section in the proof of
		 \cite[Proposition 1]{MP97a}. 
		All equivariant operators here are regarded as 
		families of equivariant operators
		acting on $L^2(G)\otimes H$.
		Since $\ind(T)=0\in K_G^1(B)$, from (\ref{eq:2.03}),
		$\sqrt{-1}T$ is $G$-homotopic to an invertible element
		in $\mathrm{Fred}^1(L^2(G)\otimes H)$ through 
		$\sqrt{-1}T_t\in \mathrm{Fred}^1(L^2(G)\otimes H)$, $t\in [0,1]$.
		As in \cite{MP97a}, all operators in these families have discrete
		spectrum in some fixed open interval $(-\var,\var)$, $0<\var<1$.
		Choose $\chi\in \cC^{\infty}(\R)$ with $\chi(\lambda)=0$ if
		$\lambda<0$ and $\chi(\lambda)=1$ if $\lambda>\var/2$.
		Set $J=\chi(T)$. Then $J$ is $G$-equivariant and smooth on
		$b\in B$. From the $G$-homotopy above, we could construct a smooth
		family of $G$-equivariant projections $P'$ on $L^2(G)\otimes H$
		in the same way as in the proof of \cite[Proposition 1]{MP97a}
		such that $J-P'$ has finite rank and
		the range of  $J-P'$ lies in $\mE$.
		By taking spectral cuts as in the proof of \cite[Theorem 3]{LP03},
		we could obtain an equivariant projection $P$ which differs
		from $J$ by an equivariant operator whose range lies
		 in the span of a finite
		number of eigenfunctions of $T$ on $\mE$ for each $b\in B$
		(see also the proof of \cite[Proposition 3.7]{Z05}).
		So $P|_{\mE}$ is an equivariant spectral section.

%
%
		
		(ii) We extend the equivariant spectral section $P$ on
		the equivariant trivial Hilbert bundle as before and
		 consider the family of operators $PTP$ on the range of $P$. 
		These are equivariant self-adjoint operators and from (\ref{d004}),
		there exists $N>0$, such that all but the first $N$ eigenfunctions of $PTP$ are eigenfunctions of $D$. Since $B$ is compact,
		we could take $0<a_1<1$ such that the first $N$ eigenvalues
		of $PTP$ are all less than $a_1$ for any $b\in B$.	
		Take $a_2\in (a_1,1)$ 
		and choose $\chi_1\in \cC^{\infty}(\R)$ with $\chi_1(\lambda)=0$ if
		$\lambda\leq a_1$ and $\chi_1(\lambda)=1$ if $\lambda\geq a_2$.
Then for $M$ large enough, the range of $P-\chi_1(T)$
is an equivariant subbundle of the range of $P$
(cf. \cite[Lemma 9.9]{BGV04}) 
 such that it contains the first $N$ eigenfunctions and is 
 contained in the span of the first $M$ eigenfunctions
  of $PTP$. 
  Let $R$ be the orthogonal projection on the complement of this subbundle
  in $\mE$. Then $R$ is an equivariant spectral section such that $PR=R$.
		If the integer $N$ is chosen large enough, then the projection $R$ will have range contained in the intersection of the ranges of any two given equivariant spectral sections $P$ and $Q$. So $QR=R$. 
				
		(iii) Let $\mathcal{P}_{\lambda\in [a_1, a_2],b}(D_b)$ be the span of the eigenfunctions corresponding to the eigenvalues $\lambda\in [a_1,a_2]$ of $D_b$.
		Since $B$ is compact, we can choose $s>0$, such that $P$ is an equivariant spectral section for $f(b)\equiv s$. By the proof of (ii), we can choose equivariant spectral sections $R'$, $R''$, such that for any $b\in B$, $R'_b=0$ on $\mathcal{P}_{\lambda\leq s,b}(D_b)$ and
		$R''_b=I$ on $\mathcal{P}_{\lambda\geq -s,b}(D_b)$. Then the operator
		\begin{align}\label{d173}
		\widetilde{D}=R'DR'+sPR''(I-R')+(I-R'')D(I-R'')-s(I-P)R''(I-R')
		\end{align}
		is an invertible equivariant $B$-family
		(cf. \cite[(8.3)]{MP97a}). Then $A_P=\widetilde{D}-D$
		satisfies all conditions.
		
		\
		
		Case 2: Let $\mathcal{F}\in \mathrm{F}_G^0(B)$ and
		at least one component of the fibers has nonzero dimension.
		
		Let 
		$D_{\pm}:=
		D|_{(\mS_Z\widehat{\otimes}E)_{\pm}}$. Let $S$ be a 
		first order positive equivariant elliptic pseudodifferential
		operator. Then in the sense of (\ref{eq:2.03}),
		 $D$ is $G$-homotopic to 
		$
		\left(\begin{array}{cc}
		S & D_- \\
		D_+ & -S
		\end{array}
		\right),
		$
		which is invertible. 
		Thus the equivariant $K^1$-index of the whole self-adjoint family $D$ vanishes. By the same process in the proof of (i) in the odd case, there exists an equivariant spectral section $P'$ in the odd sense, which means that it is an equivariant spectral section without the condition (\ref{d005}).
		
		(iii)
		By the proof of (ii) for the odd case, we could choose $P'$, which is an equivariant spectral section in the odd sense, such that $P'DP'$ is positive on the range of $P$. 
		We simply denote by $\tau=\tau^{\mS_Z\widehat{\otimes}E}$.
		 Then the operator
		\begin{align}
		\label{d107}
		A_P=P-P'-\tau(P-P')\tau+P'DP'+\tau P'\tau D \tau P'\tau -D
		\end{align}
		satisfies all conditions (cf. \cite[(2.11), (2.12)]{MP97b}).

		(i) Assume that $\ind(D)=0\in K_G^0(B)$. As in the proof of (ii) for the odd case, for $r>0$ fixed, we can choose an equivariant spectral section $P'$ in the odd sense such that $P'=0$ on $\mathcal{P}_{\lambda\leq r}(D)$. From Definition \ref{d199} (b), we have $\tau P'\tau=0$ on $\mathcal{P}_{\lambda\geq -r}(D)$. Let $V=\ker(P'+\tau P'\tau)$. Then $V$ is a finite dimensional equivariant complex vector bundle over $B$. We split the complex vector bundle by $V=V_+\oplus V_-$ with respect to $\tau$. Then $\ind(D)=[V_+]-[V_-]\in K_G^0(B)$. The assumption $\ind(D)=0\in K^0(B)$ implies that there exists a complex vector bundle $U$ such that $V_{+}\oplus U\simeq V_{-}\oplus U$ as complex vector bundles.
		
		We choose another equivariant spectral section $P''$ in the odd sense such that the range of $P'-P''$ is an equivariant complex vector bundle whose rank is large enough.
		Let $V'=\ker(P''+\tau P''\tau )$ and $V'=V'_+\oplus V'_-$ with respect 
		to $\tau$.
		Let $W_{\pm}$ be the complex vector bundles such that $V_{\pm}'=V_{\pm}\oplus W_{\pm}$. Then $D_+$ induces an isomorphism between $W_+$ and $W_-$.
		Since the rank of $W_{\pm}$ is large enough,
		there exist subbundles $U_{+}\subset W_{+}$ and $U_{-}\subset W_{-}$ such that $U_{+}\simeq U_{-}\simeq U_{}$ as complex vector bundles and $D_+(U_+)=U_-$. So $V'_+\simeq V'_-$ as complex vector bundles.
		Since $\ind(D)=0\in K_G^0(B)$, this isomorphism is $G$-equivariant. 
Let
		\begin{align}\label{eq:2.04}
		P_V=\frac{1}{2}
		\left(\begin{array}{cc}
		1 & \sqrt{-1} \\
		-\sqrt{-1} & 1
		\end{array}
		\right)
		\end{align}
		on $V'=V_+'\oplus V_-'$. Then $P=P''+P_V$
		is an equivariant spectral section.
		
		The other direction follows from (iii) easily.
		
		\
		
		Case 3:
		The dimensions of all the fibers are zero.
		
		In this case, for any self-adjoint projection $P\in \End(E)$ 
		commuting with the $G$-action,
		(\ref{d004}) holds. Thus the only restriction for $P$ as an equivariant 
		spectral section is (\ref{d005}). 
		If there exists an equivariant spectral section $P$, 
		we take 
		$A_P=P-\tau P\tau -D$. Thus $D+A_P=2P-\Id$ is invertible 
		and $P$ is the projection onto the eigenspaces of the positive 
		eigenvalues of $D+A_P$. Thus $\ind(D)=0\in K_G^0(B)$.
				
		The proof of Proposition \ref{d006} is completed.
				
	\end{proof}
	
	\begin{rem}	\label{d134}
		In zero dimensional case, we could also
		construct an equivariant spectral section
		of $D(\mF+\mF^{\mathrm{op}})$ as in (\ref{eq:2.04}).
	\end{rem}

\begin{defn}\label{d179}
	Let $D$ be an equivariant $B$-family on $\mF$. A perturbation operator with respect to $D$ is a family of bounded pseudodifferential operators $A$ such that $D+A$ is an \textbf{invertible} equivariant $B$-family on $\mF$.
\end{defn}

Note that if there exists an equivariant spectral section of $D$, the smoothing operator associated with it is a perturbation operator.

Remark that the tamings in \cite{Bu09,BS09,BS13} are perturbation operators when the manifolds there are smooth, compact and without boundary.

\subsection{Well-defined property for the push-forward map}\label{s0209}

In this subsection, we show that the push-forward map defined in Theorem \ref{d157} is well-defined. We use the notation in Section 1.4.

\begin{lemma}\label{d089}
	If $\ind(D(\mF_X))=0\in K_G^{j}(V)$, then $\ind(D(\mF_Z))=0\in K_G^{N(j)}(B)$, $j=0,1$. 	
\end{lemma}
\begin{proof}
	We only need to prove this lemma when the dimensions of the fibers
	are nonzero.		
	Let
	\begin{align}\label{d035}
	g_T^{TZ}=\pi_X^* g^{TY}\oplus \frac{1}{T^2}g^{TX}.
	\end{align}
	We denote by $C_T(TZ)$ the Clifford algebra bundle of $TZ$ with respect to $g_T^{TZ}$.
	If $U\in TV$, let $U^{H}\in T_{\pi_X}^HW$
	be the horizontal lift of $U$, such that $\pi_{X,*}(U^{H})=U$.
	Let $\{e_i\}$, $\{f_{p}\}$ be local orthonormal frames of $(TX, g^{TX})$, $(TY, g^{TY})$.
	Then $\{f_{p,1}^H\}\cup\{Te_i\}$ is a local orthonormal frame of $(TZ, g_T^{TZ})$.
	We define a Clifford algebra isomorphism
	\begin{align}\label{e02027}
	\mG_T:C_T(TZ)\rightarrow C(TZ)
	\end{align}
	by
	\begin{align}\label{e02028}
	\begin{split}
	\mG_T(c(f_{p,1}^H))=c(f_{p,1}^H),\quad \mG_T(c_T(Te_i))=c(e_i).
	\end{split}
	\end{align}
	Under this isomorphism, we can consider $((\pi_X^*\mS_Y\widehat{\otimes}\mS_X)\widehat{\otimes} E, h^{\pi_X^*\mS_Y\widehat{\otimes} \mS_X}\otimes h^{E})$
	as a self-adjoint Hermitian equivariant Clifford module of $C_T(TZ)$. 
	So
	\begin{align}\label{d028}
	\mF_{Z,T}=(W, L_Z, E, o_Z, T_{\pi_Z}^HW, g_T^{TZ},
	h^{L_Z}, \nabla^{L_Z}, h^E, \nabla^{E})
	\end{align}
	is an equivariant geometric family over $B$ and $\mF_{Z,1}=\mF_Z$
	in (\ref{d073}).
	
	If $\ind(D(\mF_{X}))=0\in K_G^*(V)$, from Proposition \ref{d006} (i),
	there exists a perturbation operator $A_X$ such that 
	$\ker (D(\mF_{X})+A_{X})=0$. 
	We extend $A_X$ to a pseudodifferential operator acting on $\cC^{\infty}(W,(\pi_X^*\mS_Y\widehat{\otimes}\mS_X)\widehat{\otimes} E)$ the same way as the extension of $c(e_i)$ in Section 1.1, denoted
	by $1\widehat{\otimes}A_X$.
	From the proof of \cite[Lemma 5.3]{Liu17},
	there exists $T'\geq 1$, such that when $T\geq T'$, $\ker (D(\mF_{Z,T})+1\widehat{\otimes}TA_{X})=0$.
	So by the homotopy invariance of the equivariant family index, for any $T\geq 1$, we have  $\ind(D(\mF_{Z,T}))=0$.
	
	
	The proof of Lemma \ref{d089} is completed.
\end{proof}

\subsection{Equivariant higher spectral flow}

In \cite{DZ98}, Dai and Zhang introduced the higher spectral flow
for odd dimensional fibers. In this subsection, we extend the
Dai-Zhang higher spectral flow to the equivariant case and 
define the equivariant higher spectral flow for even dimensional 
fibers inspired by \cite[Proposition 4]{MP97b}.

	Note that a horizontal subbundle on $W$ is simply a splitting of the exact sequence
\begin{align}\label{e01145}
0\rightarrow TZ\rightarrow TW \rightarrow \pi^*TB \rightarrow 0.
\end{align}
As the space of the splitting map is affine, 
since the $G$-action preserves (\ref{e01010}),
it follows that any pair of equivariant horizontal subbundles
can be connected by a smooth path of equivariant horizontal distributions. 

Assume that $\mF,\mF'\in \mathrm{F}_G^*(B)$ have the same topological structure, i.e., they satisfy the first three conditions 
in Definition \ref{d125}. Let $r\in I$, $I=[0,1]$, parametrize 
a smooth path of equivariant horizontal subbundles
$\{T_{\pi,r}^HW\}_{r\in [0,1]}$ such that $T_{\pi,0}^HW=T_{\pi}^HW$ and $T_{\pi,1}^HW=T_{\pi}^{'H}W$.
Let $g_r^{TZ}$, $h_r^{L_Z}$ and $h_r^{E}$
be the $G$-invariant metrics on $TZ$, $L_Z$ and $E$, depending smoothly on $r\in I$, which coincide with $g^{TZ}$, $h^{L_Z}$ and $h^{E}$
at $r=0$ and with $g^{'TZ}$, $h^{'L_Z}$ and $h^{'E}$ at $r=1$. By the same reason, we can choose $G$-invariant Hermitian connection $\nabla_r^{L_Z}$ and $\nabla_r^{E}$
on $L_Z$ and $E$,
such that $\nabla_0^{E}=\nabla^{E}$,
$\nabla_1^{E}=\nabla^{'E}$, $\nabla_0^{L_Z}=\nabla^{L_Z}$,
$\nabla_1^{L_Z}=\nabla^{'L_Z}$. 

Let $\widetilde{B}=B\times I$.
We consider the bundle $\widetilde{\pi}:\widetilde{W}:= W\times I\rightarrow\widetilde{B}$ together with the natural projection
$\mathrm{Pr}:\widetilde{W}\rightarrow W$.
Then the fiberwise $G$-action can be naturally extended to $\widetilde{\pi}:\widetilde{W}\rightarrow\widetilde{B}$
such that $G$ acts as identity on $I$.
Thus $T_{\widetilde{\pi}}^H\widetilde{W}_{(r, \cdot)}=\R\times T_{\pi,r}^HW$ defines an equivariant horizontal subbundle of $T\widetilde{W}$,
and $T\widetilde{Z}:=\mathrm{Pr}^*TZ$, $\widetilde{L_Z}:=\mathrm{Pr}^*{L_Z}$ and $\widetilde{E}:=\mathrm{Pr}^*{E}$ are naturally equipped with $G$-invariant
metrics $g^{T\widetilde{Z}}$, $h^{\widetilde{L_Z}}$, $h^{\widetilde{E}}$ and $G$-invariant Hermitian connections $\nabla^{\widetilde{L_Z}}$, $\nabla^{\widetilde{E}}$.
Let $\tilde{o}=\mathrm{Pr}^*o$.
Then we obtain equivariant geometric families
\begin{align}\label{eq:2.10}
\mF_r=(W, L_Z, E, o, T_{\pi,r}^{H}W, g_r^{TZ}, h_r^{L_Z}, \nabla_r^{L_Z}, h_r^{E}, \nabla_r^{E})
\end{align}
over $B$ and
\begin{align}\label{d212}
\widetilde{\mF}=(\widetilde{W}, \widetilde{L_Z}, \widetilde{E}, \tilde{o}, T_{\widetilde{\pi}}^{H}\widetilde{W}, g^{T\widetilde{Z}}, h^{\widetilde{L_Z}}, \nabla^{\widetilde{L_Z}}, h^{\widetilde{E}}, \nabla^{\widetilde{E}})
\end{align}
over $\wi{B}$ such that $\mF_0=\mF$ and $\mF_1=\mF'$.

		If $\mF\in\mathrm{F}_G^1(B)$ and $R$, $P$ are two equivariant spectral sections of an equivariant $B$-family $D$ such that $PR=R$, then 
	the cokernel of  $P_bR_b: \Im(R_b)\rightarrow \Im(P_b)$ for $b\in B$
	forms an equivariant complex vector bundle over $B$, denoted by $[P-R]$. Hence for any two equivariant spectral
	sections $P$, $Q$, the difference element $[P-Q]$ can be defined as an element
	in $K_G^0(B)$ as follows:
	\begin{align}\label{d009}
	[P-Q]:=[P-R]-[Q-R]\in K_G^0(B),
	\end{align}
	where $R$ is an equivariant spectral section which majors $P$, $Q$ as in Proposition \ref{d006} (ii) . 	
	From (\ref{d009}), we can obtain that if $P_1$, $P_2$, $P_3$ are equivariant spectral sections of $D$, then
	\begin{align}\label{d011}
	[P_3-P_1]=[P_3-P_2]+[P_2-P_1]\in K_G^0(B).
	\end{align}
	Thus the class in (\ref{d009}) is independent of the choice
	of $R$.

Let $\mF,\mF'\in \mathrm{F}_G^1(B)$ which have the same topological structure. Let $\mF_r$ and $\wi{\mF}$
are equivariant geometric families in (\ref{eq:2.10}) and (\ref{d212}).
	Now we consider a continuous family of operators $D_r$ on $\mF_r$ for $r\in I$ such that $D_r$ is an equivariant $B$-family on $\mF_r$. Assume that $\ind(D_0)=0\in K_G^1(B)$.
	Then the homotopy invariance of the equivariant family index implies that 
	the equivariant indice of $D_r$ vanish.
	Let $Q_0$ and $Q_1$ be equivariant spectral sections of $D_0$
	and $D_1$ respectively.
	If we consider the
	total family $\widetilde{D} =\{D_r\}$ parametrized by $B\times I$, then there exists a total equivariant spectral
	section $\widetilde{P}$. Let $P_r$ be the restriction of $\widetilde{P}$ over $B\times\{r\}$. Thus we have the natural equivariant extension of
	the higher spectral flow in \cite[Definition 1.5]{DZ98}.
	
	\begin{defn}\label{d152}
		The equivariant Dai-Zhang higher spectral flow $\mathrm{sf}_G\{(D_0, Q_0),$ $ (D_1, Q_1)\}$ between the pairs
		$(D_0, Q_0)$, $(D_1, Q_1)$ is an element in $K_G^0(B)$ defined by
		\begin{align}\label{d103}
		\mathrm{sf}_G\{(D_0, Q_0), (D_1, Q_1)\}=[Q_1-P_1]-[Q_0-P_0]\in K_G^0(B).
		\end{align}
	\end{defn}
	
	From (\ref{d011}), we know that this definition is independent of
	the choice of the total equivariant spectral section
	$\widetilde{P}$. 
	
	\
		
	In the following, we define the equivariant higher spectral flow for the even case.
		
		Let $\mF\in \mathrm{F}_G^0(B)$.
		Let $D$ be an equivariant $B$-family on $\mF$. \textbf{We assume that there exists an equivariant spectral section $P$ with respect to $D$.}
		Let $A_P$ be the family of self-adjoint equivariant smoothing operators associated with $P$ by Proposition \ref{d006} (iii).
						
		Now we use the notation in Example \ref{d129} d). 
		Let $p_1^*\mF\times_{B\times S^1} p_2^*\mF^L$ be the odd equivariant geometric family  in Example \ref{d129} d) with fibers $Z\times S^1$. 		
		Let $\tau$ be the $\Z_2$-grading of the $\mS_Z\widehat{\otimes}E$ in $\mF$.
%
We consider the vector bundle part in $p_1^*\mF\times_{B\times S^1} p_2^*\mF^L$
as an ungraded one.
		Then from Definition \ref{d199}, 
		\begin{align}\label{eq:2.11}
		D_P=(D+A_P)\otimes 1+\tau\otimes D(\mF^L)
		\end{align}
		is an equivariant $B\times S^1$-family on the odd geometric family $p_1^*\mF\times_{B\times S^1} p_2^*\mF^L $ and commutes with the group action.
		
				
		Since $D$ and $A_P$ anti-commute with $\tau$,  
		\begin{align}\label{eq:2.12}
		D_P^2=((D+A_P)\otimes 1+ \tau\otimes D(\mF^L))^2=(D+A_P)^2\otimes 1+1\otimes D(\mF^L)^2>0.
		\end{align}
		It implies that $D_P$ is invertible. 
		Thus the APS projection $P'$ is an equivariant spectral section of $D_P$. 			 
		 Similarly, let $Q$ be another equivariant spectral section of $D$, we can construct the equivariant spectral section $Q'$ of $D_Q$ as above.
		Since $p_1^*\mF\times_{B\times S^1} p_2^*\mF^L\in \mathrm{F}_G^1(B)$, from Definition \ref{d152}, we could define $\mathrm{sf}_G\{(D_P,P'), (D_Q, Q')\}\in K_G^0(B\times S^1)$.
		
		Now we  consider Example \ref{d129} c) more explicitly. It is easy to calculate that for $\theta\in[0,1)$ fixed, the eigenvalues of $D(\mF^L)$ are $\lambda_k(\theta)=2\pi k+2\pi(\theta-1/2),\ k\in \Z$.
		So for $\theta\in [0,1)$, $\theta\neq 1/2$, we have $D(\mF^L)^2>0$. 
		Thus as in (\ref{eq:2.12}), for any $s\in [0,1]$, $\theta\neq 1/2$, restricted on $B\times \{\theta\}$,  $(1-s)D_P+sD_Q$ is invertible. From Definition \ref{d152}, it means that for $\theta\neq 1/2$, $\mathrm{sf}_G\{(D_P,P'), (D_Q, Q')\}|_{B\times \{\theta\}}=0\in K_G^0(B\times \{\theta\})$. From (\ref{d124}), there exists
		an element in $K_G^1(B)$, which we denote by $[Q-P]$, such that 
		\begin{align}\label{d138}
		j([Q-P])=\mathrm{sf}_G\{(D_P,P'), (D_Q, Q')\}\in K_G^0( B\times S^1).
		\end{align}
		
		The idea for this construction comes from \cite[Proposition 4]{MP97b}.
		We note that when the group $G$ is trivial, this definition is equivalent to that there.
	
	Similarly, if $P_1$, $P_2$, $P_3$ are equivariant spectral sections of $D$, then
	\begin{align}\label{d012}
	[P_3-P_1]=[P_3-P_2]+[P_2-P_1]\in K_G^1( B).
	\end{align}
		
	Now we extend the difference $[Q-P]$ to the equivariant higher spectral flow.
	Let $\mF,\mF'\in \mathrm{F}_G^0(B)$, which have the same topological structure, and $D_0$, $D_1$ be two equivariant $B$-families on $\mF$, $\mF'$ respectively. For $i=0,1$,
	let $Q_i$ be an equivariant spectral section of $D_i$
	with corresponding smoothing operators $A_{Q_i}$. Let $D(r)$, $r\in [0,1]$ be a continuous curve of equivariant $B$-families on $\mF_r$ such that $D(i)=D_i+A_{Q_i}$, $i=0,1$.
	Let 
	\begin{align}\label{eq:2.15}
	D_{i,Q_i}=(D_i+A_{Q_i})\otimes 1+\tau\otimes D(\mF^L).
	\end{align}
	By (\ref{eq:2.12}), they are invertible. Let $Q_i'$ be their
	APS projections. 	
	 Let $\widetilde{D} =\{D(r)\otimes 1+\tau\otimes D(\mF^L)\}$ parametrized by $B\times S^1\times I$. 
	 By (\ref{eq:2.12}), $D_{0,Q_0}$ is invertible. Thus
	  $\ind(D_{0,Q_0})=0\in K_G^1(B\times S^1)$. So $\ind(\widetilde{D})=0\in K_G^1(B\times S^1\times I)$.
	Let $\widetilde{P} =\{P(r)\}_{r\in [0,1]}$ be an equivariant spectral
	section with respect to $\widetilde{D}$ such that $P(i)$ majors $Q_i'$
	for $i=0,1$.
	Then from Definition \ref{d152}, 
	\begin{align}\label{eq:2.16}
	\mathrm{sf}_G\{(D_{0,Q_0}, Q_0'), (D_{1,Q_1}, Q_1')\}=[Q_1'-P(1)]-[Q_0'-P(0)]\in K_G^0(B\times S^1).
	\end{align}
	
	Furthermore, we could obtain that this equivariant higher
	spectral flow lies in the image of $j$ in (\ref{d124}).
	In fact, when restricted on $B\times \{\theta\}\times I$ for $\theta\neq 1/2$, as in (\ref{eq:2.12}),  $\widetilde{D}|_{B\times \{\theta\}\times I}$ is invertible. For $\theta\neq 1/2$, let $\{P'(r)_{\theta}\}_{r\in [0,1]}$ be the APS projection of $\widetilde{D}|_{B\times \{\theta\}\times I}$. Then $P'(0)_{\theta}=Q_0'|_{B\times \{\theta\}}$ and $P'(1)_{\theta}=Q_1'|_{B\times \{\theta\}}$. Since 
	$P'(r)_{\theta}$ and $P(r)|_{B\times \{\theta\}\times I}$ are two equivariant
	spectral sections of $\widetilde{D}|_{B\times \{\theta\}\times I}$ and $P(i)|_{B\times \{\theta\}}$ majors $Q_i'|_{B\times \{\theta\}}=P'(i)_{\theta}$ for $i=0,1$, we see that
	$[P'(r)_{\theta}-P(r)|_{B\times \{\theta\}\times I}]$ forms an equivariant complex vector bundle over $B\times \{\theta\}\times I$. Thus
	we have $([Q_1'-P(1)]-[Q_0'-P(0)])|_{B\times \{\theta\}}=0\in K_G^0(B\times \{\theta\})$. It implies that $\mathrm{sf}_G\{(D_{0,Q_0}, Q_0'), (D_{1,Q_1}, Q_1')\}\in \Im (j)$.
	
	\begin{defn}\label{d15}
	If	$\mF,\mF'\in \mathrm{F}_G^0(B)$,
the equivariant higher spectral flow $\mathrm{sf}_G\{(D_0, Q_0), (D_1, Q_1)\}$ between the pairs
		$(D_0, Q_0)$, $(D_1, Q_1)$ is an element in $K_G^1(B)$ defined by
		\begin{align}\label{d174}
		j\big(\mathrm{sf}_G\{(D_0, Q_0), (D_1, Q_1)\}\big)=\mathrm{sf}_G\{(D_{0,Q_0}, Q_0'), (D_{1,Q_1}, Q_1')\}.
		\end{align}
	\end{defn}
	
	Note that when $\mF=\mF'$, $D_0=D_1=D$, the equivariant higher spectral flow $\mathrm{sf}_G\{(D, Q_0), (D, Q_1)\} =[Q_1-Q_0]$.
	
	The following proposition says that any element of equivariant K-group
	could be generated by equivariant higher spectral flows. Our proof is constructive.
	
	\begin{prop}\label{d123}
		(i) For any $x\in K_G^0(B)$, there exist $\mF_1, \mF_2\in 
		\mathrm{F}_G^1(B)$ and equivariant spectral sections $P_i$, $Q_i$ with respect to $D(\mF_i)$ for $i=1,2$, such that  $x=[P_1-Q_1]-[P_2-Q_2]$.
		
(ii) For any $x\in K_G^1(B)$,
		there exist $\mF\in \mathrm{F}_G^0(B)$ and equivariant spectral sections $P$, $Q$ with respect to $D(\mF)$, such that  $x=[P-Q]$.
	\end{prop}
	\begin{proof}
		Let $(E,h^E)$ be a Hermitian vector bundle and $\nabla^{E}$
		be a Hermitian connection on $(E, h^{E})$.
		Let $\pi:B\times S^1\rightarrow B$ be the projection onto the first part. Let $\mF=(B\times S^1, \pi^*E, o, T^H(B\times S^1), g^{TS^1}, \pi^*h^{E}, \pi^*\nabla^{E})\in \mathrm{F}_G^1(B)$, where $o$, $g^{TS^1}$ are the canonical orientation and metric on $S^1$ and $T^H(B\times S^1)=TB\times S^1$. Let $\partial_t$ be the generator of $TS^1$. Then $D(\mF)=-\sqrt{-1}\partial_t\otimes \Id_{E}$. We could calculate that the eigenvalues of $D(\mF)$ are $\lambda_k=k$ for $k\in\Z$. We denote by $P_{\lambda\geq k}$ the orthogonal projection onto the union of the eigenspaces of $\lambda\geq k$. Then for any $k$, $P_{\lambda\geq k}$ is an equivariant spectral section of $D(\mF)$. In particular, we have $[P_{\lambda\geq k}-P_{\lambda\geq k+1}]=[E]\in K_G^0(B)$. 
		Thus we obtain Proposition \ref{d123} in the even case.
		
		For any $x\in K_G^1(B)$, from Lemma \ref{lemma:1.04}, there exists a finite dimensional complex unitary representation $V$ of $G$, such that $x$ can be represented as a $G$-invariant unitary element  $F\in \cC^{\infty}(B,\End(V))$. Let $\mF_1=(B, E_+=E_-=B\times V)\in \mathrm{F}_G^0(B)$, with fiber $Z=\mathrm{pt}$ and trivial metric and connection on $E_{\pm}$. Let  
		\begin{align*}	
			A_0=
			\left(\begin{array}{cc}
			0 & I \\
			I & 0
			\end{array}
			\right),\quad
			A_1=
			\left(\begin{array}{cc}
			0 & F(b)^* \\
			F(b) & 0
			\end{array}
			\right)
		\end{align*}
		be endomorphisms of $V\oplus V$.
		Let $P_i$ be the orthogonal projection onto the positive part of the spectrum of $A_i$ for $i=0,1$.
		It is easy to calculate that 
		for $i=0,1$, $P_i\tau+\tau P_i=\tau$.
		From Definition \ref{d003}, we know that $P_0$ and $P_1$ are equivariant spectral sections with respect to $D(\mF_1)=0$ on $\mF_1$. Let $D_i=A_i\otimes 1+\tau\otimes D(\mF^L)$ on $p_1^*\mF_1\times_{B\times S^1} p_2^*\mF^L$ and $P_i'$ be the APS projections of $D_i$. Let $D_s=(1-s)D_0+sD_1$ for $s\in [0,1]$. We claim that
		\begin{align}\label{d175}
		\mathrm{sf}_G\{(D_0, P_0'), (D_1, P_1')\}=[W]-[U]\in K_G^0(B\times S^1),
		\end{align}
		where $W$ and $U$ are bundles constructed above Lemma \ref{lemma:1.04}.
		Then from (\ref{d138}) and (\ref{d175}), we obtain Proposition \ref{d123} in the odd case.
		
		We prove the claim (\ref{d175}) constructively.
		Let $\lambda_{b,i}$ be the eigenvalues of $F(b)$ on $V$ with
		unitary eigenvectors $v_{b,i}$. Then $\ov{\lambda}_{b,i}$ are the eigenvalues of $F^*(b)$ on $V$ with the same eigenvectors. Let $v_{b,i}^{\pm}$ be the corresponding vectors in $E_{\pm,b}$. Let $v_k$ be the eigenvector of $\lambda_k(\theta)$ with respect to $D(\mF^L)$ (see Appendix B).
		From (\ref{eq:2.12}),
		it is easy to calculate that the nonnegative eigenvalues of $D_s$ are
		\begin{align}\label{eq:2.19}
		\lambda_{s,b,i,k}(\theta)=\sqrt{\lambda_k(\theta)^2+(1-2s)^2
			+s(1-s)(\lambda_{b,i}+\ov{\lambda}_{b,i}+2)}.
		\end{align}
		Since $F$ is unitary, $|\lambda_{b,i}|=1$. So 
		$\lambda_{s,b,i,k}(\theta)=0$ if and only if $k=0, \theta=\frac{1}{2},
		s=\frac{1}{2}$ and $\lambda_{b,i}=-1$.
		From (\ref{eq:2.11}), we calculate that the
		eigenfunctions of $\lambda_{s,b,i,1}(\theta)$ with respect to $D_s$
		are
		\begin{align}\label{eq:2.20}
		\begin{split}
		&u_{s,b,i}^{(1)}(\theta)=((s\ov{\lambda}_{b,i}+1-s)v_{b,i}^++
		(\lambda_{s,b,i,1}(\theta)-\lambda_1(\theta))
		v_{b,i}^-)\otimes v_1,
		\\
		& u_{s,b,i}^{(2)}(\theta)=((\lambda_{s,b,i,-1}(\theta)+\lambda_{-1}(\theta))
		v_{b,i}^++(s\lambda_{b,i}+1-s)v_{b,i}^-)\otimes v_{-1}.
		\end{split}
		\end{align}
Let	
\begin{align}\label{eq:2.21}
\begin{split}
&u_{s,b,i}^{(3)}(\theta)=((s\ov{\lambda}_{b,i}+1-s)v_{b,i}^++
(\lambda_{s,b,i,0}(\theta)-\lambda_0(\theta))
v_{b,i}^-)\otimes v_0,\quad \quad  0\leq \theta\leq 1/2,
\\
& u_{s,b,i}^{(4)}(\theta)=((\lambda_{s,b,i,0}(\theta)+\lambda_{0}(\theta))
v_{b,i}^++(s\lambda_{b,i}+1-s)v_{b,i}^-)\otimes v_{0}, \quad\ 
1/2\leq \theta\leq 1.
\end{split}
\end{align}	
Then  $u_{s,b,i}^{(3)}(\theta)$ and $u_{s,b,i}^{(4)}(\theta)$ are the
eigenfunctions of $\lambda_{s,b,i,0}(\theta)$ with respect to $D_s$.
Let 
\begin{align}\label{eq:2.23}
w_{s,b,i}^{(j)}(\theta)=u_{s,b,i}^{(j)}(\theta)/\|u_{s,b,i}^{(j)}(\theta)
\|,\quad j=1,2,3,4.
\end{align}
Remark that when $\theta=1/2$, if $s=1/2$ and $\lambda_{b,i}=-1$, then 
$u_{s,b,i}^{(3)}(1/2)=u_{s,b,i}^{(4)}(1/2)=0$. In this case, we define
\begin{align}
w_{s,b,i}^{(3)}(1/2)=\lim_{\theta\rightarrow 1/2^-}
u_{s,b,i}^{(3)}(\theta)/\|u_{s,b,i}^{(3)}(\theta)\|,\quad
w_{s,b,i}^{(4)}(1/2)=\lim_{\theta\rightarrow 1/2^+}
u_{s,b,i}^{(4)}(\theta)/\|u_{s,b,i}^{(4)}(\theta)\|.
\end{align}

		Choose $\chi(\theta)\in \cC^{\infty}([0,1/2])$ with $\chi(\theta)=1/2$ near $\theta=0$ and $\chi(\theta)=0$ near $\theta=1/2$. Let
		\begin{align}\label{eq:2.25}
		\begin{split}
		&w_{s,b,i}^{(5)}(\theta)=\chi(\theta)w_{s,b,i}^{(1)}(\theta)
		+(1-\chi(\theta))w_{s,b,i}^{(3)}(\theta),
		\quad\quad\quad\quad 0\leq \theta\leq 1/2,
		\\
		&w_{s,b,i}^{(6)}(\theta)=\chi(1-\theta)w_{s,b,i}^{(2)}(\theta)
		+(1-\chi(1-\theta))w_{s,b,i}^{(4)}(\theta),
		\quad 1/2\leq \theta\leq 1.
		\end{split}
		\end{align}
Since
$v_1(\theta=0)=v_0(\theta=1)$, $v_0(\theta=0)=v_{-1}(\theta=1)$,
$\lambda_0(0)=\lambda_{-1}(1)=-\pi$ and  $\lambda_0(1)=\lambda_{1}(0)=\pi$,
from (\ref{eq:2.20})-(\ref{eq:2.23}), we have
$w_{s,b,i}^{(1)}(0)=w_{s,b,i}^{(3)}(1)$ and $w_{s,b,i}^{(2)}(1)=w_{s,b,i}^{(4)}(0)$.
By (\ref{eq:2.25}), we have
\begin{align}\label{eq:2.26}
w_{s,b,i}^{(5)}(0)=w_{s,b,i}^{(6)}(1).
\end{align}	
So $\bigoplus_i\C\{w_{s,b,i}^{(5)}(\theta), 0\leq \theta< 1/2\}$ and 
$\bigoplus_i\C\{w_{s,b,i}^{(6)}(\theta), 1/2< \theta\leq 1\}$ 
can be connected as a trivial equivariant complex vector bundle over $B\times (S^1_{\theta}\backslash\{1/2\})\times [0,1]_s$. Then
we could glue $w_{s,b,i}^{(5)}(1/2)$ and $w_{s,b,i}^{(6)}(1/2)$ for any $i$ 
to get an equivariant complex vector bundle $\widetilde{W}$ over $B\times S^1_{\theta}\times [0,1]_s$. Let $\widetilde{R}$ be the orthogonal projection onto the sum of $\widetilde{W}$ and the eigenspaces with non-positive eigenvalues of $\widetilde{D}=\{D_s\}$. Then $\widetilde{Q}=1-\widetilde{R}$ is an equivariant spectral section with respect to $\widetilde{D}$. 
Since $\Ker D_s\neq \emptyset$ only when $s=1/2$, 
from (\ref{eq:2.16}), we have 
\begin{align}\label{eq:2.27}
\mathrm{sf}_G\{(D_{0}, P_0'), (D_{1}, P_1')\}=[P_1'-\widetilde{Q}|_{s=1}]
-[P_0'-\widetilde{Q}|_{s=0}]=[\widetilde{W}|_{s=1}]-[\widetilde{W}|_{s=0}].
\end{align}

For $s=1$, from (\ref{eq:2.21}), (\ref{eq:2.23}) and (\ref{eq:2.25}), we have $w_{0,b,i}^{(5)}(1/2)=w_{0,b,i}^{(3)}(1/2)=(\bar{\lambda}_{b,i}v_{b,i}^+
+v_{b,i}^-)/\sqrt{2}$ and $w_{0,b,i}^{(6)}(1/2)=w_{0,b,i}^{(4)}(1/2)=(
v_{b,i}^++\lambda_{b,i}v_{b,i}^-)/\sqrt{2}$. So $w_{0,b,i}^{(6)}(1/2)
=\lambda_{b,i}\cdot w_{0,b,i}^{(5)}(1/2)$. From the construction
before Lemma \ref{lemma:1.04}, we have
$[\widetilde{W}|_{s=1}]=[W]$.

For $s=0$, in the same way, we calculate that $w_{1,b,i}^{(6)}(1/2)
=w_{1,b,i}^{(5)}(1/2)=(v_{b,i}^+
+v_{b,i}^-)/\sqrt{2}$. So $[\widetilde{W}|_{s=0}]=[U]$.

Therefore, we obtain the claim (\ref{d175}) from (\ref{eq:2.27}).
			
		The proof of Proposition \ref{d123} is completed.
	\end{proof}
		
		Note that the proof of Proposition \ref{d123} in odd case gives a nontrivial example of the equivariant higher spectral flow for even dimensional fibers and an example of the equivariant spectral section without the spectral gap.
	
	\begin{rem}
		In non-equivariant case, there is a stronger version of Proposition \ref{d123} in \cite[Proposition 12]{MP97b}.
%
	\end{rem}

	\subsection{Equivariant local family index theorem}\label{s0202}
	
	In this subsection, we use the notation in Section 1.2 to
	describe the equivariant local index theorem for $\mF\in \mathrm{F}_G^*(B)$ when \textbf{the $G$-action on $B$ is trivial}.
	
	For $b\in B$, let $\cE_{b}$ be the set of smooth sections over $Z_b$ of $\mS_Z\widehat{\otimes} E|_{Z_b}$. As in \cite{Bi86},
	we will regard $\cE$ as
	an infinite dimensional vector bundle over $B$.
	
	Let $\nabla^{TB}$ be the Levi-Civita connection on 
	$(B, g^{TB})$.
	Let $\,^0\nabla^{TW}$ be the connection on $TW=T_{}^HW\oplus TZ$ defined by
	\begin{align}\label{e01015}
	\,^0\nabla^{TW}=\pi^*\nabla^{TB}\oplus\nabla^{TZ}.
	\end{align}
	Then $\,^0\nabla^{TW}$ preserves the metric $g^{TW}$ in (\ref{e01013}).
	Set
	\begin{align}\label{e01017}
	S=\nabla^{TW}-\,^0\nabla^{TW}.
	\end{align}
		
	If $V\in TB$, let $V^H\in T_{\pi}^HW$ be its horizontal lift in $T_{\pi}^HW$ so that $\pi_*V^H=V$.
	For any $V\in TB$, $s\in \cC^{\infty}(B,\cE)=\cC^{\infty}(W,\mS_Z\widehat{\otimes} E)$,
	by \cite[Proposition 1.4]{BF86I}, the connection
	\begin{align}\label{e01028}
	\nabla_V^{\cE}s:=\nabla_{V^H}^{\mS_Z\widehat{\otimes} E}s-\frac{1}{2}\la S(e_i)e_i, V^H \ra \,s
	\end{align}
	preserves the $L^2$-product on $\cE$.
	
	Let $\{f_p\}$ be a local orthonormal frame of $TB$ and $\{f^p\}$ be its dual.
	We denote by $\nabla^{\cE}=f^p\wedge \nabla^{\cE}_{f_p}$.
	Let $T$ be the torsion of $\,^0\nabla^{TW}$. Then $T(f_p^H, f_q^H)\in TZ$.
	We denote by
	\begin{align}\label{e01030}
	c(T)=\frac{1}{2}\,c\left(T(f_p^H, f_q^H)\right)f^p\wedge f^q\wedge.
	\end{align}
	
	By \cite[(3.18)]{Bi86}, the (rescaled) Bismut superconnection
	\begin{align}\label{e01031}
	\mathbb{B}_u:\cC^{\infty}(B,\Lambda(T^*B)\widehat{\otimes}\cE)\rightarrow\cC^{\infty}(B, \Lambda(T^*B)\widehat{\otimes}\cE)
	\end{align}
	is defined by
	\begin{align}\label{e01032}
	\mathbb{B}_u=\sqrt{u}D(\mF)+\nabla^{\cE}-\frac{1}{4\sqrt{u}}c(T).
	\end{align}
	Obviously, the Bismut superconnection $\mathbb{B}_u$ commutes with the $G$-action.
	Moreover, $\mathbb{B}_u^2$ is a $2$-order elliptic differential operator along the fibers $Z$ (cf. \cite[(3.4)]{Bi86}).
	Let $\exp(-\mathbb{B}_u^2)$ be the family of heat operators associated with the fiberwise elliptic operator $\mathbb{B}_u^2$.
	From \cite[Theorem 9.50]{BGV04}, $\exp(-\mathbb{B}_u^2)$ is a smooth family of smoothing operators.
		
If $P$ is a trace class operator acting on $\Lambda(T^*B)\widehat{\otimes}\End(\cE)$ which takes values in $\Lambda(T^*B)$,
	we use the convention that if $\omega\in \Lambda(T^*B)$,
	\begin{align}\label{e01056}
	\tr_s[\omega P]=\omega\tr_s[P].
	\end{align}
	 We denote by $\tr^{\mathrm{odd/even}}_s[P]$ the part of $\tr_s[P]$
	 which takes values in odd or even forms. Set
	\begin{align}\label{i16}
	\widetilde{\tr}[P]=
	\left\{
	\begin{array}{ll}
	\tr_s[P], & \hbox{if $\dim Z$ is even;} \\
	\tr_s^{\mathrm{odd}}[P], & \hbox{if $\dim Z$ is odd.}
	\end{array}
	\right.
	\end{align}
	
	Recall that in this subsection
	we assume that $G$ acts trivially on $B$.
	Take $g\in G$.
	Let $W^g$ be the fixed point set of $g$ on $W$.
	Then $W^g$ is a submanifold of $W$ and $\pi:W^g\rightarrow B$ is a fiber bundle with compact fibers $Z^g$.
	Set
	\begin{align}\label{e01132}
	\ch_g(E, \nabla^E)=
	\tr_s\left[g\exp\left(\frac{\sqrt{-1}}{2\pi}R^E|_{W^g}\right)\right].
	\end{align}
	Let $\ch_g(E)\in H^{even}(W^g,\C)$ denote the cohomology class of $\ch_g(E, \nabla^E)$.
	When the fiber $Z$ is a point, it descends to the equivariant Chern character map
	\begin{align}\label{d500}
	\ch_g:K_G^0(B)\longrightarrow H^{even}(B,\C).
	\end{align}
	By (\ref{d124}), for $x\in K_G^1(B)$,  $j(x)\in K_G^0(B\times S^1)$. The odd equivariant Chern character map
	\begin{align}\label{d130}
	\ch_g:K_G^1(B)\longrightarrow H^{odd}(B,\C)
	\end{align}
	is defined by
	\begin{align}\label{d131}
	\ch_g(x):=\int_{S^1}\ch_g(j(x)).
	\end{align}
	We adopt the sign notation in the integral as in (\ref{e01136}). 
	This is just the equivariant version of the odd Chern character in 
	\cite{Getzler93} and \cite[(1.50)]{Z01} (see e.g., \cite[(3.10)]{LM18a}).
	
	Let $N$ be the normal bundle of $W^g$ in $W$. As $G$ is compact,
	there is an orthonormal decomposition of real vector bundles
	over $W^g$,
	\begin{align}\label{eq:2.40}
	TZ|_{W^g}=TZ^g\oplus N.
	\end{align}
	Let $\nabla$ be a Euclidean connection on $(TZ,g^{TZ})$ commuting with
	the $G$-action. Then its restriction on $W^g$ preserves the 
	decomposition (\ref{eq:2.40}). Let $\nabla^{TZ^g}$ and $\nabla^N$
	be the corresponding induced connections on $TZ^g$ and $N$,
	with curvatures $R^{TZ^g}$ and $R^N$ respectively. Set
	\begin{multline}\label{eq:07}
	\widehat{\mathrm{A}}_{g}(TZ,\nabla)
	:=\mathrm{det}^{1/2}\left(\frac{\frac{\sqrt{-1}}{4\pi}
		R^{TZ^g}}{\sinh \left(\frac{\sqrt{-1}}{4\pi}R^{TZ^g}\right)}\right)
	\\
\times\left(\sqrt{-1}^{\frac{1}{2}\dim
		N}\left.\mathrm{det}^{1/2}\right|_{N}\left(1-
	g \exp\left(\frac{\sqrt{-1}}{2\pi}R^{N}\right)\right)\right)^{-1}.
	\end{multline}	
	If $g$ acts on $L|_{W^g}$ by multiplying by $e^{\sqrt{-1}v}$, we write
	\begin{align}\label{bl0659}
	\ch_g(L_Z^{1/2}, \nabla^{L_Z^{1/2}}):=
	\exp\left(\frac{\sqrt{-1}}{4\pi}R^{L}|_{W^g}+\frac{\sqrt{-1}}{2}v \right).
	\end{align}
	We denote by
	\begin{align}\label{bl0660} 
	\td_g(\nabla, \nabla^{L_Z}):=\widehat{\mathrm{A}}_g(TZ,\nabla) \ch_g(L_Z^{1/2}, \nabla^{L_Z^{1/2}}).
	\end{align}
	Let $\td_g(TZ, L_Z)\in H^{even}(W^g,\C)$ denote the cohomology class of $\td_g(\nabla, \nabla^{L_Z})$.
		
	For $\alpha\in \Omega^i(B)$, set
	\begin{align}\label{e01059}
	\psi_B(\alpha)=\left\{
	\begin{array}{ll}
	\left(\frac{1}{2\pi \sqrt{-1}}\right)^{\frac{i}{2}}\cdot \alpha, & \hbox{if $i$ is even;} \\
	\frac{1}{\sqrt{\pi}}\left(\frac{1}{2\pi \sqrt{-1}}\right)^{\frac{i-1}{2}}\cdot \alpha, & \hbox{if $i$ is odd.}
	\end{array}
	\right.
	\end{align}	
	We state the equivariant family local index theorem here
	(cf. e.g., \cite[Theorem 4.17]{Bi86}, \cite[Theorem 2.10]{BF86II},
	\cite[Theorem 2.2]{Liu17}, \cite[Theorem 2.2]{Liu19} and \cite[Theorem 1.3]{LM00}). Note that
	from \cite[Lemma 4.1]{LMZ00},
	$Z^g$ is naturally oriented.
	\begin{thm}\label{d069}
		For any $u>0$ and $g\in G$, the differential form $\psi_{B}\widetilde{\tr}
		[g\exp(-\mathbb{B}_u^2)]\in \Omega^{*}(B,\C)$
		is closed and its cohomology class represents $\ch_g(\ind(D(\mF)))\in H^*(B, \C)$.
		As $u\rightarrow 0$, we have
		\begin{align}\label{d501}
		\lim_{u\rightarrow 0}\psi_{B}\widetilde{\tr}
		[g\exp(-\mathbb{B}_u^2)]
		=\int_{Z^g}\td_g(\nabla^{TZ},\nabla^{L_Z})\,
		\ch_g(E, \nabla^{E}).
		\end{align}
	\end{thm}
		
	To simplify the notations, we set
	\begin{align}\label{d015}
	\mathrm{FLI}_g(\mF)=\int_{Z^g}\td_g(\nabla^{TZ},\nabla^{L_Z})\,
	\ch_g(E, \nabla^{E})\in \Omega^{*}(B,\C).
	\end{align}
	So Theorem \ref{d069} says that for $\mF\in \mathrm{F}_G^{0/1}(B)$,
	\begin{align}\label{d062}
	[\mathrm{FLI}_g(\mF)]=\ch_g(\ind(D(\mF)))\in H^{even/odd}(B,\C).
	\end{align}
		
		When $\mF$ is the equivariant geometric family in Example \ref{d129} a),
		$Z=\mathrm{pt}$, the equivariant family local index theorem degenerates to the equivariant Chern-Weil theory:
		\begin{align}\label{d177}
		\psi_{B}\widetilde{\tr}
		[g\exp(-\mathbb{B}_u^2)]=\psi_{B}\tr_s[g\exp(-(\nabla^{E})^2)]
		=
		\ch_g(E, \nabla^{E}).
		\end{align}
		In this case, $\mathrm{FLI}_g(\mF)=
		\ch_g(E, \nabla^{E})=\ch_g(E_+, \nabla^{E_+})-\ch_g(E_-, \nabla^{E_-}).$
		
			If $\alpha\in \Lambda(T^*(\R_+\times B))$, 
		\begin{align}\label{e01003}
		\alpha=\alpha_0+ds\wedge \alpha_1,\quad \alpha_0, \alpha_1\in \Lambda (T^*B).
		\end{align}
		Set
		\begin{align}\label{e01005}
		[\alpha]^{ds}=\alpha_1.
		\end{align}
		
			Let $\mF, \mF'\in \mathrm{F}_G^*(B)$ which have the same topological structure. 	By (\ref{d062}), we have $[\mathrm{FLI}_g(\mF)]=[\mathrm{FLI}_g(\mF')]\in H^{*}(B,\C)$. 
		
		We use the notation in (\ref{eq:2.10}) and (\ref{d212}).
		By \cite[Theorem B.5.4]{MM07}, modulo exact forms on $W^g$, the
		equivariant Chern-Simons forms
		\begin{align}\label{e01127}
		\begin{split}
		&\widetilde{\td}_g(\nabla^{TZ}, \nabla^{L_Z},\nabla^{' TZ}, \nabla^{'L_Z} ):=-\int_0^1[\td_g(\nabla^{T\widetilde{Z}}, \nabla^{\widetilde{L_Z}})]^{ds}ds,
		\\
		&\widetilde{\ch}_g(\nabla^{E}, \nabla^{'E}):=-\int_0^1[\ch_g(\widetilde{E},
		\nabla^{\widetilde{E}})]^{ds}ds
		\end{split}
		\end{align}
		depend only on the connections in $\mF$ and $\mF'$.  Moreover,
		\begin{align}\label{e01110}
		\begin{split}
		&d^{W^g}\,\widetilde{\td}_g(\nabla^{TZ},\nabla^{L_Z},\nabla^{' TZ},  \nabla^{'L_Z} )=\td_g(\nabla^{' TZ}, \nabla^{'L_Z} )
		-\td_g(\nabla^{TZ}, \nabla^{L_Z} ),
		\\
		&d^{W^g}\,\widetilde{\ch}_g(\nabla^{E}, \nabla^{'E})=
		\ch_g(E,
		\nabla^{'E})-\ch_g(E,
		\nabla^{E}).
		\end{split}
		\end{align}

		Set
		\begin{multline}\label{d112} 
		\widetilde{\mathrm{FLI}}_g(\mF, \mF')=\int_{Z^g}\widetilde{\td}_g(\nabla^{TZ},\nabla^{L_Z},\nabla^{' TZ},  \nabla^{'L_Z} )
		\, \ch_g(E, \nabla^{E})
		\\
		+\int_{Z^g}\td_g( \nabla^{'TZ}, \nabla^{'L_Z} )
		\, \widetilde{\ch}_g(\nabla^{E},\nabla^{'E})\in \Omega^*(B,\C)/d\Omega^*(B,\C).
		\end{multline}
		From \cite[(1.7)]{BGV04}, for $\sigma\in \Omega^{*}(W^g)$,
		using the sign convention in (\ref{e01136}), we have
		\begin{align}\label{eq:3.23}
		d^B\int_{Z^g}\sigma=\int_{Z^g}d^{W^g}\sigma.
		\end{align}
		From (\ref{e01110}) and (\ref{eq:3.23}), we have
		\begin{align}
		d^B\,\widetilde{\mathrm{FLI}}_g(\mF, \mF')=\mathrm{FLI}_g(\mF')-\mathrm{FLI}_g(\mF).
		\end{align}

	\subsection{Equivariant eta form}\label{s0203}
	
	In this subsection, we also assume that $G$ acts trivially on $B$. 
	We define the equivariant Bismut-Cheeger eta form with perturbation operator in Definition \ref{d179}.
%
%
%
%
	
	In this subsection, \textbf{we assume that there exists a perturbation operator with respect to $D(\mF)$ on $\mF$.} It implies that $\ind(D(\mF))=0\in K_G^*(B)$.

	Let $A$ be a perturbation operator with respect to $D(\mF)$.
	We extend $A$ to $1\widehat{\otimes}A$ on 
	$\cC^{\infty}(B,\pi^*\Lambda(T^*B)\widehat{\otimes}\cE)$
	as an element of the $\Z_2$-graded tensor product of 
	$\Z_2$-graded algebras. In this case,
	\begin{align}
	(\alpha\widehat{\otimes}1)(1\widehat{\otimes}A)=
	(-1)^{\deg \alpha}(1\widehat{\otimes}A)(\alpha\widehat{\otimes}1).
	\end{align}
	We usually abbreviate $1\widehat{\otimes}A$ by $A$ when there is 
	no confusion.
	
Let $\chi\in \cC^{\infty}(\R)$ be a cut-off function such that
\begin{align}\label{d013}
\chi(u)=
\left\{
\begin{aligned}
&0,  &\hbox{if $u<1$;} \\
&1,  &\hbox{if $u>2$.}
\end{aligned}
\right.
\end{align}
	
	Set
	\begin{align}\label{d070}
	\mathbb{B}_u'=\mathbb{B}_u+\sqrt{u}\chi(\sqrt{u})A.
	\end{align}
	Since $\chi(\sqrt{u})=0$ if $u\in (0,1)$, by
	(\ref{d501}) and (\ref{d015}),
	\begin{align}\label{d101}
	\lim_{u\rightarrow 0}\psi_{B}\widetilde{\tr}
	[g\exp(-(\mathbb{B}_u')^2)]=\mathrm{FLI}_g(\mF)\in \Omega^{*}(B, \C).
	\end{align}
	Since $\chi(\sqrt{u})=1$ if $u\in (2,+\infty)$,
	from \cite[Theorem 9.19]{BGV04},  we have
	\begin{align}\label{d016}
	\lim_{u\rightarrow +\infty}\psi_{B}\widetilde{\tr}\left[g\exp\left(-(\mathbb{B}_u')^{2}\right)\right]=0.
	\end{align}

	\begin{defn}\label{d017}
		For any $g\in G$, modulo exact forms on $B$, the equivariant eta form with perturbation operator $A$ is defined by
		\begin{multline}\label{i17}
		\tilde{\eta}_g(\mF, A)
		=
		-\int_0^{\infty}\left\{\psi_{\R\times B}\left.\widetilde{\tr}\right.\left[g\exp\left(-\left(
		\mathbb{B}_{u}'+du\wedge\frac{\partial}{\partial u}\right)^2\right)\right]\right\}^{du}du
		\\
		\in \Omega^*(B,\C)/d\Omega^*(B,\C).
		\end{multline}		
	\end{defn}
	
	The regularities of the integral in the right hand side of (\ref{i17}) are proved in \cite[Section 2.4]{Liu17}.
	As in \cite[(2.81)]{Liu17}, we have
	\begin{align}\label{d019}
	d\tilde{\eta}_g(\mF, A)=\mathrm{FLI}_g (\mF).
	\end{align}
	
	As in \cite[(2.95)]{Liu17}, the value of $\widetilde{\eta}_g(\mF,A)$ in $\Omega^*(B,\C)/d\Omega^*(B,\C)$ is independent of the choice of the cut-off function.  Similarly, if $A_P$ and $A_P'$ are two smoothing operators associated with the same equivariant spectral section $P$, we have $\widetilde{\eta}_g(\mF,A_P)=\widetilde{\eta}_g(\mF,A_P')\in \Omega^*(B,\C)/d\Omega^*(B,\C)$.
	In this case, we often simply denote it by $\widetilde{\eta}_g(\mF,P)$.

	If the fiber $Z$ is connected, we could calculate the equivariant eta form explicitly:
	\begin{multline}\label{i19}
	\tilde{\eta}_g(\mF, A)=
	\left\{
	\begin{aligned}
	& \int_0^{\infty}\left.\frac{1}{\sqrt{\pi}}\psi_{B}\tr_s^{\mathrm{even}}\right.\left[g\left.\frac{\partial \mathbb{B}_u'}{\partial u}\right.
	\exp(-(\mathbb{B}_u')^{2})\right] du \in \Omega^*(B,\C)/d\Omega^*(B,\C),
	\\
	& \quad\quad\quad\quad\quad\quad\quad\quad\quad\quad\quad\quad\quad\quad\quad\quad\quad\quad\quad
	\quad\quad\quad\hbox{if $\mF$ is odd;} \\
	&\int_0^{\infty} \left.\frac{1}{2\sqrt{\pi}\sqrt{-1}}\psi_{B}\tr_s\right.\left[g\left.\frac{\partial \mathbb{B}_u'}{\partial u}\right.
	\exp(-(\mathbb{B}_u')^{2})\right] du
	\in \Omega^*(B,\C)/d\Omega^*(B,\C),
	\\
	& \quad\quad\quad\quad\quad\quad\quad\quad\quad\quad\quad\quad\quad\quad\quad\quad\quad
	\hbox{if $\mF$ is even and $\dim Z>0$.} \\
	&\int_0^{\infty} \left.\frac{\sqrt{-1}}{2\pi}\tr_s\right.\left[g\left.\frac{\partial \nabla^E_u}{\partial u}\right.
		\exp\left(-\frac{(\nabla^{E}_u)^2}{2\pi \sqrt{-1}}\right)\right] du
		\in \Omega^*(B,\C)/d\Omega^*(B,\C), 
		\\
		&\quad\quad\quad\quad\quad\quad\quad\quad\quad\quad\quad\quad\quad\quad\quad\quad\quad
		\hbox{if $\dim Z=0$,} \\
	\end{aligned}
	\right.
	\end{multline}
	where $\nabla_u^E=\nabla^E+\sqrt{u}\chi(\sqrt{u})A$.
		
	When $\dim Z=0$, the equivariant geometric family degenerates to the case of Example \ref{d129} a). 
	Then there exists a complex vector bundle $E'$ such that $E_+\oplus E'\simeq
	E_-\oplus E'$ as complex vector bundles.
	As in \cite[Definition B.5.3]{MM07}, from (\ref{d177}) and (\ref{d019}), the equivariant eta form in this case is just the equivariant transgression between $\ch_g(E_+\oplus E',\nabla^{E_+\oplus E'})$ and $\ch_g(E_-\oplus E',\nabla^{E_-\oplus E'})$.
		
	Furthermore, by changing the variable (see also \cite[Remark 2.4]{Liu17}), we could get another form of equivariant eta form:
	\begin{align}\label{i18}
	\tilde{\eta}_g(\mF, A)
	=
	-\int_0^{\infty}\left\{\psi_{\R \times B}\left.\widetilde{\tr}\right.\left[g\exp\left(-\left(
	\mathbb{B}_{u^2}'+du\wedge\frac{\partial}{\partial u}\right)^2\right)\right]\right\}^{du}du.
	\end{align}
		
	Let $(Z', g^{TZ'})$ be an even dimensional Spin$^c$ manifold and $(E', h^{E'}, \nabla^{E'})$ be a $\Z_2$-graded Hermitian vector bundle over $Z'$ with a Hermitian connection $\nabla^{E'}$. Let $\mathrm{pr}_2: B\times Z'\rightarrow Z'$ be the projection onto the second part. Then all the bundles and geometric data above could be pulled back on $B\times Z'$. Thus the fiber bundle $B\times Z'\rightarrow B$ and the structures pulled back by $\mathrm{pr}_2$ form a geometric family $\mF'$ with fibers $Z'$. In this case, $\ind (D(\mF'))$ 
	is a trivial virtual complex vector bundle over $B$.
	It could also be regarded as a locally constant function on $B$
	with values in $\Z$.
	  We assume that the group action on $\mF'$ is trivial. For $\mF\in\mathrm{F}_G^*(B)$,
	let $A$ be a perturbation operator with respect to $D(\mF)$ on $\mF$.
	Let $\tau'$ be the $\Z_2$-grading of $\mS_{Z'}\widehat{\otimes}E'$.
	As in (\ref{eq:1.06}), we define 
	\begin{align}\label{eq:2.66}
	A\widehat{\otimes}1:=A\otimes \tau'
	\end{align}
	on $\mF\times_B\mF'$. By (\ref{eq:1.06}), we have
	\begin{align}
	(D(\mF\times_B\mF')+A\widehat{\otimes}1)^2=(D(\mF)+A)^2\widehat{\otimes}1
	+1\widehat{\otimes}D(\mF')^2>0.
	\end{align}
	Thus $A\widehat{\otimes} 1$ is a perturbation operator with respect to $D(\mF\times_B \mF')$.
	
	\begin{lemma}\label{d180}
	For $g\in G$, we have
	\begin{align}\label{d181}
	\tilde{\eta}_g(\mF\times_B \mF', A\widehat{\otimes} 1)=\tilde{\eta}_g(\mF, A)\cdot \ind (D(\mF'))\in 
	\Omega^*(B,\C)/d\Omega^*(B,\C).
	\end{align}
	Here we consider $\ind(D(\mF'))$ as a locally constant function on $B$
	with values in $\Z$.
	\end{lemma}
	\begin{proof}
		We denote by $\tr|_{\mF}$ the trace operator associated with $\mF$. Then
		from (\ref{i18}),
		\begin{multline}
		\tilde{\eta}_g(\mF\times_B \mF', A\widehat{\otimes} 1)
		\\
		=
		-\int_0^{\infty}\left\{\psi_{\R\times B}\left.\widetilde{\tr}|_{\mF\times_B \mF'}\right.\left[g\exp\left(-\left(
		\mathbb{B}_{u^2}'\widehat{\otimes} 1+1\widehat{\otimes} u D(\mF')+du\wedge\frac{\partial}{\partial u}\right)^2\right)\right]\right\}^{du}du
		\\
		=\int_0^{\infty}\left\{\psi_{\R\times B}\left.\widetilde{\tr}|_{\mF\times_B \mF'}\right.\left[g(1\widehat{\otimes} D(\mF'))\exp\left(-\left(
		\mathbb{B}_{u^2}'\widehat{\otimes} 1\right)^2-\left(1\widehat{\otimes} u D(\mF')\right)^2\right)\right]\right\}du
		\\
		-\int_0^{\infty}\left\{\psi_{\R\times B}\left.\widetilde{\tr}|_{\mF\times_B \mF'}\right.\left[g\exp\left(-\left(
		\mathbb{B}_{u^2}'\widehat{\otimes} 1+du\wedge\frac{\partial}{\partial u}\right)^2-1\widehat{\otimes} u^2 D(\mF')^2\right)\right]\right\}^{du}du
		\\
		=\int_0^{\infty}\left\{\psi_{\R\times B}\left.\widetilde{\tr}|_{\mF}\right.\left[g\exp\left(-\left(
		\mathbb{B}_{u^2}'\right)^2\right)\right]\cdot\tr_s|_{\mF'} \left[D(\mF')\exp\left(-u^2 D(\mF')^2\right)\right]\right\}du
		\\
		-\int_0^{\infty}\left\{\psi_{\R\times B}\left.\widetilde{\tr}|_{\mF}\right.\left[g\exp\left(-\left(
		\mathbb{B}_{u^2}+du\wedge\frac{\partial}{\partial u}\right)^2\right)\right]\cdot\tr_s|_{\mF'} \left[\exp\left(-u^2 D(\mF')^2\right)\right]\right\}^{du}du.
		\end{multline}
		
		From the definition of $\mF'$ and the local index theorem,
		as functions on $B$, we have
		\begin{align}
		\begin{split}
		&\tr_s|_{\mF'} \left[D(\mF')\exp\left(-\left(u^2 D(\mF')^2\right)\right)\right]=0,
		\\
		&\tr_s|_{\mF'} \left[\exp\left(-u^2 D(\mF')^2\right)\right]=\ind (D(\mF')).
		\end{split}
		\end{align}		
		So we get $\tilde{\eta}_g(\mF\times_B \mF', A\widehat{\otimes}1)=\tilde{\eta}_g(\mF, A)\cdot \ind (D(\mF'))$. 
		
		The proof of Lemma \ref{d180} is completed.
	\end{proof}
		
	\subsection{Anomaly formula for odd equivariant geometric families}\label{s0204}
	
	In this subsection, we will study the anomaly formula of the equivariant eta forms for two odd equivariant geometric families $\mF$ and $\mF'$ with the same topological structure.
	In this subsection, we also assume that $G$ acts on $B$ trivially.
	
		Assume that $\mF\in \mathrm{F}_G^1(B)$. Let $A$ be a family of bounded pseudodifferential operator on $\mF$ such that $D(\mF)+A$ is an equivariant $B$-family.
	Let $P$, $Q$ be two equivariant spectral sections with respect to $D(\mF)+A$. 
	Let $A_P$, $A_Q$ be smoothing operators associated with $P$, $Q$.
	Then $A+A_P$ and $A+A_Q$ are perturbation operators of $D(\mF)$.
	In this case, by (\ref{d019}), the difference of 
	$\tilde{\eta}_g(\mF, A+A_P)$ and $\tilde{\eta}_g(\mF, A+A_Q)$
	is closed. Furthermore, we have the following lemma.
	
	\begin{lemma}\label{d106}(Compare with \cite[Proposition 17]{MP97a})
		For any $g\in G$, modulo exact forms on $B$, we have
		\begin{align}\label{d018}
		\tilde{\eta}_g(\mF, A+A_P)-\tilde{\eta}_g(\mF, A+A_Q)=\ch_g([P-Q])\in H^{even}(B,\C).
		\end{align}
	\end{lemma}
	\begin{proof}
		
		Note that $A$, $A_P$, $A_Q$ preserve the $\Z_2$-grading of $E$ 
		and if we reverse the orientation of the fibers, the $\eta$-form
		is changed to its minus.
		From (\ref{d011}) and the orientation reversing trick in the proof of Proposition \ref{d006} (i), we only need to prove the lemma when $Q$ majorizes $P$ and $E_-=0$ in $\mF$.
		
		Let $\widetilde{\mF}$ be the equivariant geometric family defined in (\ref{d212}) such that $\mF_r= \mF$ for any $r\in[0,1]$. Let $\widetilde{\mathbb{B}}_{u}$ be the Bismut superconnection associated with $\widetilde{\mF}$.
		We choose $s>0$ large enough such that
		$P$ and $Q$ satisfy (\ref{d004})  for $f(b)\equiv s$.
		We choose equivariant spectral sections $R'$ and $R''$
		as in (\ref{d173}).
		Since the $\eta$-form is independent of the smoothing operators with respect to the same equivariant spectral section,
		we may choose the smoothing operators $A_P$, $A_Q$ as in (\ref{d173}). Set $A_r:=A+rA_Q+(1-r)A_P$. Let
		\begin{align}
		\widetilde{\mathbb{B}}_{u}'|_{(u,r)}:=\widetilde{\mathbb{B}}_u|_{(u,r)}+\sqrt{u}\chi(\sqrt{u})A_r
		\end{align}
		as in (\ref{d070}).
		We simply denote by 
		\begin{align}
		D_r:=D(\mF)+A_r,\quad \wi{\nabla}:=\nabla^{\cE}+dr\wedge\frac{\partial }{\partial r}.
				\end{align}
		Then from (\ref{e01032}), when 
		$u>2$, we have
		\begin{align}\label{eq:2.72}
	\left(	\widetilde{\mathbb{B}}_{u^2}'\right)^2
		=u^2D_r^2+u [D_r, \wi{\nabla}]+ \wi{\nabla}^2
		+\frac{1}{4}[D_r, c(T)]+\frac{1}{4u}[\wi{\nabla},c(T)]
		+\frac{1}{16u^2}c(T)^2.
		\end{align}

		For a family of bounded operators $\mA_u$, $u\in \R_+$, we write
		$\mA_u=O(u^{-k})$ as $u\to +\infty$ if there exists $C>0$
		such that if $u$ is large enough, the norm of $\mA_u$ is dominated by $C/u^k$.

			Let $\Pi$ be the orthogonal projection onto $[P-Q]$, an equivariant complex vector bundle over $\widetilde{B}$. 
	Let
	\begin{align}
	\begin{split}
	&E_u=\Pi\circ	\left(	\widetilde{\mathbb{B}}_{u^2}'\right)^2\circ \Pi,\quad 
	F_u=\Pi\circ 	\left(	\widetilde{\mathbb{B}}_{u^2}'\right)^2\circ \Pi^{\bot},\quad 
	\\
	&G_u=\Pi^{\bot}\circ 	\left(	\widetilde{\mathbb{B}}_{u^2}'\right)^2\circ \Pi,\quad 
	H_u=\Pi^{\bot}\circ	\left(	\widetilde{\mathbb{B}}_{u^2}'\right)^2\circ \Pi^{\bot}.
	\end{split}
	\end{align}		
			From (\ref{d173}),
	since $\Pi$ commutes with $P, Q, R', R''$, $\Pi Q=\Pi R'=0$
	and $\Pi R''=\Pi P=\Pi$,
			  we have
		\begin{multline}\label{d111} 
		\Pi\circ D_r\circ\Pi=
		r\,\Pi\circ \big(R'(D(\mF)+A)R'+sQR''(I-R') +(I-R'')(D(\mF)+A)(I-R'')
		\\
		-s(I-Q)R''(I-R') \big)\circ\Pi
		+(1-r)\,\Pi\circ \big(R'(D(\mF)+A)R'+sPR''(I-R')
		\\
		 +(I-R'')(D(\mF)+A)(I-R'')-s(I-P)R''(I-R') \big)\circ\Pi
		=
		s(1-2r)\Pi.
		\end{multline}
	Since $D_r$ preserves the splitting $\mathrm{Range}(\Pi)
	\oplus \mathrm{Range}(\Pi^{\bot})$, 
	\begin{align}
	\Pi\circ [D_r, \wi{\nabla}]\circ\Pi=[\Pi\circ D_r\circ \Pi, \ \Pi\circ\wi{\nabla}\circ\Pi]
	=dr\wedge \frac{\partial}{\partial r}(\Pi\circ D_r\circ \Pi)
	=-2sdr\wedge\circ \Pi.
	\end{align}
	Similarly, we have $\Pi\circ [D_r, c(T)]\circ\Pi=0$
and 
\begin{align}
\Pi\circ \wi{\nabla}^2\circ \Pi =\Pi\circ (\nabla^{\cE})^2\circ \Pi.
\end{align}	
	
Let 
\begin{align}\label{eq:2.77}
\begin{split}
&E'= u^2s^2(1-2r)^2\Pi -2us dr\wedge\circ\, \Pi +\Pi\circ (\nabla^{\cE})^2\circ \Pi,
\quad  F'=\Pi^{\bot}\circ [D_r, \wi{\nabla}]\circ \Pi,
\\
& G'=\Pi\circ [D_r, \wi{\nabla}]\circ \Pi^{\bot},\quad
 H'=\Pi^{\bot}\circ D_r^2\circ \Pi^{\bot}.
\end{split}
\end{align}
From (\ref{eq:2.72})-(\ref{eq:2.77}), when $u\to +\infty$,
\begin{align}\label{eq:2.78}
\begin{split}
&E_u=E'+O(u^{-1}),
\quad  F_u=uF'+F''+O(1),
\\
& G_u=uG'+G''+O(1),\quad
H_u=u^2H'+uH''+H'''+O(1),
\end{split}
\end{align}
where $F'', G'', H'', H'''$ are first order differential
operators along the fiber.
	Let 
	\begin{align}
	\nabla^{\Pi}:=\Pi\circ\nabla^{\cE}\circ\Pi.
	\end{align}
We have
\begin{align}
E'-F'H'^{-1}G'=u^2s^2(1-2r)^2\Pi -2us dr\wedge\circ\, \Pi+(\nabla^{\Pi})^2.
\end{align}
Following the same way as the proof of \cite[Theorem 5.13]{Liu17}, we can obtain
\begin{align}
\exp\left(-	\left(	\widetilde{\mathbb{B}}_{u^2}'\right)^2\right)
=\Pi\circ\exp(-(E'-F'H'^{-1}G'))\circ\Pi+O(u^{-1}).
\end{align}
Thus we have
\begin{align}
\left[\exp\left(-	\left(	\widetilde{\mathbb{B}}_{u^2}'\right)^2\right)\right]^{dr}=2use^{-u^2s^2(1-2r)^2}\Pi\circ\exp\left(-\left(\nabla^{\Pi}\right)^2\right)\circ\Pi+O(u^{-1}).
\end{align}

		
		Set $$r_1(u,r)=\left.\left\{\psi_B\tr_s^{\mathrm{odd}}\left[g\exp\left(-	\left(	\widetilde{\mathbb{B}}_{u^2}'\right)^2\right)\right]\right\}^{dr}\right|_{(u,r)}.$$
		From \cite[(2.95)]{Liu17}, modulo exact forms on $B$,
		 we have
		\begin{multline}
		\tilde{\eta}_g(\mF, A+A_P)-\tilde{\eta}_g(\mF, A+A_Q)=\lim_{u\rightarrow+\infty}\int_0^1r_1(u,r)dr
		\\
		=\frac{1}{\sqrt{\pi}}\lim_{u\rightarrow+\infty}\int_0^12use^{-u^2s^2(1-2r)^2}dr\cdot\psi_B\tr[g\exp(-(\nabla^{\Pi})^2)]
		\\
		=\frac{1}{\sqrt{\pi}}\lim_{u\rightarrow+\infty}
		\int_{-us}^{us}e^{-x^2}dx\cdot\ch_g([P-Q])=\ch_g([P-Q]).
		\end{multline}
		
	The proof of Lemma \ref{d106} is completed.
	\end{proof}

	Using Lemma \ref{d106}, we obtain the anomaly formula in odd case as follows. 
	\begin{prop}\label{d022}(Compare with \cite[Theorem 0.1]{DZ98})
		Let $\mF$, $\mF'\in \mathrm{F}_G^1(B)$ which have the same topological structure.	Let $A$, $A'$ be perturbation operators with respect to $D(\mF)$, $D(\mF')$
		and $P$, $P'$ be APS projections with respect to $D(\mF)+A$, $D(\mF')+A'$ respectively.
		For any $g\in G$, modulo exact forms on $B$, we have
		\begin{align}\label{d023}
		\tilde{\eta}_g(\mF', A')-\tilde{\eta}_g(\mF, A)=\widetilde{\mathrm{FLI}}_g(\mF, \mF')
		+
		\ch_g\left(\mathrm{sf}_G\{(D(\mF)+A, P), (D(\mF')+A', P')\}\right).
		\end{align}
	\end{prop}
	\begin{proof}
		Let $\widetilde{\mF}$ be the equivariant geometric family defined in (\ref{d212}).
		Let $D_r=D(\mF_r)+(1-r)A+r A'$ and $\widetilde{D}=\{D_r\}_{r\in[0,1]}$ on $\widetilde{\mF}$.
		Since the equivariant family index of $D(\mF)$ vanishes, so are $D_r$ and $\widetilde{D}$.
		If we consider the
		total family $\widetilde{\mF}$, from Proposition \ref{d006}(i), there exists a total equivariant spectral
		section $\widetilde{P}$ of $\widetilde{D}$. Let $P_r$ be the restriction of $\widetilde{P}$ over $\{r\}\times B$.
		Then it is an equivariant spectral section of $D_r$.
		Let $A_{P_r}$ be an equivariant smoothing operator associated with $P_r$.
		Following the proof of \cite[Theorem 2.7]{Liu17}, we can get
		\begin{align}\label{d025}
		\tilde{\eta}_g(\mF', A'+A_{P_1})-\tilde{\eta}_g(\mF, A+A_{P_0})=\widetilde{\mathrm{FLI}}_g(\mF, \mF').
		\end{align}
		Thus Proposition \ref{d022} follows from Lemma \ref{d106}, (\ref{d103}) and (\ref{d025}).
		
		The proof of Proposition \ref{d022} is completed.
	\end{proof}

	\subsection{Functoriality of equivariant eta forms}\label{s0205}
In this subsection, we will study the functoriality of the equivariant eta forms and use it to prove the anomaly formula of equivariant eta forms for even equivariant geometric families.
	In this subsection, we use the notation in Section 1.4 and assume that $G$ acts trivially on $B$.
	
Recall that in (\ref{e02003}), $TZ=T_{\pi_X}^HZ\oplus TX$. 
Let $\nabla^{TY, TX}$ be the connection on $TZ$ defined by
\begin{align}\label{e02114}
\nabla^{TY, TX}=\pi_X^*\nabla^{TY}\oplus\nabla^{TX}
\end{align}
as in (\ref{e01015}). 

Let $\nabla$, $\nabla'$ be Euclidean connections on $(TZ,g^{TZ})$
and $\nabla^{L_Z}$, $\nabla^{'L_Z}$ be Hermitian connections on $(L_Z,h^{L_Z})$.
Similarly as (\ref{d015}), we define
\begin{align}\label{d183}
\mathrm{FLI}_g(\nabla, \nabla^{L_Z}):=\int_{Z^g}\td_g(\nabla, \nabla^{L_Z})\,
\ch_g(E, \nabla^{E}).
\end{align}
As in (\ref{e01127}) and (\ref{e01110}), there exists a well-defined equivariant Chern-Simons form $\widetilde{\td}_g(\nabla, \nabla^{L_Z}, \nabla', \nabla^{'L_Z})\in \Omega^*(W^g,\C)/d\Omega^*(W^g,\C)$ such that
\begin{align}\label{d182}
d^{W^g}\,\widetilde{\td}_g(\nabla, \nabla^{L_Z}, \nabla', \nabla^{'L_Z})=\td_g(\nabla', \nabla^{'L_Z})-\td_g(\nabla, \nabla^{L_Z}).
\end{align}
Set
\begin{align}\label{d195}
\widetilde{\mathrm{FLI}}_g(\nabla, \nabla^{L_Z}, \nabla', \nabla^{'L_Z}):=\int_{Z^g}\widetilde{\td}_g(\nabla, \nabla^{L_Z}, \nabla', \nabla^{'L_Z})\,
\ch_g(E, \nabla^{E}).
\end{align}
From (\ref{eq:3.23}) and (\ref{d182}), we have 
\begin{align}\label{d184}
d^B\,\widetilde{\mathrm{FLI}}_g(\nabla, \nabla^{L_Z}, \nabla', \nabla^{'L_Z})=\mathrm{FLI}_g(\nabla', \nabla^{'L_Z})-\mathrm{FLI}_g(\nabla, \nabla^{L_Z}).
\end{align}

From the proof of Lemma \ref{d089}, we obtain that if $A_X$ is a perturbation
operator of $D(\mF_X)$, there exists $T'>0$ such that when $T\geq T'$,
$1\widehat{\otimes}TA_X$ is a perturbation operator of $D(\mF_{Z,T})$.

Note that when $A_X$ is a family of smoothing operators along the fibers $X$, $1\widehat{\otimes}A_X$ is only bounded, not a family of smoothing operators along the fibers $Z$. This is the reason for us to define the eta form for the bounded perturbation operator instead of the smoothing operator in \cite{Bu09,BS09,DZ98,MP97a}.
	
	The following technical lemma is a modification of the main result in \cite{Liu17}. The proof of it will be left to the next subsection.
	
	\begin{lemma}\label{d030}
		Modulo exact forms on $B$, for $T\geq T'$, we have
		\begin{multline}\label{d037}
		\widetilde{\eta}_g(\mF_{Z, T}, 1\widehat{\otimes}TA_{X})=\int_{Y^g}\td_g(\nabla^{TY}, \nabla^{L_Y})\, \widetilde{\eta}_g(\mF_{X},A_X)
		\\
		+
		\widetilde{\mathrm{FLI}}_g\left(\nabla^{TY,TX}, \nabla^{L_Z}, \nabla_{T}^{TZ}, \nabla^{L_Z} \right).
		\end{multline}
		Here $\nabla_T^{TZ}$ is the connection associated with $(T_{\pi_Z}^HW, g_T^{TZ})$ as in (\ref{e01014}).
	\end{lemma}
		
	Using Lemma \ref{d030}, we could extend the anomaly formula Proposition \ref{d022} to the general case.
		\begin{thm}\label{d176}
			Let $\mF$, $\mF'\in \mathrm{F}_G^*(B)$ which have the same topological structure.	Let $A$, $A'$ be perturbation operators with respect to $D(\mF)$, $D(\mF')$
			and $P$, $P'$ be the APS projections with respect to $D(\mF)+A$, $D(\mF')+A'$ respectively.
			For any $g\in G$, modulo exact forms on $B$, we have
			\begin{align}\label{d178}
			\tilde{\eta}_g(\mF', A')-\tilde{\eta}_g(\mF, A)=\widetilde{\mathrm{FLI}}_g(\mF, \mF')
			+
			\ch_g\left(\mathrm{sf}_G\{(D(\mF)+A, P), (D(\mF')+A', P')\}\right).
			\end{align}
		\end{thm}
		\begin{proof}
			We only need to prove the even case. 
			
			 Let $L\rightarrow S^1\times S^1$ be the Hermitian line bundle in Example \ref{d129} c) with $\nabla^L$ constructed there. 
			 We use the notation in Example \ref{d129} c).
			 Let $p: B\times S^1\times S^1\rightarrow S^1\times S^1$ be the natural projection.  Then all the bundles and geometric data in $\mF^L$ could be pulled back on $B\times S^1\times S^1$. Thus the fiber bundle $B\times S^1\times S^1\rightarrow B$ and the structures pulled back by $p$ form an even geometric family $\mF_0$ over $B$. In this case, $\ind (D(\mF_0))=1$.
			 Here we consider $\ind (D(\mF_0))$ as a locally constant function on $B$ as in Lemma \ref{d180}.
			  The key observation is
			 \begin{align}\label{d210} 
			 	p_1!(p_1^*\mF\times_{B\times S^1} p_2^*\mF^L)=\mF\times_B \mF_0.
			 \end{align}
			
			Recall that $A\widehat{\otimes} 1_{\mF_0}$ is defined in (\ref{eq:2.66}).
			Since $A\widehat{\otimes} 1_{\mF_0}$ is a perturbation operator of $D(\mF\times_B \mF_0)$, we could choose $T'=1$ in Lemma \ref{d030}.
			By Lemmas \ref{d180} and \ref{d030}, we have
			\begin{multline}\label{d187}
			\tilde{\eta}_g(\mF, A)=\tilde{\eta}_g(\mF\times_B \mF_0, A\widehat{\otimes} 1_{\mF_0})
			\\
			=
			\int_{S^1}\tilde{\eta}_g(p_1^*\mF\times_{B\times S^1} p_2^*\mF^L, A\widehat{\otimes} 1_{\mF^L})-\widetilde{\mathrm{FLI}}_g\left(\nabla^{T(Z\times S^1)}, \nabla^{L_Z}, \nabla^{TZ,TS^1}, \nabla^{L_Z} \right).
			\end{multline}
			
			As in (\ref{eq:2.11}), $D(p_1^*\mF\times_{B\times S^1} p_2^*\mF^L)=D(\mF)\otimes 1+\tau\otimes D^L$.
			By Proposition \ref{d022}, the construction of the equivariant higher spectral flow for even case and (\ref{d131}), we have
			\begin{multline}\label{d211}
			\tilde{\eta}_g(\mF', A')-\tilde{\eta}_g(\mF, A)
			\\
			=
			\int_{S^1}\left\{\tilde{\eta}_g(p_1^*\mF'\times_{B\times S^1} p_2^*\mF^L, A'\widehat{\otimes} 1_{\mF^L})-\tilde{\eta}_g(p_1^*\mF\times_{B\times S^1} p_2^*\mF^L, A\widehat{\otimes} 1_{\mF^L})\right\}
			\\
			+\int_{S^1}\int_{Z^g}\left\{\widetilde{\td}_g\left(\nabla^{T(S^1\times Z')}, \nabla^{L_{Z'}}, \nabla^{TS^1,TZ'}, \nabla^{L_{Z'}}\right)\ch_g(E',\nabla^{E'})
			\right.
			\\
			-\left.\widetilde{\td}_g\left(\nabla^{T(S^1\times Z)}, \nabla^{L_{Z}}, \nabla^{TS^1,TZ}, \nabla^{L_{Z}}\right)
			\ch_g(E,\nabla^E)\right\}
			\\
			=\widetilde{\mathrm{FLI}}_g(\mF, \mF')
			+\int_{S^1}\ch_g\left(\mathrm{sf}_G\{(D(p_1^*\mF\times_{B\times S^1} p_2^*\mF^L)+A\widehat{\otimes} 1, P_0), (D(p_1^*\mF'\times_{B\times S^1} p_2^*\mF^L)+A'\widehat{\otimes} 1, P_0')\}\right)
			\\
			=\widetilde{\mathrm{FLI}}_g(\mF, \mF')+\ch_g\left(\mathrm{sf}_G\{(D(\mF)+A, P), (D(\mF')+A', P')\}\right),
			\end{multline}
			where $P_0$, $P_0'$ are the associated APS projections respectively. Note that in order to adapt the sign convention (\ref{e01136}), the sign in the beginning of the fifth line of (\ref{d211}) is alternated.  
					
		The proof of Theorem \ref{d176} is completed.
		
	\end{proof}
	
	Using Theorem \ref{d176}, we could write Lemma \ref{d030} as a more elegant form.
	 \begin{thm}\label{d188}
	 	Let $A_Z$ and $A_X$ be perturbation operators with respect to $D(\mF_Z)$ and $D(\mF_X)$.
	 	Then modulo exact forms on $B$, for $T\geq 1$ large enough, we have
	 	\begin{multline}\label{d189}
	 	\widetilde{\eta}_g(\mF_{Z}, A_{Z})=\int_{Y^g}\td_g(\nabla^{TY}, \nabla^{L_Y})\, \widetilde{\eta}_g(\mF_{X},A_X)
	 +
	 	\widetilde{\mathrm{FLI}}_g\left(\nabla^{TY,TX}, \nabla^{L_Z}, \nabla_{}^{TZ}, \nabla^{L_Z} \right)
	 	\\
	 	+\ch_g(\mathrm{sf}_G\{(D(\mF_{Z,T})+ 1\widehat{\otimes}TA_X, P), (D(\mF_Z)+A_Z, P')\}),
	 	\end{multline}
	 	where $P$ and $P'$ are the associated APS projections respectively.
	 \end{thm}
	 
	 From Theorems \ref{d176} and \ref{d188}, we could extend Lemma \ref{d180} to the general case.
	 	 
	 \begin{thm}\label{d198}(Compare with \cite[(24)]{BS09})
	 	Let $\mF, \mF'\in \mathrm{F}_G^*(B)$. Let $A$ and $A'$ be the perturbation operators with respect to $D(\mF)$ and $D(\mF\times_B\mF')$.
	 	Then there exists $x\in K_G^*(B)$, such that
	 	\begin{align}
	 	\widetilde{\eta}_g(\mF\times_B\mF', A')=\widetilde{\eta}_g(\mF, A)\, \mathrm{FLI}_g(\mF')+\ch_g(x).
	 	\end{align}
	 \end{thm}
	 \begin{proof}
	 	Here we use a trick in \cite{BS09} similarly as (\ref{d210}). Let $\pi':W'\rightarrow B$ be the submersion in $\mF'$. 
	 	We could obtain a pullback family $\pi^{'*}\mF$ by choosing a horizontal 
	 	subbundle $T^H(\pi'^*W)$ such that $d\pi'(T^H(\pi'^*W))\subset
	 	T^HW$.
	 	Let $\pi^{'*}\mF\otimes E'$ be the equivariant geometric family which is obtained from $\pi^{'*}\mF$ by twisting with $P_{W'}^*(\mS_{Z'}\otimes E')$, where $P_{W'}: W\times_B W'\rightarrow W'$.
	 	Then we have
	 	\begin{align}
	 	\mF\times_B \mF'\simeq \pi'!(\pi^{'*}\mF\otimes  E').
	 	\end{align} 
	 	Since the fibers of $\pi^{'*}W\rightarrow B$ is $Z'\times Z$, the fiberwise connection $\nabla^{T(Z'\times Z)}=\nabla^{TZ',TZ}$. So Theorem \ref{d198} follows from Theorem \ref{d188}.
	 	
	 	The proof of Theorem \ref{d198} is completed.
	 \end{proof}
	 
	 \begin{rem}\label{d114} 
	 	When the parameter space $B$ is a point and $\dim Z$ is odd, letting $A=P_{\ker D}$ be the orthogonal projection onto the kernel of $D(\mF)$, which we simply denote by $D$, we have
	 	\begin{multline}\label{bl0313} 
	 	\tilde{\eta}_g(\mF, A)
	 	=\frac{1}{\sqrt{\pi}}\int_0^{+\infty}\tr\left[g(D+(u\chi(u))'P_{\ker D})\exp(-(uD+u\chi(u)P_{\ker D})^2)\right]du
	 	\\
	 	=\frac{1}{\sqrt{\pi}}\int_0^{+\infty}\tr\left[g(D+(u\chi(u))'P_{\ker D})\exp(-u^2D-u^2\chi(u)^2P_{\ker D})\right]du
	 	\\
	 	=\frac{1}{\sqrt{\pi}}\int_0^{+\infty}\tr\left[gD\exp(-u^2D^2)\right]du+\frac{1}{\sqrt{\pi}}\int_0^{+\infty}\tr\left[g(u\chi(u))'P_{\ker D}\exp(-u^2\chi(u)^2P_{\ker D})\right]du
	 	\\
	 	=\frac{1}{2\sqrt{\pi}}\int_0^{+\infty}u^{-1/2}\tr[gD\exp(-uD^{2})]du+\frac{1}{\sqrt{\pi}}\int_0^{+\infty}\exp(-u^2)du\cdot\tr[gP_{\ker D}]
	 	\\
	 	=\frac{1}{2\sqrt{\pi}}\int_0^{+\infty}u^{-1/2}\tr[gD\exp(-uD^{2})]du+\frac{1}{2}\tr[gP_{\ker D}],
	 	\end{multline}
	 	which is just the usual \textbf{equivariant reduced eta invariant} in \cite{Do78}. So Theorem \ref{d188} naturally degenerates to the case of equivariant reduced eta invariants and the equivariant higher spectral flow degenerates to the canonical equivariant spectral flow \cite{F05}. 
	 \end{rem}
		
	\subsection{Proof of Lemma \ref{d030}}\label{s0206}
		
	The proof of Lemma \ref{d030} is almost the same as the proof of \cite[Theorem 3.4]{Liu17}.
	Observe that Assumptions 3.1 and 3.3 in \cite{Liu17}
	naturally hold in our case.
	
	Let $T'\geq 1$ be the constant taking in the proof of Lemma \ref{d089}.
	For $T\geq T'$, let $\mathbb{B}_{u,T}$ be the Bismut superconnection associated with the equivariant geometric family $\mF_{Z,T}$.
	Let
	\begin{align}\label{d074}
	\widehat{\mathbb{B}}|_{(T,u)}=\mathbb{B}_{u^2, T}+uT\chi(uT)(1\widehat{\otimes}A_X)
	+dT\wedge\frac{\partial}{\partial T}+du\wedge\frac{\partial}{\partial u}.
	\end{align}
	We define $\beta_g=du\wedge \beta_g^u+dT\wedge \beta_g^T$ to be the part of
	$\psi_{B}\widetilde{\tr}[g\exp(-\widehat{\mathbb{B}}^2)]$
	of degree one with respect to the coordinates $(T,u)$, with functions $\beta_g^u$, $\beta_g^T: \mathbb{R}_{+,T}\times\mathbb{R}_{+,u}
	\rightarrow \Omega^*(B,\C)$.
	
	Comparing with \cite[Proposition 4.2]{Liu17},
	there exists a smooth family $\alpha_g:\mathbb{R}_{+,T}\times\mathbb{R}_{+,u}\rightarrow\Omega^*(B,\C)$ such that
	\begin{align}\label{e03004}
	\left(du\wedge\frac{\partial}{\partial u}+dT\wedge\frac{\partial}{\partial T}\right)\beta_g=dT\wedge du\wedge d^B\alpha_g.
	\end{align}
	
	Take $\var, A, T_0$, $0<\var\leq 1\leq A<\infty$, $T'\leq T_0<\infty$. Let $\Gamma=\Gamma_{\var,A,T_0}$ be the oriented contour in
	$\mathbb{R}_{+,T}\times\mathbb{R}_{+,u}$.
	
	\
	
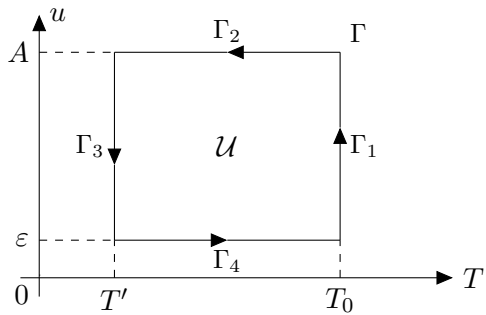
	\begin{center}\label{e03010}
		\begin{tikzpicture}
		\draw[->][ -triangle 45] (-0.25,0) -- (5.5,0);
		\draw[->][ -triangle 45] (0,-0.25) -- (0,3.5);
		\draw[->][ -triangle 45] (1,0.5) -- (2.5,0.5);
		\draw (2.5,0.5) -- (4,0.5);
		\draw[->][ -triangle 45] (1,3) -- (1,1.5);
		\draw (1,1.5) -- (1,0.5);
		\draw[->][ -triangle 45] (4,0.5) -- (4,2);
		\draw (4,2) -- (4,3);
		\draw[->][ -triangle 45] (4,3) -- (2.5,3);
		\draw (2.5,3) -- (1,3);
		\draw[dashed] (0,0.5) -- (1,0.5);
		\draw[dashed] (0,3) -- (1,3);
		\draw[dashed] (1,0) -- (1,0.5);
		\draw[dashed] (4,0) -- (4,0.5);
		\foreach \x in {0}
		\draw (\x cm,1pt) -- (\x cm,1pt) node[anchor=north east] {$\x$};
		\draw
		(2.5,1.75)  node {$\mU$}(2.5,1.75);
		\draw
		(0,3.5)  node[anchor=west] {$u$}(0,3.5);
		\draw
		(5.5,0)  node[anchor=west] {$T$}(5.5,0);
		\draw
		(0,0.5)  node[anchor=east] {$\varepsilon$}(0,0.5);
		\draw
		(0,3)  node[anchor=east] {$A$}(0,3);
		\draw
		(1,0)  node[anchor=north] {$T'$}(1,0);
		\draw
		(4,0)  node[anchor=north] {$T_0$}(4,0);
		\draw
		(4,1.75)  node[anchor=west] {\small{$\Gamma_1$}}(4,1.75);
		\draw
		(2.5,0.5)  node[anchor=north] {\small{$\Gamma_4$}}(2.5,0.5);
		\draw
		(2.5,3)  node[anchor=south] {\small{$\Gamma_2$}}(2.5,3);
		\draw
		(1,1.75)  node[anchor=east] {\small{$\Gamma_3$}}(1,1.75);
		\draw
		(4,3)  node[anchor=south west] {\small{$\Gamma$}}(4,3);
		\end{tikzpicture}
	\end{center}
	
	\
	
	The contour $\Gamma$ is made of four oriented pieces $\Gamma_1,\cdots,\Gamma_4$ indicated in the above picture.
	For $1\leq k\leq 4$,
	set $I_k^0=\int_{\Gamma_k}\beta_g$. Then by Stocks' formula and (\ref{e03004}),
	\begin{align}\label{e03011}
	\sum_{k=1}^4I_k^0=\int_{\partial \mU}\beta_g=\int_{\mU}\left(du\wedge\frac{\partial}{\partial u}
	+dT\wedge\frac{\partial}{\partial T}\right)\beta_g=d^B\left(\int_\mU \alpha_g dT\wedge du\right).
	\end{align}
	
	The following theorems are the analogues of  \cite[Theorems 4.3-4.6]{Liu17}. 	Note that Theorem \ref{e03018} is the analogue of \cite[(6.8)]{Liu17}.
	We will sketch the proofs in the next subsection.
	
	\begin{thm}\label{e03009}
		i) For any $u>0$, we have 
		\begin{align}\label{e03015}
		\lim_{T\rightarrow \infty}\beta_g^u(T,u)=0.
		\end{align}
		
		ii) For $0<u_1<u_2$ fixed, there exists $C>0$ such that, for $u\in [u_1,u_2]$, $T\geq 1$, we have
		\begin{align}\label{e03016}
		|\beta_g^u(T,u)|\leq C.
		\end{align}
		
		iii) We have the following identity:
		\begin{align}\label{e03017}
		\lim_{T\rightarrow +\infty} \int_{1}^{\infty}\beta_g^u(T,u)du=0.
		\end{align}
	\end{thm}
	
	\begin{thm}\label{e03018}
		For $u_0>0$ fixed, there exist $C, C'>0$, $T_0\geq 1$,
		such that for $u\geq u_0$, $T\geq T_0$, 
		\begin{align}\label{e03100}
		\left|\beta_g^T(T,u)\right|\leq C\exp(-C'u^2).
		\end{align}
	\end{thm}
	
	We know that $\mathrm{\widehat{A}}_g(TZ, \nabla)$ only depends on $g\in G$ and
	$R:=\nabla^2$.
	So we also denote it by $\mathrm{\widehat{A}}_g(R)$. Let
	$R_T^{TZ}:=(\nabla_T^{TZ})^2$.
	Set
	\begin{align}\label{e03005}
	\gamma_{\Omega}(T)=-\left.\frac{\partial}{\partial b}\right|_{b=0}\mathrm{\widehat{A}}_g\left(R_T^{TZ}
	+b\frac{\partial \nabla_T^{TZ}}
	{\partial T}\right).
	\end{align}
	By a standard argument in Chern-Weil theory, we know that
	\begin{align}\label{e03008}
	\frac{\partial}{\partial T}\widetilde{\mathrm{\widehat{A}}}_g(TZ, \nabla_{T'}^{TZ},
	\nabla^{TZ}_T)=-\gamma_{\Omega}(T).
	\end{align}
	
	\begin{thm}\label{e03047}
		When $T\rightarrow +\infty$, we have $\gamma_{\Omega}(T)=O(T^{-2})$. Moreover, modulo exact forms on $W^g$, we have
		\begin{align}\label{e03048}
		\widetilde{\mathrm{\widehat{A}}}_g(TZ, \nabla_{T'}^{TZ},
		\nabla^{TY,TX})=-\int_{T'}^{+\infty}\gamma_{\Omega}(T)dT.
		\end{align}
	\end{thm}
	
	Let
	$\mathbb{B}_{X,T}$ be the Bismut superconnection associated with the equivariant geometric family $\mF_{X,T}$, which 
	is the same as $\mF_X$ except for replacing $g^{TX}$ to $T^{-2}g^{TX}$.
	Set
	\begin{align}\label{d075}
	\gamma_1(T)=\left\{\psi_{V^g}\widetilde{\tr}|_{V^g}\left[g\exp\left(-\left(\mathbb{B}_{X,T^2}|_{V^g}+
	T\,\chi(T)A_X|_{V^g}+dT\wedge\frac{\partial }{\partial T}\right)^2\right)\right]\right\}^{dT}.
	\end{align}
	Then from (\ref{i18}),
	\begin{align}\label{d076}
	\widetilde{\eta}_g(\mF_X, A_X)=-\int_0^{\infty}\gamma_1(T) dT.
	\end{align}
	
	\begin{thm}\label{e03022}
		i) For any $u>0$, there exist $C>0$ and $\delta>0$ such that, for $T\geq T'$, we have
		\begin{align}\label{e03023}
		|\beta_g^T(T,u)|\leq\frac{C}{T^{1+\delta}}.
		\end{align}
		
		ii) For any $T>0$, we have
		\begin{align}\label{e03024}
		\lim_{\var\rightarrow 0}\var^{-1}\beta_g^T(T\var^{-1},\var)=
		\int_{Y^g}\td_g(\nabla^{TY}, \nabla^{L_Y})\, \gamma_1(T).
		\end{align}
		
		iii) There exists $C>0$ such that for $\var\in (0,1/T']$, $\var T' \leq T\leq 1$,
		\begin{align}\label{e03025}
		\var^{-1}\left|\beta_g^T(T\var^{-1},\var)
		+ \int_{Z^g}\gamma_{\Omega}(T\var^{-1})\,
		\ch_g(L_Y^{1/2}, \nabla^{L_Y^{1/2}})\,
		\ch_g(E, \nabla^{E})\right|
		\leq C.
		\end{align}
		Note that as in (\ref{bl0659}), $\ch_g(L_Y^{1/2}, \nabla^{L_Y^{1/2}})$ is well-defined even if $L_Y^{1/2}$
		does not exist.
		
		iv) There exist $\delta\in (0,1]$,  $C>0$ such that, for $\var\in (0,1]$, $T\geq 1$,
		\begin{align}\label{e03026}
		\var^{-1}|\beta_g^T(T\var^{-1},\var)|\leq \frac{C}{T^{1+\delta}}.
		\end{align}
	\end{thm}
		
	Now we prove Lemma \ref{d030} using the theorems above.
	
	By (\ref{e03011}), we know that
	\begin{align}\label{e03027}
	\int_{\var}^{A}\beta_g^u(T_0,u)du-\int_{T'}^{T_0}\beta_g^T(T,A)dT-\int_{\var}^{A}\beta_g^u(T',u)du+
	\int_{T'}^{T_0}\beta_g^T(T,\var)dT
	=\sum_{k=0}^4I_k^0
	\end{align}
	is an exact form.
	We take the limits $A\rightarrow+\infty$, $T_0\rightarrow+\infty$ and then $\var\rightarrow 0$
	in the indicated order. Let $I_j^k$, $j=1,2,3,4$, $k=1,2,3$, denote the value of the part $I_j$ after
	the $k$th limit. 
	
	Since the definition of the equivariant eta form does not depend on the cut-off function,
	from (\ref{i17}), we obtain that modulo exact forms on $B$,
	\begin{align}\label{e03029}
	I_3^3=\tilde{\eta}_{g}(\mF_{Z, T'}, 1\widehat{\otimes}T'A_X).
	\end{align}
	Furthermore, by Theorem \ref{e03018}, we get
	\begin{align}\label{e03030}
	I_2^3=I_2^2=0.
	\end{align}
	From Theorem \ref{e03009}, we have
	\begin{align}\label{e03031}
	I_1^3
	=0.
	\end{align}
	
	Finally, using Theorem \ref{e03022}, we get
	\begin{align}\label{e03032}
	I_4^3=-\int_{Y^g}\td_g(\nabla^{TY}, \nabla^{L_Y})\, \widetilde{\eta}_g(\mF_{X},A_X)
	+
	\widetilde{\mathrm{FLI}}_g\left(\nabla_{T'}^{TZ}, \nabla^{L_Z}, \nabla^{TY,TX}, \nabla^{L_Z}\right)
	\end{align}
	as follows: We write
	\begin{align}\label{e03033}
	\begin{split}
	\int_{T'}^{+\infty}\beta_g^T(T,\var)dT
	=\int_{\var T'}^{+\infty}\var^{-1}\beta_g^T(T\var^{-1},\var)dT.
	\end{split}
	\end{align}
	Convergence of the integrals above is guaranteed by (\ref{e03023}).
	Using Theorem \ref{e03047} and (\ref{e03024})-(\ref{e03026}), we get
	\begin{align}\label{e03034}
	\lim_{\var\rightarrow 0}\int_1^{+\infty}\var^{-1}\beta_g^T(T\var^{-1},\var)dT
	=\int_{Y^g}\td_g(\nabla^{TY}, \nabla^{L_Y})\,
	\int_1^{+\infty}\gamma_1(T)dT
	\end{align}
	and
	\begin{multline}\label{e03035}
	\lim_{\var\rightarrow 0}\int_{\var T'}^1\var^{-1}\left[\beta_g^T(T\var^{-1},\var)dT
	+\int_{Z^g}\gamma_{\Omega}(T\var^{-1})\,
	\ch_g(L_Y^{1/2}, \nabla^{L_Y^{1/2}})\,
	\ch_g(E, \nabla^{E})\right]dT
	\\
	=\int_{Y^g}\td_g(\nabla^{TY}, \nabla^{L_Y})\,
	\int_0^{1}\gamma_1(T)dT.
	\end{multline}
	The remaining part of the integral yields by (\ref{e03025})
	\begin{multline}\label{e03037}
	\lim_{\var\rightarrow 0}\int_{\var T'}^1\var^{-1}\int_{Z^g}\gamma_{\Omega}(T\var^{-1})\,
	\ch_g(L_Y^{1/2}, \nabla^{L_Y^{1/2}})\,
	\ch_g(E, \nabla^{E})dT
	\\
	=\int_{Z^g}\int_{T'}^{+\infty}\gamma_{\Omega}(T)\,
	\ch_g(L_Y^{1/2}, \nabla^{L_Y^{1/2}})
	\ch_g(E, \nabla^{E})dT
	=\widetilde{\mathrm{FLI}}_g\left(\nabla^{TY,TX}, \nabla^{L_Z}, \nabla_{T'}^{TZ}, \nabla^{L_Z}\right).
	\end{multline}
	These four equations for $I_k^3$, $k=1,2,3,4$, (\ref{e03011}) and  (\ref{e03027}) imply that
	\begin{align}
\lim_{\var\to 0}\lim_{T_0\to+\infty}\lim_{A\to +\infty}d^B\left( \int_{T'}^{T_0}\int_{\var}^A\alpha_g dT\wedge du\right)
	\end{align}
	exists and equal to
	\begin{multline}\label{eq:2.131}
	\Theta:=\widetilde{\eta}_g(\mF_{Z, T'}, 1\widehat{\otimes}TA_{X})
	-\int_{Y^g}\td_g(\nabla^{TY}, \nabla^{L_Y})\, \widetilde{\eta}_g(\mF_{X},A_X)
	\\-
	\widetilde{\mathrm{FLI}}_g\left(\nabla^{TY,TX}, \nabla^{L_Z}, \nabla_{T'}^{TZ}, \nabla^{L_Z} \right)\in \Omega^*(B).
	\end{multline}
	
	Since the convergences for $A\rightarrow+\infty$, $T_0\rightarrow+\infty$ and $\var\rightarrow 0$
	are uniform on compact manifold $B$, they commute with
	the integration on $B$. So
	for any closed form $\theta\in \Omega^*(B)$,
	$\int_B\Theta\wedge\theta=0$.
	By \cite[\S 22, Theorem 17']{dR73}, there exists a
	current $\mT$ such that $\Theta=d\mT$. Since
	$\Theta\in \Omega^*(B)$ is smooth,
	we have 
	$\Theta\in d\Omega^*(B)$. 
	So the right hand side of (\ref{eq:2.131}) is an exact form on $B$.  Therefore we obtain Lemma \ref{d030}.
	
	
	\subsection{Proofs of Theorems \ref{e03009}-\ref{e03022}}\label{s0207}
		
	Since
	$\ker (D(\mF_X)+A_X)=0$, the proofs of Theorems \ref{e03009}-\ref{e03022} in our case are much easier than those in \cite{Liu17}.
	We only need to replace $D^X$ and $D_T^Z$ somewhere in \cite{Liu17} by $D(\mF_X)+A_X$ and $D(\mF_{Z,T})+1\widehat{\otimes}TA_X$ and take care with the local index computation in the proof of Theorem \ref{e03022} ii).
	 In this subsection, we only sketch the local index part here. 
	
	Set (cf. \cite[(7.1)]{Liu17})
	\begin{align}\label{e06004}
	\mB_{\var, T/\var}'
	=
	(\mathbb{B}_{\var^2,T/\var}+T\chi(T)A_X)^2+\var^{-1}dT\wedge \left.\frac{\partial (\mathbb{B}_{\varepsilon^2,T'}+\var T'\chi(\var T')A_X)}{\partial T'}\right|_{T'=T\varepsilon^{-1}}.
	\end{align}
	By the definition of $\beta_g^T(T,\var)$, we have
	\begin{align}\label{e06005}
	\var^{-1}\beta_g^T(T/\var,\var)=\left\{\psi_B\widetilde{\tr}[g\exp(-\mB_{\var, T/\var}')]\right\}^{dT}.
	\end{align}
	Let $S_X$ be the tensor in (\ref{e01017}) with respect to $\pi_X$.
	Let $\{e_i\}$, $\{f_p\}$ and $\{g_{\alpha}\}$ be the local orthonormal frames of $TX$, $TY$ and $TB$ and $\{f_{p,1}^H\}$ and $\{g_{\alpha,3}^H\}$ be the corresponding horizontal lifts.
	Precisely, by (\ref{e01032}), we have
	\begin{multline}\label{e06002}
	\var^{-1} \left.\frac{\partial (\mathbb{B}_{\varepsilon^2,T'}+\var T'\chi(\var T')A_X)}{\partial T'}\right|_{T'=T\varepsilon^{-1}}
	=D^X+\chi(T)A_X+T\chi'(T)A_X
	\\
	-\frac{1}{8T^2}\big(\la \var^2[f_{p,1}^H, f_{q,1}^H],e_i\ra c(e_i)c(f_{p,1}^H)c(f_{q,1}^H)
	\\
	+4\var \la S_X(g_{\alpha,3}^H)e_i, f_{p,1}^H \ra c(e_i)c(f_{p,1}^H) g^{\alpha}_3\wedge
	+\la [g_{\alpha,3}^H, g_{\beta,3}^H],e_i\ra c(e_i)g^{\alpha}\wedge g^{\beta}\wedge\big).
	\end{multline}
	
	As in (\ref{e06004}), we set
	\begin{align}\label{e05013}
	\mB_{T^2}''|_{V^g}=(\mathbb{B}_{X,T^2}|_{V^g}+ T\chi(T)A_X)^2+dT\wedge \left.\frac{\partial (\mathbb{B}_{X,T^2}+ T\chi(T)A_X)}{\partial T}\right|_{V^g}.
	\end{align}
	Then by (\ref{d075}), we have
	\begin{align}\label{e06051}
	\gamma_1(T)=\left\{\psi_{V^g}\widetilde{\tr}[g\exp(-\mB_{ T^2}''|_{V^g})]\right\}^{dT}.
	\end{align}
	
	As the same process in \cite[Section 7]{Liu17}, we could localize the problem near $\pi_X^{-1}(V^g)$ and define the operator $\mB_{\var,T/\var}'$ on a neighborhood of $\{0\}\times X_{y_0}$ in $T_{y_0}Y\times X_{y_0}$.
	
	Let $\mathrm{dist}^V$, $\mathrm{dist}^W$ be the distance functions on $V$, $W$ associated with $g^{TV}$, $g^{TW}$.
	Let $\mathrm{Inj}^{V}$, $\mathrm{Inj}^{W}$ be the injective radius of $V$, $W$.
	In the sequel, we assume that given $0<\alpha<\alpha_0<\inf\{\mathrm{Inj}^V,\mathrm{Inj}^W\}$
	are chosen small enough so that if $y\in V$, $\mathrm{dist}^V(g^{-1}y,y)\leq \alpha$,
	then $\mathrm{dist}^V(y,V^g)\leq\frac{1}{4}\alpha_0$, and if $z\in W$, $\mathrm{dist}^W(g^{-1}z,z)\leq \alpha$,
	then $\mathrm{dist}^W(z,W^g)\leq\frac{1}{4}\alpha_0$.
	Let $\rho:T_{y_0}Y\rightarrow [0,1]$ be a smooth function such that
	\begin{align}\label{e06017}
	\rho(U)=
	\left\{
	\begin{array}{ll}
	1, & \hbox{$|U|\leq \alpha_0/4$;} \\
	0, & \hbox{$|U|\geq \alpha_0/2$.}
	\end{array}
	\right.
	\end{align}
	Let $\Delta^{TY}$ be the ordinary Laplacian operator on $T_{y_0}Y$. Let $\cE_{Z,y_0}:=\cC^{\infty}(X_{y_0}, \mS_Z\widehat{\otimes}E|_{X_{y_0}})$.
	
	Set
	\begin{align}\label{e06019}
	L_{\var,T}^1=(1-\rho^2(U))(-\var^2\Delta^{TY}+T^2(D^X+A_X)_{y_0}^{2})+\rho^2(U)\mB_{\var,T/\var}'
	\end{align}
	on $\pi_2^*\Lambda(T^*S)\widehat{\otimes}\cC^{\infty}(T_{y_0}Y,
	\cE_{Z,y_0})$.
	For $(U,x)\in N_{Y^g/Y, y_0}\times X_{y_0}$, $|U|<\alpha_0/4$, $\var>0$, set
	\begin{align}\label{e566}
	(S_{\var}s)(U, x)=s\left(U/\var,x\right).
	\end{align}
	Put
	\begin{align}\label{e567}
	\begin{split}
	L_{\var,T}^2:=S_{\var}^{-1}L_{\var,T}^1S_{\var}.
	\end{split}
	\end{align}
	
	Set $\dim T_{y_0}Y^g=l'$ and $\dim N_{Y^g/Y,y_0}=2l''$.
	Let $\{f_1,\cdots,f_{l'}\}$ be an orthonormal basis of $T_{y_0}Y^g$ and let $\{f_{l'+1},\cdots,f_{l'+2l''}\}$ be an orthonormal basis
	of $N_{Y^g/Y, y_0}$.
	Let $R_{\var}$ be a rescaling operator such that
	\begin{align}\label{e568}
	\begin{split}
	&R_{\var}(c(e_i))=c(e_i),
	\\
	&R_{\var}(c(f_{p,1}^H))=\frac{f^{p}\wedge}{\var}-\var\, i_{f_{p}}, \quad {\rm for}\ 1\leq p\leq l',
	\\
	&R_{\var}(c(f_{p,1}^H))=c(f_{p,1}^H), \quad {\rm for}\ l'+1\leq p\leq l'+2l''.
	\end{split}
	\end{align}
	Then $R_{\var}$ is a Clifford algebra homomorphism. Set
	\begin{align}\label{e569}
	L_{\var,T}^3=R_{\var}(L_{\var,T}^2)
	\end{align}
	on $\pi_2^*\Lambda(T^*S)\widehat{\otimes} \Lambda(T_{y_0}^*Y^g)
	\widehat{\otimes}\cC^{\infty}(T_{y_0}Y, \cE_{X,N,y_0})$,
	where $\cE_{X,N,y_0}:=\cC^{\infty}(X_{y_0}, \pi_2^*\mS_N\widehat{\otimes}\mS_X\widehat{\otimes}E|_{X_{y_0}})$
	and $\mS_N$ is the spinor for $N_{Y^g/Y, y_0}$.
	
	Corresponding to \cite[Lemma 4.4]{Liu17}, from (\ref{e06005})-(\ref{e05013}), we have
	\begin{lemma}
		When $\var\rightarrow 0$, the limit $\displaystyle L_{0,T}^3=\lim_{\var\rightarrow 0}L_{\var,T}^3$ exists
		in the sense of \cite[(7.108)]{Liu17}
		 and
		\begin{align}\label{e594}
		L_{0,T}^3|_{V^g}=-\left(\partial_{p}+\frac{1}{4}\la R^{TY}|_{V^g}U,f_{p,1}^H\ra\right)^2+\frac{1}{2}R^{L_Y}|_{V^g}+\mB_{T^2}''|_{V^g}.
		\end{align}
	\end{lemma}
	
	So all the computations in our case are the same as \cite[Section 7]{Liu17}.
		
	\section{Equivariant differential K-theory}
	
	In this section, we assume that the $G$-action on $B$ has finite stabilizers only, i.e., for any $b\in B$, $G_b:=
	\{g\in G: gb=b \}$ is finite. 
	With this action, we construct an analytic model of equivariant differential K-theory and prove some properties using the results in Section 2.
	
	\subsection{Definition of equivariant differential K-theory}
	
	In this subsection, we construct an analytic model of equivariant differential K-theory. When $G=\{e\}$, this construction is similar as that in \cite{BS09} except replacing the taming and KK-theory to the spectral section and higher spectral flow.
	
	Let $E$ be a $G$-equivariant  complex vector bundle over $B$. Then its restriction to $B^g$ is acted on fibrewise by $g$ for $g\in G$. So it decomposes as a direct sum of subbundles $E_{v}$ for each eigenvalue $v$ of $g$. Set $\phi_g(E):=\sum v E_{v}$.
	Then it induces a homomorphism (for $K_G^1$, replacing $B$ by $B\times S^1$
	and use (\ref{d124}))
	\begin{align}\label{d119}
	\phi_g: K_G^*(B)\otimes\mathbb{C}\longrightarrow[ K^*(B^g)\otimes\mathbb{C}]^{\mathrm{C}_G(g)},
	\end{align}
	where $\mathrm{C}_G(g)$ is the centralizer of $g$ in $G$.
	Let $(g)$ be the conjugacy class of $g\in G$.  For $g, g'\in (g)$, 
	there exists $h\in G$, such that $g'=h^{-1}gh$. Furthermore, the map
	\begin{align}\label{d190}
	h:B^{g'}/\mathrm{C}_G(g')\rightarrow B^{g}/\mathrm{C}_G(g)
	\end{align}
	is a homeomorphism .
	So 
	\begin{align}\label{eq:3.03}
	[ K^*(B^{g})\otimes\mathbb{C}]^{\mathrm{C}_G(g)}\simeq [ K^*(B^{g'})\otimes\mathbb{C}]^{\mathrm{C}_G(g')}.
	\end{align}
	By \cite[Corollary 3.13]{ALR07}, we know that
		the 	
		additive decomposition
		\begin{align}\label{d113}
		\phi=\oplus_{(g),g\in G}\phi_g:
		K_G^*(B)\otimes\mathbb{C}\rightarrow \bigoplus_{(g),g\in G}[ K^*(B^g)\otimes\mathbb{C}]^{\mathrm{C}_G(g)}
		\end{align}
		is an isomorphism, where $(g)$ ranges over the conjugacy classes of $G$.
	
If $B^g\neq \emptyset$, then there exists $b\in B^g$ such that
$g\in G_b$. The conjugacy class of $G_b$ is the type of the 
orbit $G\cdot b$. Since $B$ is compact, there are only finitely
many orbit types. Since all stablizers are finite groups, we see that
the direct sum in (\ref{d113}) only has finite terms.
From the isomorphism (\ref{eq:3.03}), the direct sum in (\ref{d113})
does not depend on the choice of the element in $(g)$ in the sense
of (\ref{eq:3.03}).
		
	From (\ref{d190}), we also know that the map $h^*$ induces an isomorphism
	\begin{align}\label{d160}
	h^*:[\Omega^*(B^g,\C)]^{\mathrm{C}_G(g)}\rightarrow [\Omega^*(B^{g'},\C)]^{\mathrm{C}_G(g')}.
	\end{align}
	We denote by
	\begin{align}\label{d105}
	\Omega_{deloc, G}^*(B,\C):=\bigoplus_{(g),g\in G} \left\{[\Omega^*(B^g,\C)]^{\mathrm{C}_G(g)}\right\},
	\end{align}
	the set of delocalized differential forms, where $\{\cdot\}$ denotes the isomorphic class in sense of (\ref{d160}). The definition above
	does not depend on the choice of $g\in (g)$.
	It is easy to see that the exterior differential operator $d$ preserves $\Omega_{deloc, G}^*(B,\C)$.
	We denote by
	the delocalized de Rham cohomology $H^*_{deloc, G}(B,\C)$  the cohomology of the differential complex $(\Omega_{deloc, G}^*(B,\C), d)$.		
	Then from (\ref{d119}) and (\ref{d113}), the equivariant Chern character isomorphism can be naturally defined by
	\begin{align}\label{d116}
	\begin{split}
	\ch_{G}:
	K_G^*(B)\otimes\mathbb{C}&\overset{\simeq}{\longrightarrow} H^*_{deloc, G}(B,\C),
	\\
	\mathcal{K}\quad&\mapsto \bigoplus_{(g),g\in G} \left\{\ch(\phi_g(\mathcal{K}))\right\}.
	\end{split}
	\end{align}
	We note that $\ch(\phi_g(\mathcal{K}))=\ch_g(\mathcal{K})$ is $\mathrm{C}_G(g)$-invariant by the definition.
	
	Observe that the fixed point set
	for $g$-action coincides with that for $g^{-1}$-action. Set 
	\begin{align}\label{d102}
	H^*_{deloc, G}(B,\R):=\{c=\oplus_{(g),g\in G}\{c_{g}\}\in H^*_{deloc, G}(B,\C): \forall g\in G, c_{g^{-1}}=\overline{c_{g}}\}.
	\end{align}
	Let $\Omega^*_{deloc, G}(B,\R)\subset \Omega^*_{deloc, G}(B,\C)$ be the ring of forms $\omega=\oplus_{(g),g\in G}\{\omega_{g}\}$, such that $\forall g\in G, \omega_{g^{-1}}=\overline{\omega_{g}}$.
	Then $H^*_{deloc, G}(B,\R)$ is the cohomology of the differential complex $(\Omega^*_{deloc, G}(B,\R), d)$.
Since $\ch(\phi_{g^{-1}}(\mathcal{K}))=\overline{\ch(\phi_{g}(\mathcal{K}))}$, from (\ref{d116}), for any $\mathcal{K}\in K_G^*(B)$, 
$\ch_G(\mathcal{K})\in H_{deloc, G}^*(B,\R)$. 
Thus $\ch_G(K_G^*(B)\otimes\R)\subseteq H_{deloc, G}^*(B,\R)$.
Since (\ref{d116}) is an isomorphism, we obtain a group isomorphism
	\begin{align}\label{d120}
	\begin{split}
	\ch_{G}:K_{G}^*(B)\otimes\mathbb{R}\overset{\simeq}{\longrightarrow} H^*_{deloc, G}(B,\R).
	\end{split}
	\end{align}
	
	\begin{defn}\label{d043}(Compare with \cite[Definition 2.4]{BS09})
		A cycle for an equivariant differential K-theory class over $B$
		is a pair $(\mF, \rho)$, where $\mF\in \mathrm{F}_G^*(B)$ and $\rho\in \Omega^*_{deloc, G}(B,\R)/\Im \,d$.
		The cycle $(\mF, \rho)$ is called even (resp. odd) if $\mF$ is even (resp. odd) and
		$\rho\in \Omega^{\mathrm{odd}}_{deloc, G}(B,\R)/\Im \,d$ (resp. $\rho\in \Omega^{\mathrm{even}}_{deloc, G}(B,\R)/\Im \,d$).
		Two cycles ($\mF, \rho$) and ($\mF', \rho'$) are called isomorphic if $\mF$ and $\mF'$ are isomorphic
		and $\rho=\rho'$. Let $\widehat{\mathrm{IC}}_G^0(B)$ (resp. $\widehat{\mathrm{IC}}_G^1(B)$) denote the set of isomorphic classes of even (resp. odd) cycles over $B$ with a natural abelian semi-group structure by $(\mF,\rho)+(\mF',\rho')=(\mF+\mF',\rho+\rho')$.
	\end{defn}
	
	For $\mF\in \mathrm{F}_G^*(B)$, we assume that there exists a perturbation operator $A$ with respect to $D(\mF)$.
	For any $g\in G$, by Definition \ref{d017}, the equivariant eta form restricted on the fixed point set of $g$ is $\mathrm{C}_G(g)$-invariant, that is,  $\widetilde{\eta}_g(\mF|_{B^g},A|_{B^g})\in [\Omega^*(B^g, \C)]^{\mathrm{C}_G(g)}$. Let $h^*$ be the map in (\ref{d160}). Since the perturbation operator $A$ is equivariant, from Definition \ref{d017}, we have
	\begin{align}\label{d121}
	\widetilde{\eta}_{g'}(\mF|_{B^{g'}},A|_{B^{g'}})=h^*\widetilde{\eta}_g(\mF|_{B^g},A|_{B^g}).
	\end{align}
	From Definition \ref{d017}, $\widetilde{\eta}_{g^{-1}}(\mF|_{B^g},A|_{B^g})=\overline{\widetilde{\eta}_g(\mF|_{B^g},A|_{B^g})}$.
So	the following definition is well-defined.
	
	\begin{defn}\label{d122}
		The delocalized eta form $\widetilde{\eta}_G(\mF,A)$
		is defined by
		\begin{align}\label{d046}
		\widetilde{\eta}_G(\mF,A)=\bigoplus_{(g),g\in G}\left\{\widetilde{\eta}_g(\mF|_{B^g},A|_{B^g})\right\}\in \Omega^*_{deloc, G}(B,\R)/\Im\, d.
		\end{align}	
	\end{defn}
	
	By the same process, we can define
	\begin{align}\label{d059}
	\mathrm{FLI}_G(\mF)=\bigoplus_{(g), g\in G}\left\{\mathrm{FLI}_g(\mF)\right\}\in \Omega^*_{deloc, G}(B,\R).
	\end{align}
	From (\ref{d019}), we have
	\begin{align}\label{eq:3.12}
	d\widetilde{\eta}_G(\mF,A)=\mathrm{FLI}_G(\mF).
	\end{align}
	
	Let $\mF\in \mathrm{F}_G^*(B)$ and $A$ be a perturbation operator with respect to $D(\mF)$. Then by Definition \ref{d041}, there exists a perturbation operator $A^{\mathrm{op}}$ with respect to $D(\mF^{\mathrm{op}})$ such that
	\begin{align}\label{d140}
	\widetilde{\eta}_G(\mF^{\mathrm{op}},A^{\mathrm{op}})=-\widetilde{\eta}_G(\mF,A).
	\end{align}
	Let $\mF, \mF'\in \mathrm{F}_G^*(B)$, $A$, $A'$ be perturbation operators with respect to $D(\mF)$, $D(\mF')$ respectively. By Definition \ref{d017}, we have
	\begin{align}\label{d145}
	\widetilde{\eta}_G(\mF+\mF', A\sqcup_B A')=\widetilde{\eta}_G(\mF,A)+\widetilde{\eta}_G(\mF',A').
	\end{align}
   From Remark \ref{d134} , we know that for any $\mF\in \mathrm{F}_G^*(B)$, 
	there exists a perturbation operator $A$ with respect to $D(\mF+\mF^{\mathrm{op}})$ and $A=A^{\mathrm{op}}$.
	From (\ref{d140}), we have
	\begin{align}\label{d141}
	\widetilde{\eta}_G(\mF+\mF^{\mathrm{op}},A)=0.
	\end{align}
	
	\begin{defn}\label{d049}(Compare with \cite[Definition 2.10]{BS09})
		We call two cycles ($\mF, \rho$) and ($\mF', \rho'$) \textbf{paired} if $\ind(D(\mF))=\ind(D(\mF'))$, and there exists
		a perturbation operator $A$ with respect to $D(\mF+\mF^{'\mathrm{op}})$ such that
		\begin{align}\label{d050}
		\rho-\rho'= \widetilde{\eta}_G(\mF+\mF^{'\mathrm{op}}, A).
		\end{align}
		
	\end{defn}
	
	From (\ref{d140})-(\ref{d141}), we have
	\begin{lemma}\label{d139}(Compare with \cite[Lemmas 2.11, 2.12]{BS09})
		The relation "paired" is symmetric, reflexive and compatible with the semigroup structure on $\widehat{\mathrm{IC}}_G^*(B)$.
	\end{lemma}
		
	\begin{defn}\label{d051}(Compare with \cite[Definition 2.14]{BS09})
		Let $\sim$ denote the equivalence relation generated by the relation "paired".	
		The equivariant differential K-group $\widehat{K}_G^0(B)$ (resp. $\widehat{K}_G^{1}(B)$) is the group completion
		of the abelian semigroup $\widehat{\mathrm{IC}}_G^{\mathrm{even}}(B)/\sim$ (resp. $\widehat{\mathrm{IC}}_G^{\mathrm{odd}}(B)/\sim$).
	\end{defn}
	
	
	If $(\mF, \rho)\in \widehat{\mathrm{IC}}_G^{*}(B)$, we denote by $[\mF, \rho]\in \widehat{K}_G^*(B)$ the corresponding class in equivariant
	differential K-group. From (\ref{d140})-(\ref{d141}), for any $[\mF, \rho],[\mF', \rho']\in \widehat{K}_G^*(B)$, we have
	\begin{align}\label{d146}
	[\mF, \rho]=[\mF+\mF^{'\mathrm{op}}, \rho-\rho']+[\mF', \rho'].
	\end{align}
	So every element of $\widehat{K}_G^*(B)$ can be represented in the form $[\mF, \rho]$. Furthermore, we have	
	$-[\mF, \rho]=[\mF^{\mathrm{op}}, -\rho]$. 
	
	\subsection{Push-forward map}
	
	In this subsection, we construct a well-defined  push-forward map in equivariant differential K-theory and prove the functoriality of it using the theorems in Section 2. This solves a question proposed in \cite{BS09} when $G=\{e\}$. We use the notation in Section 1.4.
		
	Let $\pi_Y: V\rightarrow B$ be an equivariant smooth surjective proper submersion of compact $G$-manifolds with compact orientable fibers $Y$.
	We assume that the $G$-action on $B$ has finite stabilizers only. 
	Thus, so is the action on $V$. We assume that $TY$ is oriented 
	and $\pi_Y$ has an equivariant K-orientation in Definition
	\ref{d154}.
    
    For $g\in G$, the fixed point set $V^g$ is the total space of the fiber bundle $\pi_Y|_{V^g}: V^g\rightarrow B^g$ with fibers $Y^g$. Since the pullback isomorphism $h^*$ in (\ref{d160}) commutes with the integral along the fiber, for $\alpha=\oplus_{(g),g\in G} \{\alpha_{g}\}\in \Omega^{*}_{deloc, G}(V, \R)$, the integral
    \begin{align}\label{d193}
    \int_{Y,G}\alpha:=\bigoplus_{(g),g\in G}\left\{\int_{Y^g} \alpha_g\right\}\in\Omega^{*}_{deloc, G}(B, \R)
    \end{align} 
    does not depend on $g\in (g)$. So it defines an integral map
    \begin{align}
    \int_{Y,G}: \Omega^{*}_{deloc, G}(V, \R)\rightarrow \Omega^{*}_{deloc, G}(B, \R).
    \end{align} 
	
	Consider the set $\widehat{\mathcal{O}}_G^*(\pi_Y)$ of equivariant
	geometric data $\widehat{o}_Y=(T_{\pi_Y}^HV, g^{TY}, \nabla^{L_Y}, \sigma_Y)$, where $\sigma_Y\in \Omega^{odd}_{deloc, G}(V)/\Im d$.
		
	Let
	\begin{align}\label{d080}
	\td_G(\nabla^{TY}, \nabla^{L_Y}):=\bigoplus_{(g),g\in G}\left\{\td_g(\nabla^{TY}, \nabla^{L_Y})\right\}\in \Omega_{deloc, G}^*(V, \R).
	\end{align}
	Let $\widehat{o}_Y'=(T^{'H}_{\pi_Y}V, g^{'TY}, \nabla^{'L_Y}, \sigma_Y')\in \widehat{\mathcal{O}}_G^*(\pi_Y)$ be another equivariant tuple with the same equivariant K-orientation in Definition \ref{d154}.
	As in (\ref{d182}),
	from \cite[Theorem B.5.4]{MM07}, we can construct the Chern-Simons form
	$\widetilde{\td}_G(\nabla^{TY}, \nabla^{L_Y}$, $\nabla^{'TY}, \nabla^{'L_Y})\in\Omega^{odd}_{deloc, G}(V)/\Im d$ such that
	\begin{align}\label{d078}
	d\,\widetilde{\td}_G(\nabla^{TY}, \nabla^{L_Y}, \nabla^{'TY}, \nabla^{'L_Y})=\td_G(\nabla^{'TY}, \nabla^{'L_Y})-\td_G(\nabla^{TY}, \nabla^{L_Y}).
	\end{align}
	
	We introduce a relation $\widehat{o}_Y\sim \widehat{o}_Y'$ as in \cite{BS09}: two equivariant tuples
	$\widehat{o}_Y$, $\widehat{o}_Y'$ are related if and only if
	\begin{align}\label{d079}
	\sigma_Y'-\sigma_Y=\widetilde{\td}_G(\nabla^{TY}, \nabla^{L_Y}, \nabla^{'TY}, \nabla^{'L_Y}),
	\end{align}
	where we mark the objects associated with the second tuple by $'$.
	
	\begin{defn}\label{d081}(Compare with \cite[Definition 3.5]{BS09})
		The set of equivariant differential K-orientations is the set of equivalence classes $\widehat{\mathcal{O}}_G^*(\pi_Y)/\sim$.
	\end{defn}
	
	We now start with the construction of the push-forward map $\widehat{\pi}_Y!:\widehat{K}_G^*(V)\rightarrow \widehat{K}_G^*(B)$
	for a given equivariant differential K-orientation which extends Theorem \ref{d157} to the differential case.
	For $[\mF_X, \rho]\in \widehat{K}_G^{*}(V)$,
	let $\mF_Z$ be the equivariant geometric family defined in (\ref{d073}). 
	We define (cf. \cite[(17)]{BS09})
	\begin{multline}\label{d082}
	\widehat{\pi}_Y!([\mF_X, \rho])=\left[\mF_Z, \int_{Y,G} \td_G(\nabla^{TY}, \nabla^{L_Y})\wedge \rho+\widetilde{\mathrm{FLI}}_G\left(\nabla^{TY,TX}, \nabla^{L_Z}, \nabla^{TZ}, \nabla^{L_Z}\right)\right.
	\\\left.	
	+ \int_{Y,G} \sigma_Y\wedge (\mathrm{FLI}_G(\mF_X)-d\rho) \right]\in \widehat{K}_G^*(B),
	\end{multline}
	where $\widetilde{\mathrm{FLI}}_G:=\bigoplus_{(g),g\in G}\widetilde{\mathrm{FLI}}_g\in \Omega_{deloc, G}^*(B, \R)/\Im d$.
	
	\begin{thm}\label{d083}(Compare with \cite[Lemma 3.14]{BS09})
		The map $\widehat{\pi}_Y!:\widehat{K}_G^*(V)\rightarrow \widehat{K}_G^*(B)$ in (\ref{d082}) is well-defined.
	\end{thm}
	\begin{proof}
		Let $(\mF_X, \rho)$, $(\mF_X', \rho')$ be two cycles over $V$.
		By (\ref{d082}), we have
		\begin{align}\label{d084}
		\widehat{\pi}_Y!(\mF_X, \rho)-\widehat{\pi}_Y!(\mF_X', \rho')=\widehat{\pi}_Y!(\mF_X+\mF_X^{'\mathrm{op}}, \rho-\rho').
		\end{align}
		If $(\mF_X, \rho)$ is paired with $(\mF_X', \rho')$, there exists a perturbation operator $A$, such that
		\begin{align}\label{d085}
		\rho-\rho'=\widetilde{\eta}_G(\mF_X+\mF_X^{'\mathrm{op}}, A).
		\end{align}
		So we only need to prove that if there exists a perturbation operator $A_X$ with respect to $D(\mF_X)$,
		 $\widehat{\pi}_Y!([\mF_X, \widetilde{\eta}_G(\mF_X, A_X)])=0\in \widehat{K}_G^*(B)$.
		
		From (\ref{d082}), we have
		\begin{multline}\label{d086}
		\widehat{\pi}_Y!([\mF_X, \widetilde{\eta}_G(\mF_X, A_X)])=\left[\mF_Z, \int_{Y,G} \td_G(\nabla^{TY}, \nabla^{L_Y})\,
		\widetilde{\eta}_G(\mF_X, A_X)\right.
		\\
		+\widetilde{\mathrm{FLI}}_G\left(\nabla^{TY,TX}, \nabla^{L_Z}, \nabla^{TZ}, \nabla^{L_Z}\right)
		+ \left.\int_{Y,G} \sigma_Y\wedge (\mathrm{FLI}_G(\mF_X)-d \widetilde{\eta}_G(\mF_X, A_X)) \right].
		\end{multline}
		From Proposition \ref{d006} (iii) and Lemma \ref{d089}, there exists a perturbation operator $A_Z$ with respect to $D(\mF_Z)$.
		By Theorem \ref{d188}, (\ref{eq:3.12}), (\ref{d145}) and (\ref{d193}), there exists $x\in K_G^*(B)$ such that
		\begin{align}\label{d087}
		\widehat{\pi}_Y!(\mF_X, \widetilde{\eta}_G(\mF_X, A_X))=\left[\mF_Z,
		\widetilde{\eta}_G(\mF_Z, A_Z)-\ch_G(x) \right].
		\end{align}
				From Proposition \ref{d123}, 
		if $x\in K_G^1(B)$,there exist $\mF\in \mathrm{F}_G^1(B)$  and equivariant spectral sections $P$, $Q$ with respect to $D(\mF)$, such that  $[P-Q]=x$. Let $A_P$, $A_Q$ be perturbation operators associated with $P$, $Q$ respectively. From Theorem \ref{d176}, we have
		\begin{align}\label{d115} 
		\ch_G(x)=\widetilde{\eta}_G(\mF, A_P)-\widetilde{\eta}_G(\mF, A_Q).
		\end{align} 
		If $x\in K_G^0(B)$, by Proposition \ref{d123}, there exist $\mF_1, \mF_2\in \mathrm{F}_G^0(B)$  and equivariant spectral sections $P_1$, $Q_1$ of $D(\mF_1)$ and $P_2$, $Q_2$ of $D(\mF_2)$, such that  $x=[P_1-Q_1]-[P_2-Q_2]$. Let $A_{P_i}$, $A_{Q_i}$ be perturbation operators associated with $P_i$, $Q_i$
		for $i=0,1$. From Theorem \ref{d176}, letting 
		$\mF=\mF_1+\mF_2$, $A_P=A_{P_1}\sqcup_B A_{Q_2}$, 
		$A_Q=A_{P_2}\sqcup_B A_{Q_1}$, we also have
		\begin{multline}\label{eq:3.28} 
		\ch_G(x)=\widetilde{\eta}_G(\mF_1, A_{P_1})-\widetilde{\eta}_G(\mF_1, A_{Q_1})-(\widetilde{\eta}_G(\mF_2, A_{P_2})-\widetilde{\eta}_G(\mF_2, A_{Q_2}))
		\\
		=\widetilde{\eta}_G(\mF_1+\mF_2, A_{P_1}\sqcup_B A_{Q_2})-\widetilde{\eta}_G(\mF_1+\mF_2, A_{P_2}\sqcup_B A_{Q_1})
		=\widetilde{\eta}_G(\mF, A_P)-\widetilde{\eta}_G(\mF, A_Q).
		\end{multline} 
		By (\ref{d140}), (\ref{d087})-(\ref{eq:3.28}) and Definition \ref{d049}, we have
			\begin{multline}\label{d200}
			\widehat{\pi}_Y!(\mF_X, \widetilde{\eta}_G(\mF_X, A_X))=\left[\mF_Z,
			\widetilde{\eta}_G(\mF_Z, A_Z)-\widetilde{\eta}_G(\mF, A_P)-\widetilde{\eta}_G(\mF^{\mathrm{op}}, A_Q^{\mathrm{op}})\right]
			\\
			=[\mF+\mF^{\mathrm{op}},0]=[\mF,0]-[\mF,0]=0\in \widehat{K}_G^*(B).
			\end{multline}
				
		Then from Theorem \ref{d157}, we complete the proof of Theorem \ref{d083}.
	\end{proof}
	
	Here our construction of $\widehat{\pi}_Y!$ involve an explicit choice of a representative $\widehat{o}_Y=(T^H_{\pi_Y}V, g^{TY}$, $\nabla^{L_Y}, \sigma_Y)$
	of the equivariant differential K-orientation. In fact, it does not depend on the choice.
	
	\begin{lemma}\label{d090}(Compare with \cite[Lemma 3.17]{BS09})
		The homomorphism $\widehat{\pi}_Y!:\widehat{K}_G^*(V)\rightarrow \widehat{K}_G^*(B)$
		only depend on the equivariant differential K-orientation.
	\end{lemma}
	\begin{proof}
		Let $\widehat{o}_Y=(T_{\pi_Y}^HV, g^{TY}, \nabla^{L_Y}, \sigma_Y)$, $\widehat{o}_Y'=(T_{\pi_Y}^{'H}V, g^{'TY}, \nabla^{'L_Y}, \sigma_Y')$
		be two representatives of an equivariant differential K-orientation.
		We will mark the objects associated with the second representative by $'$.
		From (\ref{d195}), we could get
		\begin{align}\label{d196}
		\widetilde{\mathrm{FLI}}_G(\nabla^{TY,TX}, \nabla^{L_Z}, \nabla^{'TY,TX}, \nabla^{'L_Z})=\int_{Y,G}\widetilde{\td}_G(\nabla^{TY}, \nabla^{L_Y}, \nabla^{'TY}, \nabla^{'L_Y})\, \mathrm{FLI}_G(\mF_X).
		\end{align}
		
		Then from (\ref{d078}), (\ref{d079}) and (\ref{d196}), we have
		\begin{multline}\label{d093}
		\widehat{\pi}_Y'!([\mF_X, \rho])-\widehat{\pi}_Y!([\mF_X, \rho])=[\mF_Z'+\mF_Z^{\mathrm{op}},
		\\
		\int_{Y,G} \left(\td_G(\nabla^{'TY}, \nabla^{'L_Y})
		-\td_G(\nabla^{TY}, \nabla^{L_Y})\right)\wedge \rho
		-\widetilde{\mathrm{FLI}}_G\left(\nabla^{'TZ}, \nabla^{'L_Z}, \nabla^{'TY,TX}, \nabla^{'L_Z}\right)
		\\
		\left.+\widetilde{\mathrm{FLI}}_G\left(\nabla^{TZ}, \nabla^{L_Z}, \nabla^{TY,TX}, \nabla^{L_Z}\right)
		+\int_{Y,G} (\sigma_Y'-\sigma_Y)\wedge (\mathrm{FLI}_G(\mF_X)-d \rho)\right]
		\\
		=\left[\mF_Z'+\mF_Z^{\mathrm{op}},\, \int_{Y,G} d\,\widetilde{\td}_G(\nabla^{TY}, \nabla^{L_Y}, \nabla^{'TY}, \nabla^{'L_Y})
		\wedge \rho\right.
		\\
		- \int_{Y,G} \widetilde{\td}_G(\nabla^{TY}, \nabla^{L_Y}, \nabla^{'TY}, \nabla^{'L_Y})
		\wedge d\rho
		+\int_{Y,G} \,\widetilde{\td}_G(\nabla^{TY}, \nabla^{L_Y}, \nabla^{'TY}, \nabla^{'L_Y} )\, \mathrm{FLI}_G(\mF_X)
		\\
		+\left.\widetilde{\mathrm{FLI}}_G\left(\nabla^{TZ}, \nabla^{L_Z}, \nabla^{'TZ}, \nabla^{'L_Z}\right)
		-\widetilde{\mathrm{FLI}}_G\left(\nabla^{TY,TX}, \nabla^{L_Z}, \nabla^{'TY,TX}, \nabla^{'L_Z}\right)\right]
		\\
		=[\mF_Z'+\mF_Z^{\mathrm{op}}, \widetilde{\mathrm{FLI}}_G\left(\mF_Z, \mF_Z'\right)].
		\end{multline}
		
		By Proposition \ref{d006} (iii) and Lemma \ref{d089}, there exists a perturbation operator $A$ 
		 with respect to $D(\mF_Z'+\mF_Z^{\mathrm{op}})$. By Theorem \ref{d176} and (\ref{d141}), there exists $x\in K_G^*(B)$ such that 
		\begin{align}\label{d095}
		\widetilde{\mathrm{FLI}}_G\left(\mF_Z, \mF_Z'\right)=\widetilde{\mathrm{FLI}}_G\left(\mF_Z+\mF_Z^{\mathrm{op}}, \mF_Z'+\mF_Z^{\mathrm{op}}\right)
		=-\widetilde{\eta}_G(\mF_Z'+\mF_Z^{\mathrm{op}}, A)+\ch_G(x).
		\end{align}
		Following the same process in (\ref{d087})-(\ref{d200}), we have $\widehat{\pi}_Y'!([\mF_X, \rho])=\widehat{\pi}_Y!([\mF_X, \rho])$. 
		
		The proof of Lemma \ref{d090} is completed.
	\end{proof}
		
	We now discuss the functoriality of the push-forward maps with respect to the composition of fiber bundles.
	Let $\pi_Y: V\rightarrow B$ with fibers $Y$ be as in the above subsection together with a representative of an
	equivariant differential K-orientation $\widehat{o}_{Y}=(T_{\pi_Y}^HV, g^{TY}, \nabla^{L_Y}, \sigma_Y)$.
	Let $\pi_U:B\rightarrow S$ be another equivariant smooth surjective proper submersion with compact oriented fibers $U$ together with a representative of an
	equivariant differential K-orientation $\widehat{o}_{U}=(T_{\pi_U}^HB, g^{TU}, \nabla^{L_U}, \sigma_U)$.
	
	Let $\pi_A:=\pi_U \circ \pi_Y:V\rightarrow S$ be the composition of two submersions with fibers $A$. Let $T_{\pi_A}^HV$ be a
	horizontal subbundle associated with $\pi_A$. We assume that $T_{\pi_A}^HV\subset T_{\pi_Y}^HV$.
	Set $g^{TA}=\pi_Y^* g^{TU}\oplus g^{TY}$, $\nabla^{L_A}=\pi_Y^*\nabla^{L_U}\otimes \nabla^{L_Y}$.
	
	\begin{defn}\label{d098}(Compare with \cite[Definition 3.21]{BS09})
		We define $\widehat{o}_A=\widehat{o}_U\circ \widehat{o}_Y$ by
		\begin{align}\label{d099}
		\widehat{o}_A:=(T_{\pi_A}^HV, g^{TA}, \nabla^{L_A}, \sigma_A),
		\end{align}
		where
		\begin{multline}\label{d100}
		\sigma_A:=\sigma_Y\wedge \pi_Y^*\td_G(\nabla^{TU}, \nabla^{L_U})+\td_G(\nabla^{TY}, \nabla^{L_Y})\wedge \pi_Y^* \sigma_U
		\\
		+\widetilde{\td}_G(\nabla^{TA}, \nabla^{L_A}, \nabla^{TU,TY}, \nabla^{L_A})- d\sigma_Y\wedge \pi_Y^*\sigma_U.
		\end{multline}
	\end{defn}
		
	\begin{thm}\label{d091}(Compare with \cite[Theorem 3.23]{BS09})
		We have the equality of homomorphisms $\widehat{K}_G^*(V)\rightarrow \widehat{K}_G^*(S)$
		\begin{align}\label{d092}
		\widehat{\pi}_A!=\widehat{\pi}_U!\circ\widehat{\pi}_Y!.
		\end{align}
	\end{thm}
	\begin{proof}
		The topological part of Theorem \ref{d091} is just Theorem \ref{d194} and
		the differential part follows from a direct calculation using (\ref{d082}) and (\ref{d100}).
	\end{proof}
	
	\subsection{Cup product}\label{s0303}
	
	In this subsection, we construct the cup product in equivariant differential K-theory in our model as in \cite{BS09,BS13} and prove the desired properties.

		Let $f: B_1\rightarrow B_2$ be a $G$-equivariant smooth map. 
		We define the induced homomorphism 
		$f^*:\widehat{K}_{G}^*(B_2)\to \widehat{K}_{G}^*(B_1)$
		as follows.
		For $\mF
	\in \mathrm{F}_G^*(B_2)$, in order to define $f^*\mF
	$, as remarked in Section \ref{s0102}, we need take care with the the pullback of the horizontal subbundle. Let $F$ be the natural map from $f^*W$ to $W$. We choose the new horizontal subbundle $ T_{f^*\pi}^H(f^*W)$ by the condition that $dF( T_{f^*\pi}^H(f^*W))\subseteq  T_{\pi}^H(W)$. Note that 
	the chosen of the new horizontal subbundle is not unique. If $A$ is a perturbation operator with respect to $D(\mF)$, then $f^*A$ is a perturbation operator with respect to $D(f^*\mF)$. Moreover, from Definition \ref{d017}, we have
	\begin{align}\label{d149}
	\widetilde{\eta}_G(f^*\mF, f^*A)=f^*\widetilde{\eta}_G(\mF, A).
	\end{align}
	By Proposition \ref{d022} and Definition
	\ref{d049}, if $\mF_1$, $\mF_2\in \mathrm{F}_G^*(B)$ be the 
	pullbacks of $\mF$ associated with distinct horizontal 
	subbundles, $(\mF_1, 0)\sim (\mF_2,0)$. So we obtain
	 a well defined pullback map
	 \begin{align}
	 f^*:\widehat{K}_{G}^*(B_2)\rightarrow \widehat{K}_{G}^*(B_1).
	 \end{align} 
	Evidently, $\mathrm{Id}_B^*=\mathrm{Id}_{\widehat{K}_G(B)}$.
	Let $f':B_0\rightarrow B_1$ be another equivariant smooth map. We could get
	\begin{align}
	f^{'*}f^*=(f\circ f')^*:\widehat{K}_{G}^*(B_2)\rightarrow \widehat{K}_{G}^*(B_0).
	\end{align}

	Let $[\mF,\rho]\in \widehat{K}_G^i(B)$ and $[\mF',\rho']\in \widehat{K}_G^*(B)$, where $i=0,1$.
We define (compare with \cite[Definition 4.1]{BS09})
\begin{align}
[\mF,\rho]\cup[\mF',\rho']:=[\mF\times_B\mF', (-1)^i\mathrm{FLI}_G(\mF)\wedge \rho'+\rho\wedge \mathrm{FLI}_G(\mF')-(-1)^id\rho\wedge \rho'].
\end{align}
	It is obvious that the product is natural with respect to pullbacks. 
	
	\begin{prop}(Compare with \cite[Propositions 4.2, 4.5]{BS09})
		(i) The product is well defined. It turns $B\mapsto \widehat{K}_G^*(B)$ into a contravariant functor from compact smooth $G$-manifolds with finite stablizers to unital graded commutative rings. 
		The unit is simply given by $[\mF,0]$, where $\mF$ is the equivariant geometric family in Example \ref{d129} a) such that $E_+$ is 1 dimensional trivial representation and $E_-=0$.
		
		(ii) The product is associative.
		
		(iii) Let $\pi_U:B\rightarrow S$ be an equivariant smooth proper submersion with oriented fibers and an equivariant differential K-orientation. For $x\in \widehat{K}_G^*(B)$ and $y\in \widehat{K}_G^*(S)$, we have
		\begin{align}
		\widehat{\pi}_U!(\pi_U^*y\cup x)=y\cup \widehat{\pi}_U!(x).
		\end{align}		
	\end{prop}
	\begin{proof}
	The product is obviously biadditive.

From Theorem \ref{d198} and a direct calculation, we could get the product is compatible with the equivalence relation in differential K-theory. 
	Other properties are the direct extension of the discussions in \cite[p47-50]{BS09}.	
	\end{proof}
		
	\begin{thm}\label{d031}(Compare with \cite[Section 3,4]{BS09})
		The equivariant differential K-theory $\widehat{K}_{G}$ is a contravariant functor $B\rightarrow \widehat{K}_{G}(B)$
		from the category of compact smooth $G$-manifolds with finite stabilizers
		 to unital $\Z_2$-graded commutative rings together with
		well-defined transformations
		
		(1) $R:\widehat{K}_{G}^*(B)\rightarrow \Omega_{deloc, G, cl}^{*}(B, \R)$ (curvature);
		
		(2) $I: \widehat{K}^*_G(B)\rightarrow K^*_G(B)$ (underlying $\mathrm{K}_G$-group);
		
		(3) $a: \Omega_{deloc, G}^*(B, \R)/\Im \,d\rightarrow \widehat{K}_G(X)$ (action of forms),
		
		where  $\Omega_{deloc, G, cl}^*(B, \R)$ denotes the set of closed delocalized differential forms, such that

		(1) the following diagram commutes
		
		\begin{center}\label{d032}
			\begin{tikzpicture}[>=angle 90]
			\matrix(a)[matrix of math nodes,
			row sep=2em, column sep=2.5em,
			text height=1.5ex, text depth=0.25ex]
			{\widehat{K}^*_G(B) & K^*_G(B)\\
				\Omega_{deloc, G, cl}^{*}(B, \R)& H_{deloc, G}^{*}(B, \R);\\};
			\path[->](a-1-1) edge node[above]{\footnotesize{$I$}} (a-1-2);
			\path[->](a-1-2) edge node[right]{\footnotesize{$\ch_{G}$}} (a-2-2);
			\path[->](a-1-1) edge node[right]{\footnotesize{$R$}} (a-2-1);
			\path[->](a-2-1) edge node[above]{\tiny{de Rham}} (a-2-2);
			\end{tikzpicture}
		\end{center}
				
		(2)
		\begin{align}\label{d033}
		R\circ a=d;
		\end{align}
		
		(3) $a$ is of degree 1;
				
		(4) for $x,y\in \widehat{K}_G^*(B)$ and $\alpha\in \Omega_{deloc, G}^*(B,\R)/\Im d$, we have
		\begin{align}\label{d203} 
		R(x\cup y)=R(x)\wedge R(y),\quad I(x\cup y)=I(x)\cup I(y), \quad a(\alpha)\cup x=a(\alpha\wedge R(x));
		\end{align}
		
		(5) the following sequence is exact:
		\begin{align}\label{d034}
		K_G^{*-1}(B)\overset{\footnotesize{\ch_{G}}}{\longrightarrow} \Omega^{*-1}_{deloc, G}(B, \R)/\Im\, d
		\overset{\footnotesize{a}}{\longrightarrow} \widehat{K}_G^*(B) \overset{\footnotesize{I}}{\longrightarrow} K_G^*(B)
		\longrightarrow 0.
		\end{align}		
	\end{thm}
	\begin{proof}
		We define the natural transformation
		\begin{align}\label{d053}
		I:\widehat{K}_G^*(B) \rightarrow K_G^*(B)
		\end{align}
		by
		\begin{align}\label{d054}
		I([\mF, \rho]):=\ind(D(\mF)).
		\end{align}
		From Definition \ref{d049}, the transformation $I$ is well defined.
		
		Let $a$ be a parity-reversing natural transformation
		\begin{align}\label{d055}
		a: \Omega_{deloc, G}^{\mathrm{even/odd}}(B, \R)/\Im \,d\rightarrow \widehat{K}_G^{1/0}(B)
		\end{align}
		by
		\begin{align}\label{d056}
		a(\rho):=[\emptyset, -\rho],
		\end{align}
		where $\emptyset$ is the empty geometric family.
		
		We define a transformation
		\begin{align}\label{d057}
		R: \widehat{\mathrm{IC}}_G^*(B)\rightarrow \Omega_{deloc, G, cl}^*(B, \R)
		\end{align}
		by
		\begin{align}\label{d058}
		R((\mF, \rho)):=\mathrm{FLI}_G(\mF)-d\rho.
		\end{align}
		If $(\mF', \rho')$ is paired with $ (\mF, \rho)$,
		there exists a perturbation operator $A$ with respect to $D(\mF+\mF^{'\mathrm{op}})$, such that
		$\rho-\rho'=\widetilde{\eta}_G(\mF+\mF^{'\mathrm{op}}, A)$. From (\ref{d019}) and (\ref{d050}), we have
		\begin{multline}\label{d060}
		R((\mF, \rho))=\mathrm{FLI}_G(\mF)-d\rho=\mathrm{FLI}_G(\mF)-d\rho'-
		d\widetilde{\eta}_G(\mF+\mF^{'\mathrm{op}}, A)
		\\
		=\mathrm{FLI}_G(\mF)-d\rho'-\mathrm{FLI}_G(\mF)+\mathrm{FLI}_G(\mF')
		=R((\mF', \rho')).
		\end{multline}
		Since $R$ is additive, it descends to $\widehat{\mathrm{IC}}_G^*(B)/\sim$ and finally to the map $R: \widehat{K}_G^*(B)\rightarrow \Omega_{deloc, G, cl}^*(B, \R)$. Let $f: B_1\rightarrow B_2$ be a $G$-equivariant smooth map. It follows from $\mathrm{FLI}_G(f^*\mF)=f^*\mathrm{FLI}_G(\mF)$ that $R$ is natural. 
		
		From (\ref{d056}) and (\ref{d058}), we have
		\begin{align}\label{d061}
		R\circ a=d.
		\end{align}
		
		By (\ref{d062}), the diagram commutes.
		
		The formulas in (\ref{d203}) follow from  straight calculations using the definitions.
		
		At last, we prove the exactness of the sequence (\ref{d034}).
		
		The surjectivity of $I$
		follows from Proposition \ref{d128}.
		
		Next, we show the exactness at $\widehat{K}_G^*(B)$. It is obvious that $I\circ a=0$.
		For a cycle $(\mF, \rho)$, if $I([\mF, \rho])=0$, we have $\ind(D(\mF))=0$. 
		By Example \ref{d129} b), we could take $\mF$ such that at least one component of the fiber has the nonzero dimension.
		So there exists a perturbation operator $A$ with respect to  $D(\mF)$ from Proposition \ref{d006}.
		By (\ref{d050}) and (\ref{d056}), we have
		\begin{align}\label{d065}
		[\mF, \rho]=a(\widetilde{\eta}_G(\mF, A)-\rho).
		\end{align}
				
		Finally, we prove the exactness at $\Omega^{*-1}_{deloc, G, cl}(B, \R)/\Im\, d$.
		Following the same process in (\ref{d087})-(\ref{d200}),
		for any $x\in K_G^*(B)$, by (\ref{d056}),
		\begin{align}\label{d197}
		a\circ \ch_{G}(x)=(\emptyset, \tilde{\eta}_G(\mF, A_Q)-\tilde{\eta}_G(\mF, A_P))=[\mF, 0]-[\mF,0]=0.
		\end{align}
		
		If $a(\rho)=0$, for any equivariant geometric family $\mF$ with
		a perturbation operator $A$ with respect to $D(\mF)$, by Definition \ref{d049} and (\ref{d065}), we have
		\begin{align}\label{d067}
		[\mF, \tilde{\eta}_G(\mF, A)-\rho]=a(\rho)=0=[\mF, \tilde{\eta}_G(\mF, A)].
		\end{align}
		So by Definition \ref{d051}, there exists another cycle $(\mF',\rho')$, such that $(\mF+\mF', \rho'+\tilde{\eta}_G(\mF, A)-\rho)\sim(\mF+\mF', \rho'+\tilde{\eta}_G(\mF, A))$.
		Since $\sim$ is generated by "paired", we have the cycles $\{(\mF_i, \rho_i)\}_{0\leq i\leq r}$ such that for any $1\leq i\leq r$, $(\mF_i, \rho_i)$ is paired with $(\mF_{i-1}, \rho_{i-1})$, $(\mF_0, \rho_0)=(\mF+\mF', \rho'+\tilde{\eta}_G(\mF, A)-\rho)$ and $(\mF_r, \rho_r)=(\mF+\mF', \rho'+\tilde{\eta}_G(\mF, A))$. By Definition \ref{d049}, for any $1\leq i\leq r$, there exists a perturbation operator $A_i$ with respect to $D(\mF_{i-1}+\mF_i^{\mathrm{op}})$ such that
		$
		\rho_{i-1}-\rho_i=\tilde{\eta}_G(\mF_{i-1}+\mF_i^{\mathrm{op}}, A_i).
		$
		Let $A_i'$ ($0\leq i\leq r$) be the perturbation operator with respect to  $D(\mF_{i}+\mF_i^{\mathrm{op}})$ taken in (\ref{d141}). 
		Therefore, by Theorem \ref{d176}, (\ref{d145}) and (\ref{d141}), there exists $x\in K_G^*(B)$, such that
		\begin{multline}
		-\rho=\sum_{i=1}^r(\rho_{i-1}-\rho_i)=\tilde{\eta}_G(\mF_{0}+\mF_1^{\mathrm{op}}+\cdots +\mF_{r-1}+\mF_{r}^{\mathrm{op}}, A_1\sqcup_B\cdots\sqcup_B A_r)
		\\
		=\tilde{\eta}_G(\mF_{0}+\mF_r^{\mathrm{op}}+\cdots +\mF_{r-1}+\mF_{r-1}^{\mathrm{op}}, A_1\sqcup_B\cdots\sqcup_B A_r)
		\\
		-\tilde{\eta}_G(\mF_{0}+\mF_r^{\mathrm{op}}+\cdots +\mF_{r-1}+\mF_{r-1}^{\mathrm{op}}, A_0'\sqcup_B\cdots\sqcup_B A_{r-1}')
		=\ch_G(x).
		\end{multline}
				
		The proof of Theorem \ref{d031} is completed.	
	\end{proof}
	
	The direct extension of \cite[Proposition 3.19 and Lemma 3.20]{BS09} show that the pullback map and the exact sequence (\ref{d034}) are compatible with the push-forward maps. 	
		
	\begin{rem}
		If the group $G$ is trivial, all the models of differential K-theory are isomorphic (see e.g. \cite{BS12}). For equivariant case, the uniqueness is an open question.
	\end{rem}
	
	\subsection{Differential K-theory for orbifolds}
	
	In \cite{BS13}, Bunke and Schick constructed the first model of
	the differential K-theory for orbifolds by using the language of stacks
	and proved the desired properties. It could be regarded as a model of the equivariant differential K-theory when the action has finite stabilizers.
	In the subsections above, inspired by the constructions in \cite{BS09,Or09},
	we construct the a model of the equivariant differential K-theory when the action has finite stabilizers. In this subsection, we will explain that this model
	could also be regarded as a model for orbifolds.

	Let $\mathcal{X}$ be a compact orbifold (effective orbifold in some literatures).
	There exist a compact smooth manifold $B$ and a compact Lie group $G$ such that
	$\mathcal{X}$ is diffeomorphic to a quotient for a smooth effective 
	$G$-action on $B$ with finite stabilizers (see \cite[Theorem 1.23]{ALR07}).
	
	Let $K_{orb}^0(\mathcal{X})$ be the orbifold K-group of the compact orbifold $\mathcal{X}$ defined as the Grothendieck ring of the equivalence classes of orbifold vector bundles over $\mathcal{X}$. Since $\mathcal{X}$ is an orbifold, $\mathcal{X}\times S^1$ is an orbifold. Moreover, $i: \mathcal{X}\rightarrow \mathcal{X}\times S^1$ is a morphism in the category of orbifolds. As in (\ref{d124}), we define the orbifold $K^1$ group $K_{orb}^1(\mathcal{X}):=\ker (i^*: K_{orb}^0(\mathcal{X}\times S^1)\rightarrow K_{orb}^0(\mathcal{X}))$.

	Let $p: B\rightarrow B/G=\mathcal{X}$ be the projection. Then from \cite[Proposition 3.6]{ALR07}, it induces an isomorphism
	$p^*: K_{orb}^*(\mathcal{X})\rightarrow K_G^*(B)$.
	Note that if the orbifold $\mathcal{X}$ can be presented in two different ways as a quotient, say $B'/G'\simeq \mathcal{X}\simeq B/G$, it shows that $K_{G'}^*(B')\simeq K_{orb}^*(\mathcal{X})\simeq K_G^*(B)$. 
	So we can consider the orbifold K-theory as a special case of the equivariant K-theory. 
	
	Furthermore, from the definition of the differential structure on orbifolds, we know that $\Omega_{deloc, G}^*(B,\R)/\Im d\simeq \Omega_{deloc, G'}^*(B',\R)/\Im d$. From the exact sequence in (\ref{d034}) and five lemma, we have
	\begin{align}
	\widehat{K}_{G'}^*(B')\simeq \widehat{K}_G^*(B).
	\end{align}
	Therefore, this model of equivariant differential K-theory for $G$-action
	with finite stabilizers could be regarded as a model of differential K-theory for orbifolds.

	\appendix
	
	\section{Equivariant K-theory for smooth complex vector bundles}
	
	In this appendix, we show that for a compact Lie group $G$
	and a smooth compact manifold $B$ with a smooth $G$-action,
	$K_G^0(B)$
	in \cite{Se68} defined by $G$-equivariant 
	topological complex vector bundles could be studied using
	the $G$-equivariant smooth complex vector bundles.
	Although this is certainly 
	well-known (see e.g., \cite[(2.1)]{DM20}), we were
	unable to find an explicit proof in the literature. 
	We state it here for the completeness following the suggestion
	of a referee. In this appendix, all vector bundles are complex.
	
	For a representation $V$ of $G$, for $v\in V$, if $Gv$ generates
	a finite dimensional subspace of $V$, we say $v$ is a 
	$G$-finite vector in $V$.
	
	Let $E$ be a $G$-equivariant smooth vector bundle over $B$. 
	Take a Hermitian metric on $E$ and let $\|\cdot\|_{\cC^0(B,E)}$
	be the corresponding $\cC^0$-norm.
	For $s\in \cC^{\infty}(X,E)$, 
	$gs\in \cC^{\infty}(X,E)$. Thus for any $f\in \cC^{\infty}(G)$, 
	\begin{align}\label{eq:a04}
     s_f:=\int_Gf(g)gs\, dg\in \cC^{\infty}(B,E),
	\end{align}
	 where $dg$ is the Haar measure.
	Since $G$ acts continuously on $\cC^{\infty}(B,E)$,
	for any $\var>0$, there exists a neighborhood $U$ of the unity $e\in G$
	such that for any $g\in U$, $\|gs-s\|_{\cC^0(B,E)}<\var/2$. 
	Let $v\in \cC^{\infty}(G)$ be a non-negative function vanishing
	outside $U$ with $\int_G v(g)dg=1$. Then 
	$\|s_v-s\|_{\cC^0(B,E)}<\var/2$. 
	Let $M_s:=\|\int_G gs\, dg\|_{\cC^{0}(B,E)}$.
	From Peter-Weyl theorem,
	there exists a $G$-finite vector $u\in \cC^{\infty}(G)$, such that  $\|v-u\|_{\cC^0(G)}<\frac{\var}{2M_s}$.
	Thus, $\|s_v-s_u\|<\var/2$. Observe that
	$s_u$ is a $G$-finite vector in $\cC^{\infty}(B,E)$. We have the following lemma.

	\begin{lemma}\label{lem:a01}(cf. \cite[\S 2.16]{Mo61})	
	For $s\in \cC^{\infty}(B,E)$, for any $\var>0$, there exists
	a $G$-finite vector $s'\in \cC^{\infty}(B,E)$ such that 
	$\|s-s'\|_{\cC^{0}(B,E)}<\var$.
	\end{lemma}
	
	For a finite dimensional complex representation $M$ of $G$, we consider the 
	$G$-action on $B\times M$ given by 
	\begin{align}\label{eq:a02} 
	g(b,u)=(gb,gu),\quad \forall g\in G, b\in B, u\in M.
	\end{align}
	Thus
	$B\times M\rightarrow B$ is an equivariant smooth vector bundle over $B$.  In this case, a $G$-invariant Hermitian inner product on 
	$M$ forms a $G$-invariant Hermitian smooth metric on this vector bundle.
The following lemma extends \cite[Proposition 2.4]{Se68} to the
category of $G$-equivariant smooth vector bundles. 

\begin{prop}\label{prop:a02}
Let $E$ be a $G$-equivariant smooth vector bundle over $B$.
There exist a finite dimensional complex representation $M$ of $G$
and a $G$-equivariant smooth vector bundle $F$ over $B$ such that
$E\oplus F$ is isomorphic to $B\times M$ as $G$-equivariant smooth  vector bundles.
\end{prop} 
\begin{proof}
It suffices to find a equivariant smooth surjection from some $B\times M$ 
to $E$. Then $F$ is the orthogonal complement of $E$ in $B\times M$.

For any $b\in B$, we can choose a finite set $\sigma_b\subset
\cC^{\infty}(B,E)$, such that $\{s(b) \}_{s\in \sigma_b}$
spans $E_b$. From Lemma \ref{lem:a01}, we can choose $\sigma_b$
such that it
consists of $G$-finite vectors in $\cC^{\infty}(B,E)$. There exists a neighborhood of 
$b$, $U_b$, such that for any $x\in U_b$, $\{s(x) \}_{s\in \sigma_b}$ spans $E_x$. Suppose $U_{b_1},\cdots, U_{b_m}$ 
covers $B$. Let $\sigma=\cup_i \sigma_{b_i}$. Let 
$M$ be the finite dimensional subspace of $\cC^{\infty}(B,E)$
generated by $\sigma$. Then the evaluation map $B\times M
\to E$ is the required surjection. 
\end{proof}

\begin{lemma}\label{lem:a03} (Compare with \cite[Theorem 3.5]{Hi76})
For every $G$-equivariant topological vector bundle $E$ over $B$,
there exists a $G$-equivariant smooth vector bundle $E^s$ over $B$, which is
unique up to isomorphism of $G$-equivariant smooth vector bundles, such that
$E^s$ is isomorphic to $E$ as $G$-equivariant topological vector bundles.
\end{lemma}
\begin{proof}
By  \cite[Proposition 2.4]{Se68}, the $\cC^0$-version of Proposition \ref{prop:a02}, there exists a finite dimensional complex representation $M$ of $G$ and an equivariant embedding $i:E\to B\times M$. 
Let $Gr_{M,r}$ be the Grassmannian parameterizing all complex linear
subspaces of finite dimensional complex representation $M$ of given dimension $r$.
Since $G$ acts linearly on $M$, there is a induced smooth $G$-action on smooth manifold $Gr_{M,r}$.
Let $r$ be the rank of $E$. Define continuous map $h:B\to Gr_{M,r}$ by $h(b):=i(E_b)\in Gr_{M,r}$, $\forall b\in B$.
Since $i$ is equivariant, $h$ is an equivariant map. Let $\gamma_{M,r}$ be 
the universal bundle over $Gr_{M,r}$, which is an equivariant
smooth
vector bundle over $Gr_{M,r}$. Then
$h^*\gamma_{M,r}$ is isomorphic to $E$ as 
$G$-equivariant
topological vector bundles.

For equivariant continuous map $h:B\to Gr_{M,r}$, there exists an equivariant smooth map $h_G:B\to Gr_{M,r}$, which is $G$-homotopy to $h$
continuously (see e.g.,  \cite[Theorem VI.4.2]{Br72}).
So $E^s:=h_G^*\gamma_{M,r}$ is a $G$-equivariant smooth vector bundle which is isomorphic to $ h^*\gamma_{M,r}\simeq E$ as $G$-equivariant topological vector bundles.


For two $G$-equivariant smooth vector bundles $E_1^s$ and $E_2^s$,
which are isomorphic as $G$-equivariant topological vector bundles,
there exist $G$-equivariant smooth maps $h_i: B\to Gr_{M,r}$,
$i=1,2$,
such that $E_i^s=h_i^*\gamma_{M,r}$ and $h_1$ is $G$-homotopy to $h_2$ continuously.
Since $G$ is compact, $h_1$ is $G$-homotopy to $h_2$ smoothly
(see e.g., \cite[Corollary VI.4.3]{Br72}).
Thus $E_1^s$ is isomophic to $E_2^s$ as $G$-equivariant
smooth vector bundles.

The proof of Lemma \ref{lem:a03} is completed.
\end{proof}

\begin{prop}\label{prop:a04}
Let $K^0_{G,sm}(B)$ be the Grothendieck group of the $G$-equivariant  smooth 
vector bundles over $B$. We have
\begin{align}\label{eq:a01}
K^0_{G,sm}(B)\simeq K^0_G(B).
\end{align}
\end{prop}
\begin{proof}
Forgetting the smooth structure, we obtain a well-defined map
$A:K^0_{G,sm}(B)\to K^0_G(B)$. Let $\mathrm{Vect}_G(B)$ 
and $\mathrm{Vect}_{G,sm}(B)$ be the equivalence classes of 
$G$-equivariant topological and smooth vector bundles over $B$. Then 
Lemma \ref{lem:a03} induces a well-defined map
$\mathrm{Vect}_G(B)\to \mathrm{Vect}_{G,sm}(B)$.
For $E_1$, $E_2$ in $\mathrm{Vect}_G(B)$, if $[E_1]=[E_2]
\in K^0_G(B)$, there exists topological vector bundle
$F$ such that $E_1\oplus F$ is
isomorphic to $E_2\oplus F$ as $G$-equivariant
topological vector bundles. 
Let $E_1^s$, $E_2^s$ and $F^s$ be the corresponding $G$-equivariant smooth
vector bundles. Since  $E_1^s\oplus F^s$ is
isomorphic to $E_2^s\oplus F^s$ as $G$-equivariant
topological vector bundles, from the uniqueness in Lemma 
\ref{lem:a03}, they are isomorphic as $G$-equivariant
smooth vector bundles. Thus we get a well-defined 
map $B:K^0_G(B)\to K^0_{G,sm}(B)$. Easy to see that
$A\circ B$ and $B\circ A$ are all identity maps.

The proof of Proposition \ref{prop:a04} is completed.

\end{proof}

	\section{Equivariant family index for odd dimensional fibers}
	

In this appendix, we summarize some results on the equivariant
family index for odd dimensional fibers and the explanations
for $K_G^1(B)$ (cf. \cite{AS69,FHT11,MP97b}).

We consider the equivariant $\Z_2$-graded 
Hilbert bundle $\mE$ with fiber
$L^2(Z_b,E)$ over $b\in B$. From \cite[Lemma A.32]{FHT11},
there exists an equivariant embedding from $\mE$
to the equivariant trivial Hilbert bundle 
$B\times L^2(G)\otimes C(\R)\otimes H$, where 
$C(\R)$ is 
the complex Clifford algebra and 
$H$ is a separable Hilbert space. 
As in \cite[Definitions A.39 and A.40]{FHT11}, 
for any equivariant $\Z_2$-graded 
Hilbert space $\mH$, we denote by
$\mathrm{Fred}^{0}(\mH)$ the space of odd skew-adjoint 
equivariant Fredholm
operators $A$, for which $A^2+1$ is compact, 
topologized as a subspace of $B(\mH)\times K(\mH)$, where
$B(\mH)$ and $K(\mH)$ are the sets of bounded linear operators and compact operators on $\mH$ given the compact-open topology and 
the norm topology respectively.
Denote by $\mathrm{Fred}^1(\mH)$ the subspace of 
$\mathrm{Fred}^0(C(\R)\otimes \mH)$ consisting
of odd operators $A$, which supercommute with the action of $C(\R)$ and for
which the essential spectrum of $-\sqrt{-1}c(e)A$ 
contains both positive and negative eigenvalues,
where $c(e)$ is the basis element of $C(\R)$. 
By \cite[\S 3.5.4]{FHT11}, $K_G^1(B)$ is realized as
the space of $G$-homotopy classes of $G$-equivariant maps 
from $B$ to $\mathrm{Fred}^1(L^2(G)\otimes H)$:
\begin{align}\label{eq:2.03}
K_G^1(B)\simeq [B, \mathrm{Fred}^1(L^2(G)\otimes H)]_G.
\end{align}

Let $T=D(\mF)/(1+D(\mF)^2)^{1/2}$. Then $T$ is bounded,  $G$-equivariant
and $\ind(T)=\ind(D(\mF))\in K_G^1(B)$. 
Moreover $\sqrt{-1}T$ can be extended to an equivariant map from $B$ to
$\mathrm{Fred}^1(L^2(G)\otimes H)$
by taking the identity map on the complement of $\mE$
in $B\times L^2(G)\otimes C(\R)\otimes H$.
By \cite[\S 3.5.4 and Proposition A.41]{FHT11}, $\ind(D(\mF))=\ind(T)\in K_G^1(B)$ 
corresponds to the element of 
$K_G^0(B\times (0, \frac{1}{2}))\simeq[B\times (0, \frac{1}{2}), \mathrm{Fred'}^0(L^2(G)\otimes H)]_G$ given by 
		\begin{align}\label{eq:A.06}
D(\theta)=
\cos(2\pi\theta)+\sqrt{-1}T\sin(2\pi\theta),\quad \theta\in (0,\frac{1}{2}).
\end{align}
Here $\mathrm{Fred'}^0(L^2(G)\otimes H)$ consists of the elements in $\mathrm{Fred}^0(L^2(G)\otimes H)$ which is invertible outside a compact
set of the parameter space (see also \cite[p6]{AS69}, \cite[(3.1)]{MP97b}). By applying the natural inclusion
$K_G^0(B\times (0, \frac{1}{2}))\to K_G^0(B\times S^1)$, we obtain 
an element of $K_G^0(B\times S^1)$ which lies in the image of 
$j$ in (\ref{d124}), thus an element of $K_G^1(B)$.

	The following proposition is the equivariant 
	version of	\cite[Proposition 6]{MP97b}.  We prove it here using the notation
	in Example \ref{d129} d) for the completeness. 
	
	\begin{prop}\label{prop:A.01}
	For $\mF\in \mathrm{F}_G^1(B)$, there exists inclusion $i:B\rightarrow B\times S^1$ such that $i^*\ind(D(p_1^*\mF\times_{B\times S^1} p_2^*\mF^L))=0$. Moreover, as an element of $K_G^1(B)$, we have
	\begin{align}\label{eq:A.01}
	j\big(\ind(D(\mF))\big)=\ind(D(p_1^*\mF\times_{B\times S^1} p_2^*\mF^L)).
	\end{align}
	\end{prop}
	
	\begin{proof}
		In order to compare our definition with that in \cite{AS69}, we replace 
		the connection in (\ref{eq:1.26}) by $\nabla^L=d+2\pi (\theta-1/4)\sqrt{-1}dt$.
		Since from (\ref{eq:1.04}) and (\ref{eq:1.05}),
			\begin{align}\label{eq:A.02}
		D(p_1^*\mF\times_{B\times S^1} p_2^*\mF^L)=D(\mF)\otimes J+D(\mF^L)\otimes K,
		\end{align}
		the index in the right hand side of (\ref{eq:A.01}) does not vary in $K_G^0(B\times S^1)$ after the replacement. 		
		Then we could calculate that for $\theta\in [0,1)$, the eigenvalues of $D(\mF^L)$ are $\{\lambda_k=2\pi k+2\pi (\theta-\frac{1}{4})\}_{k\in \Z}$  and the eigenspace of $\lambda_k$ is one dimensional for any $k\in \Z$. Let $s$ be a local frame of $L$. The eigenfunction of $\lambda_k$ is $ v_k(t)=\exp(2\pi k\sqrt{-1}t)s$.
		From (\ref{eq:A.02}), 
		$D(p_1^*\mF\times_{B\times S^1} p_2^*\mF^L)^2=(D(\mF)^2+D(\mF^L)^2)\otimes Id_{\C^2}.$
		So it is invertible if and only if $\theta\neq 1/4$.
		Thus for inclusion $i:B\rightarrow B\times S^1$, $i(B)=B\times \{\frac{1}{2}\}$, 
		$i^*\ind(D(p_1^*\mF\times_{B\times S^1} p_2^*\mF^L))=0$.
		 
		Fix $b\in B$. Since $\mS_{S^1\times Z}=\mS_{S^1}\otimes \mS_Z\otimes \C^2$, we have
		\begin{align}\label{eq:A.03}
		L^2(S_t^1\times Z_b, \mS_{S^1\times Z}\widehat{\otimes}L \widehat{\otimes} E)=\,'\bigoplus_{k\in \Z}\C v_k(t)\otimes L^2(Z_b, \mS_Z\widehat{\otimes} E)\otimes \C^2.
		\end{align}
		Here $\,'\oplus$ stands for the direct sum in the category of Hilbert spaces. 
		If $k\neq 0$, 
		$D(p_1^*\mF\times_{B\times S^1} p_2^*\mF^L)^2|_{\C v_k(t)\otimes L^2(Z_b, \mS_Z\widehat{\otimes} E)\otimes \C^2}>0.$
		Let
		$H_+=\C v_0(t)\otimes L^2(Z_b, \mS_Z\widehat{\otimes} E)\otimes (\C\oplus\{0\}), \, H_-=\C v_0(t)\otimes L^2(Z_b, \mS_Z\widehat{\otimes} E)\otimes (\{0\}\oplus\C).$
		Then
		\begin{align}\label{eq:A.04} 
	\ind(D(p_1^*\mF\times_{B\times S^1} p_2^*\mF^L))=\ind\left(D(p_1^*\mF\times_{B\times S^1} p_2^*\mF^L)_+:H_+\rightarrow H_-\right).
		\end{align}
		We define the isomorphisms 
		$\phi_+:H_+\rightarrow L^2(Z_b, \mS_Z\widehat{\otimes} E), \quad \phi_-:H_-\rightarrow L^2(Z_b, \mS_Z\widehat{\otimes} E),$	
		by
		$\phi_+(v_0(t)\otimes l\otimes (1,0))=l,\quad \phi_-(v_0(t)\otimes l\otimes (0,1))=l.$	
		Set
		$D_+:=\phi_-\circ D(p_1^*\mF\times_{B\times S^1} p_2^*\mF^L)_+\circ \phi_+^{-1}.$	
		Then on $L^2(Z_b, \mS_Z\widehat{\otimes} E)$, from (\ref{eq:1.04})
		and (\ref{eq:A.02}), we have
		\begin{align}\label{eq:A.05}
		D_+=\sqrt{-1}D(\mF)+\lambda_0(\theta).
		\end{align}	
		From (\ref{eq:A.06}) and (\ref{eq:A.05}), 
		for $\theta\in (0,\frac{1}{2})$, we have
		\begin{multline}\label{eq:A.08}
		\ind(D(p_1^*\mF\times_{B\times S^1} p_2^*\mF^L))=\ind\left(\frac{D(p_1^*\mF\times_{B\times S^1} p_2^*\mF^L)_+}{\sqrt{1+D(p_1^*\mF\times_{B\times S^1} p_2^*\mF^L)^2}}:H_+\rightarrow H_-\right)\\
		=\ind\left(\frac{D_+}{\sqrt{1+\lambda_0(\theta)^2+D(\mF)^2}}\right)
		=\ind\left(\frac{\lambda_0(\theta)+\sqrt{-1}D(\mF)}{\sqrt{1+\lambda_0(\theta)^2+D(\mF)^2}}\right)
		=\ind\left(D(\theta)\right).
		\end{multline}	
		Since $D(p_1^*\mF\times_{B\times S^1} p_2^*\mF^L)$ and $D(\theta)$
		are invertible for $\theta\neq \frac{1}{4}$, we obtain
Proposition \ref{prop:A.01}.
	\end{proof}
	
	\vspace{3mm}\textbf{Acknowledgements}\ \
	The author would like to thank Professors Xiaonan Ma and U. Bunke
	for helpful discussions.
	He would like to thank Shu Shen and Guoyuan Chen for the conversations.
	He would also like to thank the anonymous referee for various 
	useful suggestions that improved the clarity and quality of the paper.
	This research is partly supported by  Science and Technology Commission 
	of Shanghai Municipality (STCSM), grant No.18dz2271000, Natural Science Foundation of Shanghai, grant No.20ZR1416700 and NSFC No.11931007.
	
%


\begin{thebibliography}{10}
		
		\bibitem{ALR07}
		A.~Adem, J.~Leida, and Y.~Ruan.
		\newblock {\em Orbifolds and stringy topology}, volume 171 of {\em Cambridge
			Tracts in Mathematics}.
		\newblock Cambridge University Press, Cambridge, 2007.
		
		\bibitem{AtK}
		M.~F.~Atiyah.
		\newblock {\em {$K$}-theory}. 
		\newblock 2nd ed., Adv. Book Class., Addison–Wesley, Redwood City 1989.
		
		\bibitem{ASe04}
		M.~F.~Atiyah and G.~Segal.
		\newblock Twisted {$K$}-theory.
		\newblock {\em Ukr. Mat. Visn.}, 1(3):287--330, 2004; 
		translation in {\em Ukr. Math. Bull.} 1(3):291--334, 2014. 	
		
		
		\bibitem{AS69}
		M.~F.~Atiyah and I.~M.~Singer.
		\newblock Index theory for skew-adjoint {F}redholm operators.
		\newblock {\em Inst. Hautes \'Etudes Sci. Publ. Math.}, (37):5--26, 1969.
		
		\bibitem{AS71}
		M.~F.~Atiyah and I.~M.~Singer.
		\newblock The index of elliptic operators. {IV}.
		\newblock {\em Ann. of Math. (2)}, 93:119--138, 1971.
		
		\bibitem{BGV04}
		N.~Berline, E.~Getzler, and M.~Vergne.
		\newblock {\em Heat kernels and {D}irac operators}.
		\newblock Grundlehren Text Editions. Springer-Verlag, Berlin, 2004.
		\newblock Corrected reprint of the 1992 original.
		
		\bibitem{Be09}	A.~Berthomieu.
		\newblock  Direct image for some secondary {K}-theories. 
		\newblock {\em Ast\'erisque}, 327:289--360, 2009. 
		
		\bibitem{Bi86}
		J.-M.~Bismut.
		\newblock The {A}tiyah-{S}inger index theorem for families of {D}irac
		operators: two heat equation proofs.
		\newblock {\em Invent. Math.}, 83(1):91--151, 1986.
		
		\bibitem{BC89}
		J.-M. Bismut and J. Cheeger.
		\newblock {$\eta$}-invariants and their adiabatic limits.
		\newblock {\em J. Amer. Math. Soc.}, 2(1):33--70, 1989.
		
		\bibitem{BF86I}
		J.-M.~Bismut and D.~Freed.
		\newblock The analysis of elliptic families. {I}. {M}etrics and connections on
		determinant bundles.
		\newblock {\em Comm. Math. Phys.}, 106(1):159--176, 1986.
		
		\bibitem{BF86II}
		J.-M.~Bismut and D.~Freed.
		\newblock The analysis of elliptic families. {II}. {D}irac operators, eta
		invariants, and the holonomy theorem.
		\newblock {\em Comm. Math. Phys.}, 107(1):103--163, 1986.
		
		\bibitem{BL92}
		J.-M.~Bismut and G.~Lebeau.
		\newblock Complex immersions and {Q}uillen metrics.
		\newblock {\em Inst. Hautes \'Etudes Sci. Publ. Math.}, (74):ii+298 pp. (1992),
		1991.
		

				\bibitem{Br72}
		G.~Bredon.
		\newblock {\em Introduction to Compact Transformation Groups}.
		\newblock New York: Academic Press, 1972.
		
		\bibitem{Bu09}
		U.~Bunke.
		\newblock Index theory, eta forms, and {D}eligne cohomology.
		\newblock {\em Mem. Amer. Math. Soc.}, 198(928):vi+120, 2009.
		
		\bibitem{BM04}
		U.~Bunke and X.~Ma.
		\newblock Index and secondary index theory for flat bundles with duality.
		\newblock In {\em Aspects of boundary problems in analysis and geometry},
		volume 151 of {\em Oper. Theory Adv. Appl.}, pages 265--341. Birkh\"auser,
		Basel, 2004.
		
		\bibitem{BS09}
		U.~Bunke and T.~Schick.
		\newblock Smooth {$K$}-theory.
		\newblock {\em Ast\'erisque}, 328:45--135, 2009.
		
		\bibitem{BS12}
		U.~Bunke and T.~Schick.
		\newblock Differential {K}-theory: a survey.
		\newblock In {\em Global differential geometry}, volume~17 of {\em Springer
			Proc. Math.}, pages 303--357. Springer, Heidelberg, 2012.
		
		\bibitem{BS13}
		U.~Bunke and T.~Schick.
		\newblock Differential orbifold {K}-theory.
		\newblock {\em J. Noncommut. Geom.}, 7(4):1027--1104, 2013.
		
		\bibitem{D91}
		X.~Dai.
		\newblock Adiabatic limits, nonmultiplicativity of signature, and {L}eray
		spectral sequence.
		\newblock {\em J. Amer. Math. Soc.}, 4(2):265--321, 1991.
		
		\bibitem{DZ98}
		X.~Dai and W.~Zhang.
		\newblock Higher spectral flow.
		\newblock {\em J. Funct. Anal.}, 157(2):432--469, 1998.
		
		\bibitem{DM20}
		P.~Dimakis and R.~Melrose.
		\newblock Equivariant k-theory and resolution, i: Abelian
		actions.
		\newblock {\em Geometric Analysis: In Honor of Gang Tian’s 60th Birthday}, Progress in Mathematics, vol. 333, Birkh¨auser Basel,
		2020.
		
		
        \bibitem{Do78}
        H.~Donnelly.
        \newblock Eta invariants for $G$-spaces.
        \newblock {\em Indiana Univ. Math. J.}, 27(6):889--918, 1978.
		
		\bibitem{dR73}
		G.~de~Rham.
		\newblock {\em Vari\'et\'es diff\'erentiables. {F}ormes, courants, formes
			harmoniques}.
		\newblock Hermann, Paris, 1973.
		\newblock Troisi{\`e}me {\'e}dition revue et augment{\'e}e, Publications de
		l'Institut de Math{\'e}matique de l'Universit{\'e} de Nancago, III,
		Actualit{\'e}s Scientifiques et Industrielles, No. 1222b.
		
		\bibitem{F05}
		H.~Fang.
		\newblock Equivariant spectral flow and a {L}efschetz theorem on
		odd-dimensional {S}pin manifolds.
		\newblock {\em Pacific J. Math.}, 220(2):299--312, 2005.
		
		\bibitem{FH00}
		D.~S. Freed and M.~Hopkins. 
		\newblock On {R}amond-{R}amond fields and {$K$}-theory.
		\newblock {\em J. High Energy Phys.}, (2000) no.~5, Paper 44, 14 pp. 

		\bibitem{FHT11}
D.~S.~Freed, M.~J.~Hopkins and C.~Teleman.
\newblock Loop groups and twisted {$K$}-theory I.
\newblock {\em J. Topol. }, 4(4):737--798, 2011.
		
		\bibitem{FL10}
		D.~S.~Freed and J.~Lott.
		\newblock An index theorem in differential {$K$}-theory.
		\newblock {\em Geom. Topol.}, 14(2):903--966, 2010.
		
	
\bibitem{Getzler93}
E.~Getzler.
\newblock The odd {C}hern character in cyclic homology and spectral flow.
\newblock {\em Topology},  32(3):489--507, 1993. 

		
		\bibitem{GL18}
		A.~Gorokhovsky and J.~Lott.
		\newblock A Hilbert bundle description of differential {$K$}-theory.
		\newblock {\em Adv. Math.}, 328:661--712, 2018.
		
		\bibitem{Hi76}
		M.~W.~Hirsch.
		\newblock {\em Differential topology}.
		\newblock Graduate Texts in Mathematics, No. 33. Springer-Verlag, New York-Heidelberg, 1976. x+221 pp.
		

		
		
		\bibitem{HS05}
		M.~J.~Hopkins and I.~M.~Singer.
		\newblock Quadratic functions in geometry, topology, and {M}-theory.
		\newblock {\em J. Differential Geom.}, 70(3):329--452, 2005.
		
		\bibitem{LM89}
		H.~B.~Lawson and M.-L.~Michelsohn.
		\newblock {\em Spin geometry}, volume~38 of {\em Princeton Mathematical
			Series}.
		\newblock Princeton University Press, Princeton, NJ, 1989.
		
		\bibitem{LP03}
		E.~Leichtnam and P.~Piazza.
		\newblock Dirac index classes and the noncommutative spectral flow.
		\newblock {\em J. Funct. Anal.}, 200(2):348--400, 2003.
		
		\bibitem{Liu17}
		B.~Liu.
		\newblock Functoriality of equivariant eta forms.
		\newblock {\em J. Noncommut. Geom.}, 11(1):225--307,
		 2017.
		 
		 \bibitem{Liu19} 
		 B.~Liu.
		 \newblock Real embedding and equivariant eta forms.
		 \newblock {\em Math. Z.}, 292:849--878, 2019.
		  
		 
		 \bibitem{LM18a}
		 B.~Liu and X.~Ma.
		 \newblock Differential {$K$}-theory and localization formula for 
		 $\eta$-invariants.
		 \newblock {\em Invent. Math.}, 222(2):545--613,
		 2020.
		 

		
		\bibitem{LM00}
		K.~Liu and X.~Ma,
	 \newblock On family rigidity theorems. {I}.
	 \newblock {\em Duke Math. J.}, 102(3):451--474, 2000.	
		
		\bibitem{LMZ00}
		K.~Liu, X.~Ma, and W.~Zhang, 
		\newblock {${\rm Spin}^c$} manifolds and rigidity theorems in 
			{$K$}-theory.
			\newblock {\em Asian J. Math.}, 4(4):933--959, 2000.
		
		\bibitem{Ma99}
		X.~Ma.
		\newblock Formes de torsion analytique et familles de submersions. {I}.
		\newblock {\em Bull. Soc. Math. France}, 127(4):541--621, 1999.
		
		\bibitem{Ma00}
		X.~Ma.
		\newblock Submersions and equivariant {Q}uillen metrics.
		\newblock {\em Ann. Inst. Fourier (Grenoble)}, 50(5):1539--1588, 2000.
		
		\bibitem{Ma02}
		X.~Ma.
		\newblock Functoriality of real analytic torsion forms.
		\newblock {\em Israel J. Math.}, 131:1--50, 2002.
		
		\bibitem{MM07}
		X.~Ma and G.~Marinescu.
		\newblock {\em Holomorphic {M}orse inequalities and {B}ergman kernels}, volume
		254 of {\em Progress in Mathematics}.
		\newblock Birkh\"auser Verlag, Basel, 2007.
		
		\bibitem{MP97a}
		R.~B.~Melrose and P.~Piazza.
		\newblock Families of {D}irac operators, boundaries and the {$b$}-calculus.
		\newblock {\em J. Differential Geom.}, 46(1):99--180, 1997.
		
		\bibitem{MP97b}
		R.~B.~Melrose and P.~Piazza.
		\newblock An index theorem for families of {D}irac operators on odd-dimensional
		manifolds with boundary.
		\newblock {\em J. Differential Geom.}, 46(2):287--334, 1997.
		
	
			\bibitem{Mo61}
		G.~D.~Mostow.
		\newblock Cohomology of topological groups and 
		solvmanifolds.
		\newblock {\em Ann. of Math.}, 73:20--48, 1961.
		
		\bibitem{Or09}
		M.~L.~Ortiz.
		\newblock {\em Differential equivariant {K}-theory}.
		\newblock ProQuest LLC, Ann Arbor, MI, 2009.
		\newblock Thesis (Ph.D.)--The University of Texas at Austin.
		
		
		\bibitem{Qu85}
		D.~Quillen.
		\newblock Superconnections and the {C}hern character.
		\newblock {\em Topology}, 24(1):89--95, 1985.
		
		\bibitem{Se68}
		G.~Segal.
		\newblock Equivariant {$K$}-theory.
		\newblock {\em Inst. Hautes \'Etudes Sci. Publ. Math.}, (34):129--151, 1968.
		
		\bibitem{SS10}
		J.~Simons and D.~Sullivan.
		\newblock Structured vector bundles define differential {$K$}-theory.
		\newblock In {\em Quanta of maths}, volume~11 of {\em Clay Math. Proc.}, pages
		579--599. Amer. Math. Soc., Providence, RI, 2010.
		
		\bibitem{SV10}
		R.~J.~Szabo and A.~Valentino.
		\newblock Ramond-{R}amond fields, fractional branes and orbifold differential
		{$K$}-theory.
		\newblock {\em Comm. Math. Phys.}, 294(3):647--702, 2010.
		
		\bibitem{TWZ13}
		T.~Tradler, S.~O.~Wilson, and M.~Zeinalian.
		\newblock An elementary differential extension
		of odd {$K$}-theory.
		\newblock {\em J. K-Theory}, 12(2):331--361, 2013.

		
		\bibitem{Wi98}
		E.~Witten.
		\newblock D-branes and {$K$}-theory.
		\newblock {\em J. High Energy Phys.}, (12):Paper 19, 41 pp.\ (electronic),
		1998.
		
		\bibitem{Z01}	
		W.~Zhang.
		\newblock {\em Lectures on {C}hern-{W}eil theory and {W}itten deformations}, volume 4 of
		{\em Nankai Tracts in Mathematics}.
        \newblock World Scientific Publishing Co., Inc., River Edge, NJ, 2001.
		
		\bibitem{Z05}
		W.~Zhang. 
		\newblock An extended Cheeger-Müller theorem for covering spaces.
		\newblock {\em Topology}, 6:1093--1131, 2005.

		
	\end{thebibliography}
\end{document}